\documentclass{amsart}

\usepackage{amsmath, amsfonts, amsthm, amssymb, bm, stmaryrd, mathtools}
\usepackage{float, graphicx, epstopdf}
\usepackage{hyperref, cleveref}
\usepackage{color, xcolor, geometry}
\usepackage{tikz}

\usepackage{tikz}
\usetikzlibrary{calc}
\usepackage{subcaption} 

\newtheorem{theorem}{Theorem}[section]
\newtheorem{lemma}[theorem]{Lemma}
\newtheorem{corollary}[theorem]{Corollary}
\newtheorem{proposition}[theorem]{Proposition}

\theoremstyle{remark}
\newtheorem{remark}[theorem]{Remark}

\theoremstyle{definition}
\newtheorem{definition}[theorem]{Definition}
\newtheorem{example}[theorem]{Example}

\newtheorem{condition}{Condition}
\newtheorem{decomposition}{Decomposition}

\numberwithin{equation}{section}

\DeclareMathOperator{\Span}{span}
\DeclareMathOperator{\card}{Card}

\definecolor{pink}{RGB}{255,45,115}

\usepackage{enumitem}
\calclayout

\begin{document}

\title[A Construction of $C^{r}$ Finite Elements on the Alfeld Split]{A Construction of $C^{r}$ Conforming Finite Elements on the Alfeld Split in Any Dimension}


\author{Ting Lin}
\address{School of Mathematical Sciences, Peking University, Beijing 100871, P. R. China}
\curraddr{}
\email{lintingsms@pku.edu.cn}
\thanks{}

\author{Hendrik Speleers}
\address{Department of Mathematics, University of Rome Tor Vergata, Rome 00133, Italy}
\curraddr{}
\email{speleers@mat.uniroma2.it}
\thanks{}

\author{Qingyu Wu}
\address{School of Mathematical Sciences, Peking University, Beijing 100871, P. R. China}
\curraddr{}
\email{wu\_qingyu@pku.edu.cn}
\thanks{}



\date{}

\dedicatory{}

\begin{abstract}
Constructing $C^r$ conforming finite element spaces in any spatial dimension is a long-standing problem. For general triangulations, this problem was recently addressed by Hu-Lin-Wu (2024), under certain conditions on supersmoothness and polynomial degree. In this paper, we present the first unified construction of $C^r$ conforming finite elements on the $d$-dimensional Alfeld split for any $r \ge 1$ and $d \ge 2$, where the supersmoothness conditions and the polynomial degree requirement are relaxed. The construction combines a boundary space based on the framework of Hu-Lin-Wu (2024) with a new complementary space generated by a lifting operator. For the new family of finite element spaces, we describe the degrees of freedom, prove unisolvence, and derive explicit expressions for the dimension of the corresponding superspline spaces on a general triangulation endowed with the Alfeld split.
\end{abstract}

\maketitle

\section{Introduction}
\label{sec:introduction}

Constructing $C^r$ conforming finite element spaces for arbitrary smoothness $r \in \mathbb{N}_{0}$ in any spatial dimension $d \in \mathbb{N}$ is a long-standing problem. In \cite{hu2024construction}, the first construction was given for general triangulations. Specifically, for a given continuity vector $\bm{r} \coloneqq (r_{1}, r_{2}, \ldots, r_{d}) \in \mathbb{N}_{0}^{d}$ and polynomial degree $k \in \mathbb{N}_{0}$ such that 
\begin{equation}
\label{eq:assumption-tetra}
    2^{d - 1} r_{1} \le 2^{d - 2} r_2 \le \cdots \le r_{d}, \quad 2 r_{d} + 1 \le k,
\end{equation}
one can construct degrees of freedom such that the resulting finite element space has $C^{r_s}$ continuity on $(d-s)$-subsimplices. A natural question then arises: 

\begin{quote} 
\medskip
\emph{Can the conditions on $\bm{r}$ and $k$ be lowered, such that $C^r$ conforming finite element spaces can be constructed using fewer degrees of freedom?} 
\medskip
\end{quote}

Unfortunately, in \cite{hu2026sharpness} (see also \cite{zenisek1974} for the case $d=2$), it was shown that the conditions cannot be improved if a general triangulation is considered. However, several macro-element constructions in two and three dimensions indicate that the conditions are not tight for them; see, e.g., \cite{lai2007} and references therein.
Such a macro-element is obtained by decomposing an element into subelements according to a specific rule; this is often referred to as \emph{element split} or just \emph{split}. 

In two dimensions, one of the oldest and most popular macro-element constructions is the $C^1$ Clough--Tocher element \cite{ciarlet1974,clough1965}, with $12$ degrees of freedom. This is a piecewise cubic polynomial defined on a triangle subdivided into $3$ subtriangles, known as the Clough--Tocher split. For comparison's sake, we note that a minimum degree of $k=5$ is required to obtain a $C^1$ finite element on a triangle without split, which is achieved by the classical Argyris element \cite{argyris1968}, with $21$ degrees of freedom. Besides, on the Clough--Tocher split, several generalizations to higher degree and higher smoothness can be found in the literature; see, e.g., \cite{alfeld2002a,douglas1979,groselj2022,lai2001} and also \cite{lai2007}.
Alternatively, two prominent $C^1$ triangular macro-element constructions were proposed by Powell and Sabin \cite{powell1977}, with $9$ and $12$ degrees of freedom. These are both piecewise quadratic polynomials defined on a triangle subdivided into $6$ and $12$ subtriangles, respectively.
Also on these splits, various generalizations to higher degree and higher smoothness are available; see, e.g., \cite{alfeld2002b,groselj2016,lai2003,schumaker2006,speleers2013a} and also \cite{lai2007}.
A remarkable family of triangular splits was introduced by Wang and Shi \cite{wang1990}, enabling macro-element constructions of maximal smoothness, i.e., $C^{k-1}$ macro-elements of any degree $k$; see also \cite{lyche2022,lyche2024}. There are a few more (but unusual) triangular splits, such as in \cite{alfeld1984b,wang1992}.

In three dimensions, the most common tetrahedral split is the Alfeld split \cite{alfeld1984a}, leading to a $C^1$ quintic macro-element, with $68$ degrees of freedom \cite{awanou2002} or $65$ degrees of freedom \cite{lai2007}. Again, for comparison's sake, we note that a minimum degree of $k=9$ is required to obtain a $C^1$ finite element on a tetrahedron without split \cite{zenisek1973}, with $220$ degrees of freedom. On the Alfeld split, generalizations to higher degree and higher smoothness were proposed in \cite{alfeld2005,lai2013}.
There are also tetrahedral splits designed for the construction of $C^1$ cubic macro-elements \cite{alfeld2009,worsey1987} (with $44$ or $28$ degrees of freedom) and $C^1$ quadratic macro-elements \cite{schumaker2009,worsey1988} (with $44$ or $16$ degrees of freedom), but their constructions require a large number of subtetrahedra and/or restrictive geometric constraints.
Some of these constructions can also be applied in higher dimensions; see \cite{sorokina2008,speleers2013b,worsey1987}.

\begin{figure}[htbp]
    \centering
    \includegraphics[height = 5cm]{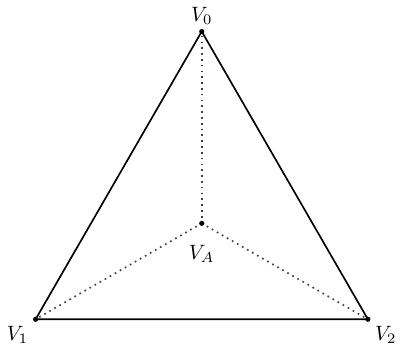}
    \qquad
    \includegraphics[height = 5cm]{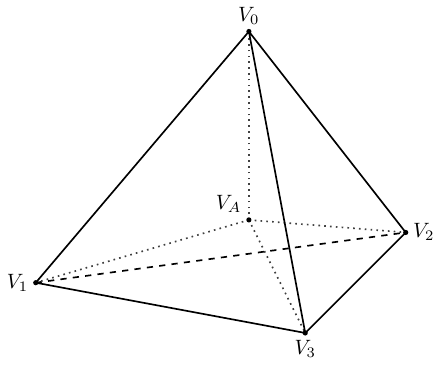}
    \caption{Alfeld split in two dimensions (Clough--Tocher split) and three dimensions.}
    \label{fig:barycentric-subdivision}
\end{figure}

In this paper, we are interested in the \emph{Alfeld split} in $d$ dimensions, which is constructed as follows. 
Let $K$ be a $d$-dimensional simplex, with vertices $V_0, V_1, \ldots, V_d$, and let $V_A$ be an interior point (usually the barycenter) of $K$. For each $j = 0, 1, \ldots, d$, we obtain the smaller simplex $K_j$ by adding $V_A$ to the vertices of $K$ and removing the vertex $V_j$.
Then, the Alfeld split of $K$, denoted by $\mathcal{T}_A(K)$, is defined as the set of the $(d+1)$ smaller simplices $K_j$. The point $V_A$ is called the (Alfeld) split point of $K$. In two dimensions, this split coincides with the Clough--Tocher split \cite{clough1965}. The Alfeld split in three dimensions was developed in \cite{alfeld1984a}. Higher-dimensional versions were considered in, e.g., \cite{kolesnikov2014,lyche2025,schenck2014}. An illustration of the Alfeld split in two and three dimensions is shown in Figure~\ref{fig:barycentric-subdivision}.

Let the index set $\mathbb{I}_{d} \coloneqq \{0, 1, \ldots, d\}$, and let $\mathcal{P}_{k}(K_j)$ be the space of $d$-variate polynomials up to total degree $k$ on the $d$-dimensional simplex $K_j$.
For the Alfeld split $\mathcal{T}_A(K)$ of $K$, we define the \emph{superspline space} with the continuity vector $\bm{r} \coloneqq (r_{1}, \ldots, r_{d}) \in \mathbb{N}_{0}^{d}$, the split point continuity $\rho \in \mathbb{N}_{0}$, and the polynomial degree $k \in \mathbb{N}_{0}$, as follows:
    \begin{equation}
        \label{eq:shape-function}
        \begin{aligned}
            \mathcal{S}_{k}^{\bm{r}, \rho}(\mathcal{T}_{A}(K)) \coloneqq \big\{u \in C^{r_{1}}(K): \; & u|_{K_{j}} \in \mathcal{P}_{k}(K_{j}), \text{ for each } j \in \mathbb{I}_{d}; \\
            & \nabla^{n} u \text{ is single-valued on } F, \\
            & \text{for each } t \text{-dimensional } (0 \le t \le d - 2) \\
            & \text{subsimplex } F \text{ of } K \text{ and } 0 \le n \le r_{d - t}; \\
            & \nabla^{n} u \text{ is single-valued at } V_{A} \text{ for } 0 \le n \le \rho\big\}.
        \end{aligned}
    \end{equation}
For later use, we always consider $\rho \le k - 1$, and define the boundary layer number $b \in \mathbb{N}_{0}$ as
    \begin{equation*}
        b \coloneqq k - \rho.
    \end{equation*}

This paper discusses the construction of macro-elements on Alfeld splits in general dimension $d \ge 2$. To this end, the following assumption on the continuity vector $\bm{r} \in \mathbb{N}_{0}^{d}$, the split point continuity $\rho \in \mathbb{N}_{0}$, the polynomial degree $k \in \mathbb{N}_{0}$, and the boundary layer number $b \in \mathbb{N}_{0}$ is given. 
\begin{condition}
    \label{cond:Cr-macro-element}
    For the continuity vector $\bm{r} \coloneqq (r_{1}, \ldots, r_{d}) \in \mathbb{N}_{0}^{d}$, the split point continuity $\rho \in \mathbb{N}_{0}$, the polynomial degree $k \in \mathbb{N}_{0}$, and the boundary layer number $b \in \mathbb{N}_{0}$, we say that $\bm{r}$, $\rho$, $k$, and $b$ jointly satisfy Condition~\ref{cond:Cr-macro-element} if 
    \begin{equation*}
        1 \le \bigg\lceil\frac{3 r_{1} - 1}{2}\bigg\rceil \le r_{2} \le 2 r_{1} - 1,
    \end{equation*}
    \begin{equation*}
        2^{d - 2} r_{2} \le 2^{d - 3} r_{3} \le \cdots \le r_{d}, \quad 2 r_{d} + 1 \le k,
    \end{equation*}
    and
    \begin{equation*}
        \rho = k - b, \quad 2 r_{1} - r_{2} \le b \le r_{2} - r_{1} + 1.
    \end{equation*}
\end{condition}

Note that $k$ and $b$ completely determine $\rho$. 
Since $(r_{2} - r_{1} + 1) - (2 r_{1} - r_{2}) = 2 r_{2} - 3 r_{1} + 1 \ge 0$, the restriction on $b$ is well posed. We only consider the case where $r_{1} \ge 1$ and $r_2 \le 2r_1 - 1$ because otherwise (if $r_{1} = 0$ or $r_2 \ge 2r_1$) the restrictions on $\bm{r}$ in Condition~\ref{cond:Cr-macro-element} would be the same as \eqref{eq:assumption-tetra}.

The main result of this paper can be summarized as follows.
\begin{theorem}
    \label{thm:Cr-macro-element-intro}
    Suppose that $\bm{r} \coloneqq (r_{1}, \ldots, r_{d})$, $\rho$, $k$, and $b$ jointly satisfy Condition~\ref{cond:Cr-macro-element}. Then, a $C^{r_1}$ conforming finite element space can be constructed with the local shape function space $\mathcal S_k^{\bm{r},\rho}(\mathcal{T}_A(K))$ defined in \eqref{eq:shape-function}.
\end{theorem}
Further details of this finite element construction are provided in Theorem~\ref{thm:Cr-macro-element}. For the corresponding set of degrees of freedom, we refer to Definition~\ref{def:local-dofs}.

Given a triangulation $\mathcal{T}(\Omega)$ of a $d$-dimensional domain $\Omega$, let $\mathcal{T}_A(\Omega)$ be the refined triangulation obtained by applying the Alfeld split $\mathcal{T}_A(K)$ to each $d$-dimensional simplex $K$ of $\mathcal{T}(\Omega)$. Under the assumption that $\bm{r} \coloneqq (r_{1}, \ldots, r_{d})$, $\rho$, $k$, and $b$ jointly satisfy Condition~\ref{cond:Cr-macro-element}, the presented finite element construction gives rise to the following superspline space
    \begin{equation}
        \label{eq:superspline}
        \begin{aligned}
            \mathcal{S}_{k}^{\bm{r}, \rho}(\mathcal{T}_{A}(\Omega)) \coloneqq \big\{u \in C^{r_{1}}(\Omega): \; & u|_{K} \in \mathcal{S}_{k}^{\bm{r}, \rho}(\mathcal{T}_{A}(K)) \text{ for each } d \text{-simplex } K \text{ of } \mathcal{T}(\Omega); \\
            & \nabla^{n} u \text{ is single-valued on } F, \\
            & \text{for each } t \text{-dimensional } (0 \le t \le d - 2) \\
            & \text{subsimplex } F \text{ of } \mathcal{T}(\Omega) \text{ and } 0 \le n \le r_{d - t}\big\}.
        \end{aligned}
    \end{equation}
Theorem~\ref{thm:dim-superspline} specifies its dimension.

\begin{table}[t]
    \centering
    \begin{tabular}{c|c|c|c|c}
        \hline
        $r_{1}$ & $r_{s}$ ($2 \le s \le d$) & $\rho$ & $k$ & $b$ \\
        \hline
        $2 m - 1$ & $2^{s - 2} \cdot (3 m - 2)$ & $(2^{d - 1} \cdot 3 - 1) m - 2^{d} + 1$ & $2^{d - 1} \cdot (3 m - 2) + 1$ & $m$ \\
        \hline
        $2 m$ & $2^{s - 2} \cdot 3 m$ & $(2^{d - 1} \cdot 3 - 1) m + 1$ & $2^{d - 1} \cdot 3 m + 1$ & $m$ \\
        \hline
    \end{tabular}    
    \caption{The local shape function space $\mathcal S_k^{\bm{r},\rho}(\mathcal{T}_A(K))$ in Theorem~\ref{thm:Cr-macro-element-intro} requires that the continuity vector $\bm{r} \coloneqq (r_{1}, \ldots, r_{d}) \in \mathbb{N}_{0}^{d}$, the split point continuity $\rho \in \mathbb{N}_{0}$, the polynomial degree $k \in \mathbb{N}_{0}$, and the boundary layer number $b \in \mathbb{N}_{0}$ jointly satisfy Condition~\ref{cond:Cr-macro-element}. The table summarizes the corresponding values after sequentially minimizing $k$ and $b$, for any $m \ge 1$ and $d \ge 2$.}
    \label{tab:minimum-degree}
\end{table}

Given $r_{1} = r \ge 1$, the choices of $\bm{r}$, $\rho$, $k$, and $b$ for the local shape function space $\mathcal{S}_{k}^{\bm{r}, \rho}(\mathcal{T}_{A}(K))$, with $k$ and $b$ minimized sequentially, are summarized in Table~\ref{tab:minimum-degree}. In two dimensions, the minimum degree (with respect to $r$) assumed in this paper is $k = 6 m - 3$ if $r = 2m - 1$ and $k = 6 m + 1$ if $r = 2 m$. This aligns well with the degree of the bivariate $C^r$ macro-elements proposed in \cite{lai2001}. In the same work, it was shown that those degrees are optimal for the Clough--Tocher split in two dimensions. In three dimensions, the minimum degree assumed in this paper is $k = 12 m - 7$ if $r = 2 m - 1$ and $k = 12 m + 1$ if $r = 2 m$. This is in agreement with the degree of the trivariate $C^r$ macro-elements proposed in \cite{lai2013}. To the best of our knowledge, there are no similar results available in the literature for higher spatial dimensions ($d > 3$) on a general triangulation.

We remark that there exist some higher-dimensional results for smooth spline spaces defined on a single simplex endowed with the Alfeld split; see, e.g., \cite{kolesnikov2014,schenck2014} for dimension formulas and \cite{lyche2025} for a basis construction. These results, however, do not directly extend to obtain smooth spline spaces on a general triangulation, and thus cannot be used in a finite element context without additional imposition of smoothness conditions.

The presented finite element construction is inspired by ideas from \cite{hu2024construction}. An important ingredient of the construction is the decomposition of the space $\mathcal{S}_{k}^{\bm{r}, \rho}(\mathcal{T}_{A}(K))$ into a boundary part $\mathcal{P}^{\partial}$ and an interior part $\mathcal{Q}^{\circ}$.
For the boundary space $\mathcal{P}^\partial$, we largely follow the framework of \cite{hu2024construction}. The principal new contribution is instead the construction of the complementary space $\mathcal{Q}^{\circ}$ through a specific lifting operator and the proof that it combines with $\mathcal{P}^\partial$ to give unisolvence. We also provide an explicit basis and dimension formulas on the Alfeld split.

The finite element construction proposed in \cite{hu2024construction} must satisfy the conditions in \eqref{eq:assumption-tetra}, so the corresponding minimum degree is $k=2^d r+1$ to achieve $C^r$ continuity in $d$ dimensions. This implies that the new construction allows us to lower the minimum degree by $2^{d - 1} m$ for $r = 2 m - 1$ or $r = 2 m$ in $d$ dimensions.

The remainder of the paper is organized as follows. In Section~\ref{sec:notation}, we provide some introductory notations related to barycentric coordinates and monomials with respect to a simplex and its Alfeld split. We also define different sets of multi-indices. Our main result is summarized in Section~\ref{sec:construction}. After defining a specific lifting operator, we describe the degrees of freedom of our finite element space under investigation. We also give explicit expressions for the dimension of the corresponding superspline space. In Section~\ref{sec:decomposition}, we propose a pair of refined intrinsic decompositions for the Alfeld split, which allow us to rewrite the degrees of freedom in an equivalent form. Then, several properties of the lifting operator are provided in Section~\ref{sec:properties}. The unisolvence and the continuity of the finite element space are addressed in Sections~\ref{sec:unisolvence} and~\ref{sec:continuity}, respectively. In Section~\ref{sec:dimension}, we study the dimension of the superspline space in detail. We end with some concluding remarks in Section~\ref{sec:conclusion}.

\section{Barycentric Coordinates and Multi-Index Sets} 
\label{sec:notation}

In this section, we introduce barycentric coordinates and related monomials, both with respect to a simplex $K$ and its refinement $\mathcal{T}_A(K)$. We also define different sets of multi-indices.

We start with some notations in a simplex in $d$ dimensions. For the $d$-dimensional simplex $K$, let $\lambda_{0}, \lambda_{1}, \ldots, \lambda_{d}$ be the barycentric coordinates with respect to $K$, corresponding to its vertices $V_{0}, V_{1}, \ldots, V_{d}$, respectively. Let $\mathcal{T}^{\partial}(K)$ be the collection of all the $t$-dimensional ($0 \le t \le d - 1$) subsimplices $F$ of $K$.
    
Let the index set $\mathbb{I}_{d} \coloneqq \{0, 1, \ldots, d\}$, and define the normalized monomial as
\begin{equation}
    \label{eq:normalized-monomial}
    \llbracket\bm{\lambda}\rrbracket^{\bm{\alpha}} \coloneqq \prod_{i \in \mathbb{I}_{d}} \frac{\lambda_{i}^{\alpha_{i}}}{\alpha_{i}!} = \frac{\lambda_{0}^{\alpha_{0}}}{\alpha_{0}!} \cdot \frac{\lambda_{1}^{\alpha_{1}}}{\alpha_{1}!} \cdot \cdots \cdot \frac{\lambda_{d}^{\alpha_{d}}}{\alpha_{d}!},
\end{equation}
where $\bm{\alpha} \coloneqq (\alpha_{0}, \alpha_{1}, \ldots, \alpha_{d}) \in \mathbb{N}_{0}^{d + 1}$ is a multi-index. Moreover, define the sum of $\bm{\alpha}$ as
\begin{equation*}
    |\bm{\alpha}| \coloneqq \sum_{i \in \mathbb{I}_{d}} \alpha_{i} = \alpha_{0} + \alpha_{1} + \cdots + \alpha_{d},
\end{equation*}
and the multi-index set
\begin{equation*}
    \Sigma(\mathbb{I}_{d}, k) \coloneqq \bigg\{\bm{\alpha} \coloneqq (\alpha_{0}, \alpha_{1}, \ldots, \alpha_{d}) \in \mathbb{N}_{0}^{d + 1}: |\bm{\alpha}| \coloneqq \sum_{i \in \mathbb{I}_{d}} \alpha_{i} = k\bigg\}.
\end{equation*}

For $0 \le t \le d - 1$, and a $t$-dimensional subsimplex $F$ of $K$ with vertices $V_{i_{0}}, V_{i_{1}}, \ldots, V_{i_{t}}$, let $\mathbb{I}_{F} \coloneqq \{i_{0}, i_{1}, \ldots, i_{t}\}$ be the vertex index set, and denote $\mathbb{I}_{\setminus F} \coloneqq \mathbb{I}_{d} \setminus \mathbb{I}_{F} = \{i_{t + 1}, \ldots, i_{d}\}$. In particular, when $t=0$ and $F$ is a vertex $V$, the notations $\mathbb I_{V}$ and $\mathbb I_{\setminus V}$ are also used. Moreover, denote $\mathbb{I}_{d, \setminus t} \coloneqq \mathbb{I}_{d} \setminus \mathbb{I}_{t} = \{t + 1, \ldots, d\}$. We additionally define the partial sums
\begin{equation*}
    |\bm{\alpha}|_{F} \coloneqq \sum_{i \in \mathbb{I}_{F}} \alpha_{i} = \alpha_{i_{0}} + \alpha_{i_{1}} + \cdots + \alpha_{i_{t}},
\end{equation*}
and
\begin{equation*}
    |\bm{\alpha}|_{\setminus F} \coloneqq |\bm{\alpha}| - |\bm{\alpha}|_{F} = \sum_{i \in \mathbb{I}_{\setminus F}} \alpha_{i} = \alpha_{i_{t + 1}} + \cdots + \alpha_{i_{d}}.
\end{equation*}

\begin{remark}
    The normalized monomial $\llbracket\bm{\lambda}\rrbracket^{\bm{\alpha}}$ in \eqref{eq:normalized-monomial}, multiplied with the factor $|\bm{\alpha}|!$, is commonly referred to as a Bernstein basis polynomial of degree $|\bm{\alpha}|$.
\end{remark}

Now, we introduce the barycentric coordinates and normalized monomial with respect to the finer simplices in the Alfeld split. Recall that $\mathcal{T}_{A}(K)$ of $K$ is defined as
\begin{equation*}
    \mathcal{T}_{A}(K) \coloneqq \{K_{j} : j \in \mathbb{I}_{d}\} = \{K_{0}, K_{1}, \ldots, K_{d}\},
\end{equation*}
where $K_{j}$ is the $d$-dimensional simplex with vertices $V_{A}$ and $V_{i}$ for $i \in \mathbb{I}_{\setminus j} \coloneqq \mathbb{I}_{d} \setminus \{j\}$. Let $\lambda_{j, 0}, \lambda_{j, 1}, \ldots, \lambda_{j, d}$ be the barycentric coordinates with respect to $K_{j}$, corresponding to its vertices $V_{j, 0}, V_{j, 1}, \ldots, V_{j, d}$, respectively, where $V_{j, j} = V_{A}$ and $V_{j, i} = V_{i}$ for $i \in \mathbb{I}_{\setminus j}$. This notation is illustrated in Figure~\ref{fig:2d-triangle-barycentric} for $d = 2$.

\begin{figure}[htbp]
    \centering
    \includegraphics[height = 5cm]{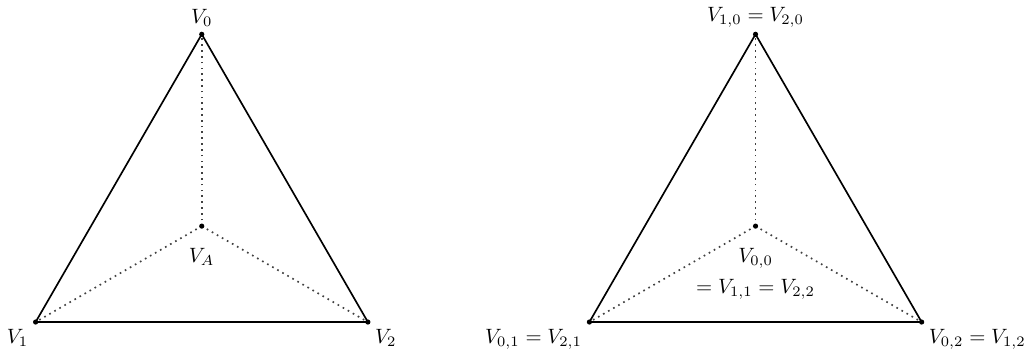}
    \caption{Illustration of the notation for the vertices in the Alfeld split in two dimensions.}
    \label{fig:2d-triangle-barycentric}
\end{figure}

We have the following relationship between different sets of barycentric coordinates.
\begin{lemma}
    \label{lem:transformation}
    Suppose that $V_{A}$ has the barycentric coordinates $(\mu_{0}, \mu_{1}, \dots, \mu_{d})$ with respect to $K$. Then, the barycentric coordinates with respect to $K_{j}$ can be determined as
    \begin{equation*}
        \lambda_{j, j} = \frac{1}{\mu_{j}} \lambda_{j}, \quad \lambda_{j, i} = \lambda_{i} - \frac{\mu_{i}}{\mu_{j}} \lambda_{j} \text{ for } i \in \mathbb{I}_{\setminus j}.
    \end{equation*}
\end{lemma}
\begin{proof}
    By the definition of barycentric coordinates, for any point $X$, it holds that
    \begin{equation*}
        X = \sum_{i \in \mathbb{I}_{d}} (\lambda_{j, i}|_{X}) V_{j, i},
    \end{equation*}
    and
    \begin{equation*}
        X = \sum_{i \in \mathbb{I}_{d}} (\lambda_{i}|_{X}) V_{i} = (\lambda_{j}|_{X}) V_{j} + \sum_{i \in \mathbb{I}_{\setminus j}} (\lambda_{i}|_{X}) V_{j, i}.
    \end{equation*}
    Since $V_{A} = \sum_{i \in \mathbb{I}_{d}} \mu_{i} V_{i}$, it follows that
    \begin{equation*}
        V_{j} = \frac{1}{\mu_{j}} V_{A} - \sum_{i \in \mathbb{I}_{\setminus j}} \frac{\mu_{i}}{\mu_{j}} V_{i} = \frac{1}{\mu_{j}} V_{j, j} - \sum_{i \in \mathbb{I}_{\setminus j}} \frac{\mu_{i}}{\mu_{j}} V_{j, i}.
    \end{equation*}
    Consequently,
    \begin{equation*}
        \sum_{i \in \mathbb{I}_{d}} \lambda_{j, i} V_{j, i} = \lambda_{j} V_{j} + \sum_{i \in \mathbb{I}_{\setminus j}} \lambda_{i} V_{j, i} = \frac{1}{\mu_{j}} \lambda_{j} V_{j, j} + \sum_{i \in \mathbb{I}_{\setminus j}} \bigg(\lambda_{i} - \frac{\mu_{i}}{\mu_{j}} \lambda_{j}\bigg) V_{j, i},
    \end{equation*}
    from which we immediately obtain this lemma.
\end{proof}

Given a multi-index $\bm{\alpha} \coloneqq (\alpha_{0}, \alpha_{1}, \ldots, \alpha_{d}) \in \mathbb{N}_{0}^{d + 1}$, we define the normalized monomial with respect to $K_{j}$ as
\begin{equation*}
    \llbracket\bm{\lambda}\rrbracket_{j}^{\bm{\alpha}} \coloneqq \prod_{i \in \mathbb{I}_{d}} \frac{\lambda_{j, i}^{\alpha_{i}}}{\alpha_{i}!} = \frac{\lambda_{j, 0}^{\alpha_{0}}}{\alpha_{0}!} \cdot \frac{\lambda_{j, 1}^{\alpha_{1}}}{\alpha_{1}!} \cdot \cdots \cdot \frac{\lambda_{j, d}^{\alpha_{d}}}{\alpha_{d}!}.
\end{equation*}

For a $t$-dimensional ($0 \le t \le d - 1$) subsimplex $F$ of $K_{j}$ with vertices $V_{j, i_{0}}, V_{j, i_{1}}, \ldots, V_{j, i_{t}}$, let $\mathbb{I}_{j, F} \coloneqq \{i_{0}, i_{1}, \ldots, i_{t}\}$ be the vertex index set, and denote $\mathbb{I}_{j, \setminus F} \coloneqq \mathbb{I}_{d} \setminus \mathbb{I}_{j, F}$. If $F \in \mathcal{T}^{\partial}(K)$, i.e., $V_{A}$ is not a vertex of $F$, it holds that $\mathbb{I}_{j, F} = \mathbb{I}_{F}$ and $\mathbb{I}_{j, \setminus F} = \mathbb{I}_{\setminus F}$. Moreover, define the partial sums
\begin{equation*}
    |\bm{\alpha}|_{j, F} \coloneqq \sum_{i \in \mathbb{I}_{j, F}} \alpha_{i}, \quad |\bm{\alpha}|_{j, \setminus F} \coloneqq |\bm{\alpha}| - |\bm{\alpha}|_{j, F} = \sum_{i \in \mathbb{I}_{j, \setminus F}} \alpha_{i}.
\end{equation*}

These concepts can be generalized to any subsimplex $F$ of $K$. Specifically, for the $t$-dimensional $(1 \le t \le d - 1)$ subsimplex $F$ of $K$ with vertices $V_{i_{0}}, V_{i_{1}}, \ldots, V_{i_{t}}$, let $\lambda_{F, i_{0}}, \lambda_{F, i_{1}}, \ldots, \lambda_{F, i_{t}}$ be the barycentric coordinates with respect to $F$, corresponding to its vertices $V_{i_{0}}, V_{i_{1}}, \ldots, V_{i_{t}}$, respectively. Let $\mathbb{I}_{F} \coloneqq \{i_{0}, i_{1}, \ldots, i_{t}\}$ be the vertex index set. Specifically, for $j = 0, 1, \ldots, d$, the $(d - 1)$-dimensional subsimplex with vertices $V_{i}$ for $i \in \mathbb{I}_{\setminus j}$ is denoted as $F_{\setminus j}$. 

If $F$ is a subsimplex of both $K$ and $K_{j}$, then for any $i \in \mathbb{I}_{F}$, it holds that
\begin{equation*}
    \lambda_{i}|_{F} = \lambda_{j, i}|_{F} = \lambda_{F, i}.
\end{equation*}
Therefore, on the subsimplex $F \in \mathcal{T}^{\partial}(K)$, we use the simplified notation $\lambda_{i}$ in place of $\lambda_{j, i}$ and $\lambda_{F, i}$, since these coordinates coincide upon restriction to $F$.

Define the normalized monomial as
\begin{equation*}
    \llbracket\bm{\lambda}\rrbracket_{F}^{\bm{\sigma}} \coloneqq \prod_{i \in \mathbb{I}_{F}} \frac{\lambda_{i}^{\sigma_{i}}}{\sigma_{i}!} = \frac{\lambda_{i_{0}}^{\sigma_{i_{0}}}}{\sigma_{i_{0}}!} \cdot \frac{\lambda_{i_{1}}^{\sigma_{i_{1}}}}{\sigma_{i_{1}}!} \cdot \cdots \cdot \frac{\lambda_{i_{t}}^{\sigma_{i_{t}}}}{\sigma_{i_{t}}!},
\end{equation*}
where $\bm{\sigma} \coloneqq (\sigma_{i_{0}}, \sigma_{i_{1}}, \ldots, \sigma_{i_{t}}) \in \mathbb{N}_{0}^{t + 1}$ is a multi-index. Moreover, define the sum of $\bm{\sigma}$ as
\begin{equation*}
    |\bm{\sigma}|_{F} \coloneqq \sum_{i \in \mathbb{I}_{F}} \sigma_{i} = \sigma_{i_{0}} + \sigma_{i_{1}} + \cdots + \sigma_{i_{t}},
\end{equation*}
and the multi-index set
\begin{equation*}
    \Sigma(\mathbb{I}_{F}, l) \coloneqq \bigg\{\bm{\sigma} \coloneqq (\sigma_{i_{0}}, \sigma_{i_{1}}, \ldots, \sigma_{i_{t}}) \in \mathbb{N}_{0}^{t + 1}: |\bm{\sigma}|_{F} \coloneqq \sum_{i \in \mathbb{I}_{F}} \sigma_{i} = l\bigg\}.
\end{equation*}

For a subsimplex $E$ of $F$, let $\mathbb{I}_{E}$ be the vertex index set, and denote $\mathbb{I}_{F, \setminus E} \coloneqq \mathbb{I}_{F} \setminus \mathbb{I}_{E}$. We additionally define the partial sums
\begin{equation*}
    |\bm{\sigma}|_{F, E} \coloneqq \sum_{i \in \mathbb{I}_{E}} \sigma_{i}, \quad |\bm{\sigma}|_{F, \setminus E} \coloneqq |\bm{\sigma}|_{F} - |\bm{\sigma}|_{F, E} = \sum_{i \in \mathbb{I}_{F, \setminus E}} \sigma_{i}.
\end{equation*}

\section{Main Construction}
\label{sec:construction}

In this section, we specify the degrees of freedom for the local shape function space $\mathcal S_{k}^{\bm{r},\rho}(\mathcal{T}_A(K))$ in \eqref{eq:shape-function}. To this end, we start by defining a specific lifting operator. Afterwards, we also give explicit expressions for the dimension of the corresponding superspline space.

\begin{definition}[$Q$-Operator]
    \label{def:Q-op}
    Given the polynomial degree $k$ and the split point continuity $\rho$ such that $\rho \le k - 1$, set the boundary layer number $b \coloneqq k - \rho$. The linear lifting operator $Q_{k}^{\rho}$ maps a polynomial in $\mathcal{P}_{\rho}(K)$ to the piecewise polynomial space 
    $$\big\{u \in L^{2}(K): u|_{K_{j}} \in \mathcal{P}_{k}(K_{j}) \text{ for each } j \in \mathbb{I}_{d}\big\}.$$ 
    For $p \in \mathcal P_{\rho}(K)$, and $j \in \mathbb I_d$, the detailed values $Q_{k}^{\rho}(p)$ on $K_j$ are determined as follows: $p$ has the following expansion (with respect to the barycentric coordinates of $K_j$) as 
    \begin{equation*}
        p = \sum_{|\bm \beta| = \rho} c_{\bm \beta} \llbracket\bm{\lambda}\rrbracket_{j}^{\bm{\beta}},
    \end{equation*}
    where the summation ranges over multi-indices $\bm{\beta} \coloneqq (\beta_{0}, \beta_{1}, \ldots, \beta_{d}) \in \mathbb{N}_{0}^{d + 1}$ such that $|\bm{\beta}| \coloneqq \sum_{i \in \mathbb{I}_{d}} \beta_{i} = \rho$. Then, we define the values $Q_{k}^{\rho}(p)$ in $K_j$ as 
    \begin{equation*}
        Q_{k}^{\rho}(p)\big|_{K_j} \coloneqq \sum_{|\bm \beta| = \rho} c_{\bm \beta}\llbracket\bm{\lambda}\rrbracket_{j}^{\bm{\beta} + b \bm{e}_{j}},
    \end{equation*}
    where the multi-index $\bm \alpha \coloneqq \bm{\beta} + b \bm{e}_{j} \in \Sigma(\mathbb{I}_{d}, k)$ is defined as $\alpha_{i} = \beta_{i} + b \cdot \delta_{j, i}$ for $i \in \mathbb{I}_{d}$, with Kronecker’s delta $\delta$. 
\end{definition}

In particular, for each normalized monomial with respect to $K_{j}$, it holds that 
\begin{equation}
    \label{eq:Q-k-b-operator}
    Q_{k}^{\rho}(\llbracket\bm{\lambda}\rrbracket_{j}^{\bm{\beta}})\big|_{K_{j}} = Q_{k}^{\rho} \Bigg(\frac{\lambda_{j, j}^{\beta_{j}}}{\beta_{j}!} \prod_{i \in \mathbb{I}_{\setminus j}} \frac{\lambda_{j, i}^{\beta_{i}}}{\beta_{i}!}\Bigg)\Bigg|_{K_{j}} = \frac{\lambda_{j, j}^{\beta_{j} + b}}{(\beta_{j} + b)!} \prod_{i \in \mathbb{I}_{\setminus j}} \frac{\lambda_{j, i}^{\beta_{i}}}{\beta_{i}!} = \llbracket\bm{\lambda}\rrbracket_{j}^{\bm{\beta} + b \bm e_j}.
\end{equation}
As we will see, $Q_{k}^{\rho}$ plays a crucial role in the construction of the degrees of freedom. Its properties will be studied in Section~\ref{sec:properties}. In particular, the range of $Q_{k}^{\rho}$ can be characterized; see Proposition~\ref{prop:Q-k-b-bijection}.
Note that a similar operator to $Q_{k}^{\rho}$ for the special case $d = 2$ is discussed in \cite{speleers2013a}. 

\begin{definition}[Local degrees of freedom]
    \label{def:local-dofs}
    Let $u \in \mathcal S_{k}^{\bm{r}, \rho}(\mathcal{T}_A(K))$. The set of (local) degrees of freedom for $u$ is divided into three subsets. 
    \begin{enumerate}
        \item For each vertex $V$ of $K$, the degrees of freedom at $V$ for $u$ are defined as follows: for each $0 \le n \le r_{d}$, and each $\bm{\theta} \in \Sigma(\mathbb{I}_{d, \setminus 0}, n)$,
        \begin{equation}
            \label{eq:set-dof-1}
            \bm{D}_{V}^{\bm{\theta}} u \big|_{V}, \text{ where } \bm{D}_{V}^{\bm{\theta}} u \coloneqq \frac{\partial^{n}}{\prod_{i = 1}^{d} \partial \bm{n}_{V, i}^{\theta_{i}}} u.
        \end{equation}
        Here $\bm{n}_{V, 1}, \ldots, \bm{n}_{V, d}$ are $d$ pairwise orthogonal unit normal vectors of $V$.
    
        \item For each $t$-dimensional $(1 \le t \le d - 1)$ subsimplex $F$ of $K$, let $\mathbb{I}_{F} \coloneqq \{i_{0}, i_{1}, \ldots, i_{t}\}$, and the degrees of freedom on $F$ for $u$ are defined as follows: for each $0 \le n \le r_{d - t}$, each $\bm{\theta} \in \Sigma(\mathbb{I}_{d, \setminus t}, n)$, and each $\bm{\sigma} \in \Sigma_{0}^{\bm{q}_{t, n}}(\mathbb{I}_{F}, k - n)$,
        \begin{equation}
            \label{eq:set-dof-2}
            \frac{1}{|F|} \int_{F} (\bm{D}_{F}^{\bm{\theta}} u) \cdot \llbracket\bm{\lambda}\rrbracket_{F}^{\bm{\sigma}},
        \end{equation}
        where the continuity vector $\bm{q}_{t, n} \coloneqq (r_{d - t + 1} - n, \ldots, r_{d} - n)$, the multi-index set $\Sigma^{\bm{q}}_{0}(\mathbb{I}_{F}, l)$ with $\bm{q} \coloneqq (q_{1}, \ldots, q_{t})$ is defined as
        \begin{equation*}
            \begin{aligned}
                \Sigma^{\bm{q}}_{0}(\mathbb{I}_{F}, l) \coloneqq \big\{\bm{\sigma} \coloneqq (\sigma_{i_{0}}, \sigma_{i_{1}}, \ldots, \sigma_{i_{t}}) \in \mathbb{N}_{0}^{t + 1} : \; & |\bm{\sigma}|_{F} = l, \text{ and } |\bm{\sigma}|_{F, \setminus E} \ge q_{t - t'} + 1, \\
                & \text{for each } t' \text{-dimensional } (0 \le t' \le t - 1) \\
                & \text{subsimplex } E \text{ of } F\big\},
            \end{aligned}
        \end{equation*}
        and the $n$-th order normal derivative $\bm{D}_{F}^{\bm{\theta}}$ is defined as
        \begin{equation*}
            \bm{D}_{F}^{\bm{\theta}} u \coloneqq \frac{\partial^{n}}{\prod_{i = t + 1}^{d} \partial \bm{n}_{F, i}^{\theta_{i}}} u.
        \end{equation*}
        Here $\bm{n}_{F, t + 1}, \ldots, \bm{n}_{F, d}$ are $(d - t)$ pairwise orthogonal unit normal vectors of $F$.
        When $t = d - 1$, the shortened notation $\bm{n}_{F}$ is also used.
        
        \item For the $d$-dimensional simplex $K$, the degrees of freedom on $K$ for $u$ are defined as follows: for each $\bm{\beta} \in \Sigma_{0}^{\bm{r}^{\circ}}(\mathbb{I}_{d}, \rho)$,
        \begin{equation}
        \label{eq:set-dof-3}
            \frac{1}{|K|} \int_{K} u \cdot Q_{k}^{\rho}(\llbracket\bm{\lambda}\rrbracket^{\bm{\beta}}),
        \end{equation}
        where the continuity vector $\bm{r}^{\circ} \coloneqq (r^{\circ}_{1}, \ldots, r^{\circ}_{d}) = (r_{1} - b, \ldots, r_{d} - b)$, and the multi-index set $\Sigma_{0}^{\bm{r}^{\circ}}(\mathbb{I}_{d}, \rho)$ is defined as
        \begin{equation}
            \label{eq:Sigma-0-d}
            \begin{aligned}
                \Sigma^{\bm{r}^{\circ}}_{0}(\mathbb{I}_{d}, \rho) \coloneqq \big\{\bm{\beta} \coloneqq (\beta_{0}, \beta_{1}, \ldots, \beta_{d}) \in \mathbb{N}_{0}^{d + 1} : \; & |\bm{\beta}| = \rho, \text{ and } |\bm{\beta}|_{\setminus F} \ge r^{\circ}_{d - t} + 1, \\
                & \text{for each } t \text{-dimensional } (0 \le t \le d - 1) \\
                & \text{subsimplex } F \text{ of } K\big\}.
            \end{aligned}
        \end{equation}
    \end{enumerate}
\end{definition}

This brings us to the main theorem of this paper.

\begin{theorem}
    \label{thm:Cr-macro-element}
    Suppose that the continuity vector $\bm{r} \coloneqq (r_{1}, \ldots, r_{d}) \in \mathbb{N}_{0}^{d}$, the split point continuity $\rho \in \mathbb{N}_{0}$, the polynomial degree $k \in \mathbb{N}_{0}$, and the boundary layer number $b \in \mathbb{N}_{0}$ jointly satisfy Condition~\ref{cond:Cr-macro-element}. 
    Then, the $d$-dimensional simplex $K$, the shape function space $\mathcal{S}_{k}^{\bm{r}, \rho}(\mathcal{T}_{A}(K))$ defined in \eqref{eq:shape-function}, and the set of degrees of freedom defined in \eqref{eq:set-dof-1}--\eqref{eq:set-dof-3} form a finite element. The resulting finite element space over the triangulation $\mathcal{T}(\Omega)$ of the $d$-dimensional domain $\Omega$ is the superspline space $\mathcal{S}_{k}^{\bm{r}, \rho}(\mathcal{T}_{A}(\Omega))$ defined in \eqref{eq:superspline}, which is of $C^{r_{1}}$ continuity. 
\end{theorem}

The proof strategy is as follows. First, we will introduce a pair of \emph{refined intrinsic decompositions}, which build a relationship between multi-index sets of type $\Sigma_{0}$ and $\Sigma_{F, n}$. Based on these refined intrinsic decompositions, we can rewrite the degrees of freedom in an equivalent form. This is done in Section~\ref{sec:decomposition}. Next, we will investigate several properties of $Q_{k}^{\rho}$ in Section~\ref{sec:properties}. The unisolvence and the continuity will be shown in Sections~\ref{sec:unisolvence} and~\ref{sec:continuity}, respectively. 

\begin{remark}
    Theorem~\ref{thm:Cr-macro-element} also holds if $\bm{n}_{V, 1}, \ldots, \bm{n}_{V, d}$ and $\bm{n}_{F, t + 1}, \ldots, \bm{n}_{F, d}$ in Definition~\ref{def:local-dofs} are linearly independent normal vectors of $V$ and $F$, respectively. This follows from the fact that the corresponding sets of degrees of freedom are related by an invertible linear transformation.
\end{remark}

Finally, we discuss the dimension of the superspline space $\mathcal{S}_{k}^{\bm{r}, \rho}(\mathcal{T}_{A}(\Omega))$ defined in \eqref{eq:superspline} and the dimension of the shape function space $\mathcal{S}_{k}^{\bm{r}, \rho}(\mathcal{T}_{A}(K))$ defined in \eqref{eq:shape-function}. To this end, we introduce the numbers $B_{t, n}$ for all $0 \le t \le d - 1$ and $0 \le n \le r_{d - t}$, which are defined as
\begin{equation}
    \label{eq:s-t-n}
    B_{0, n} \coloneqq 1, \qquad B_{t, n} \coloneqq \card \big(\Sigma_{0}^{\bm{q}_{t, n}}(\mathbb{I}_{t}, k - n)\big), \quad \forall~t,\ 1 \le t \le d - 1,
\end{equation}
where the continuity vector $\bm{q}_{t, n} \coloneqq (r_{d - t + 1} - n, \ldots, r_{d} - n)$. These numbers depend only on $\bm{r}$ and $k$, and can be computed recursively by
\begin{equation*}
    B_{t, n} = \binom{k - n + t}{t} - \sum_{t' = 0}^{t - 1} \binom{t + 1}{t' + 1} \sum_{n' = n}^{r_{d - t'}} \binom{n' - n + (t - t' -1)}{t - t' - 1} B_{t', n'},
\end{equation*}
which will be proved in Lemma~\ref{lem:number-B-t-n}. The dimension of the space $\mathcal{S}_{k}^{\bm{r}, \rho}(\mathcal{T}_{A}(\Omega))$ can be characterized as follows.

\begin{theorem}
    \label{thm:dim-superspline}
    Suppose that the continuity vector $\bm{r} \coloneqq (r_{1}, \ldots, r_{d}) \in \mathbb{N}_{0}^{d}$, the split point continuity $\rho \in \mathbb{N}_{0}$, the polynomial degree $k \in \mathbb{N}_{0}$, and the boundary layer number $b \in \mathbb{N}_{0}$ jointly satisfy Condition~\ref{cond:Cr-macro-element}. The dimension of the superspline space $\mathcal{S}_{k}^{\bm{r}, \rho}(\mathcal{T}_{A}(\Omega))$ defined in \eqref{eq:superspline} equals
    \begin{equation*}
        \begin{aligned}
            \dim \mathcal{S}_{k}^{\bm{r}, \rho}(\mathcal{T}_{A}(\Omega)) = \; & \binom{\rho + d}{d} N_{d} + \sum_{t = 0}^{d - 1} N_{t} \sum_{n = 0}^{r_{d - t}} \binom{n + (d - t - 1)}{d - t - 1} B_{t, n} \\
            & - \sum_{t = 0}^{d - 1} \binom{d + 1}{t + 1} N_{d} \sum_{n = b}^{r_{d - t}} \binom{n - b + (d - t - 1)}{d - t - 1} B_{t, n},
        \end{aligned}
    \end{equation*}
    and
    \begin{equation*}
        \begin{aligned}
            \dim \mathcal{S}_{k}^{\bm{r}, \rho}(\mathcal{T}_{A}(\Omega)) = \; & \binom{k + d}{d} N_{d} + \sum_{t = 0}^{d - 1} \bigg(N_{t} - \binom{d}{t + 1} N_{d}\bigg) \sum_{n = 0}^{r_{d - t}} \binom{n + (d - t - 1)}{d - t - 1} B_{t, n} \\
            & - \sum_{t = 0}^{d - 1} \binom{d}{t} N_{d} \sum_{n = b}^{r_{d - t}} \binom{n - b + (d - t - 1)}{d - t - 1} B_{t, n},
        \end{aligned}
    \end{equation*}
    where $N_{t}$ is the number of $t$-dimensional $(0 \le t \le d)$ simplices present in $\mathcal{T}(\Omega)$ and $B_{t, n}$ is defined in \eqref{eq:s-t-n}.
\end{theorem}

The proof of Theorem~\ref{thm:dim-superspline} is provided in Section~\ref{sec:dimension}. For the special case $\Omega = K$ (i.e., a single $d$-dimensional simplex), the number of $t$-dimensional ($0 \le t \le d$) simplices present in $K$ is
\begin{equation*}
    N_{t} = \binom{d + 1}{t + 1},
\end{equation*}
and the expression for the shape function space $\mathcal{S}_{k}^{\bm{r}, \rho}(\mathcal{T}_{A}(K))$ follows; see \eqref{eq:dim-shape}.

\begin{example}[Dimension in the two-dimensional case] 
\label{ex:dimension-2d}
    Consider the case $d = 2$, and choose $\bm{r}$, $\rho$, $k$, and $b$ as in Table~\ref{tab:minimum-degree-dim-2-3}, which minimize the dimension of $\mathcal{S}_{k}^{\bm{r}, \rho}(\mathcal{T}_{A}(K))$, according to the rule in Table~\ref{tab:minimum-degree}.
    Then, it holds that $B_{1, n} = n$ and
    \begin{equation*}
        \dim \mathcal{S}_{k}^{\bm{r}, \rho}(\mathcal{T}_{A}(\Omega)) = \begin{dcases}
            \frac{9 m^{2} - 3 m}{2} N_{0} + (2 m^{2} - m) N_{1} + (2 m^{2} - 3 m + 1) N_{2}, & r_{1} = 2 m - 1, \\
            \frac{9 m^{2} + 9 m + 2}{2} N_{0} + (2 m^{2} + m) N_{1} + (2 m^{2} - m) N_{2}, & r_{1} = 2 m. \\
        \end{dcases}
    \end{equation*}
    Note that the dimension can be expressed alternatively as
    \begin{equation*}
        \dim \mathcal{S}_{k}^{\bm{r}, \rho}(\mathcal{T}_{A}(\Omega)) = \binom{r_2+2}{2} N_{0} + \binom{r_1+1}{2} N_{1} + \binom{r_1}{2} N_{2}.
    \end{equation*}
    This is in complete agreement with the dimension results from \cite{lai2001} (see also \cite{lai2007}).
\end{example}

\begin{example}[Dimension in the three-dimensional case]
\label{ex:dimension-3d}
    Consider the case $d = 3$, and choose $\bm{r}$, $\rho$, $k$, and $b$ as in Table~\ref{tab:minimum-degree-dim-2-3}, which minimize the dimension of $\mathcal{S}_{k}^{\bm{r}, \rho}(\mathcal{T}_{A}(K))$, according to the rule in Table~\ref{tab:minimum-degree}.
    Then, it holds that $B_{1, n} = n$, $B_{2, n} = \binom{r_{2}}{2} + 2 r_{2} n + \binom{n + 1}{2}$ and
    \begin{equation*}
        \dim \mathcal{S}_{k}^{\bm{r}, \rho}(\mathcal{T}_{A}(\Omega)) = \begin{dcases}
            (36 m^{3} - 36 m^{2} + 11 m - 1) N_{0} + (9 m^{3} - 9 m^{2} + 2 m) N_{1} \\
            \quad + \frac{134 m^{3} - 174 m^{2} + 58 m}{6} N_{2} \\
            \quad + \frac{311 m^{3} - 639 m^{2} + 430 m - 96}{6} N_{3}, & r_{1} = 2 m - 1, \\
            (36 m^{3} + 36 m^{2} + 11 m + 1) N_{0} + (9 m^{3} + 9 m^{2} + 2 m) N_{1} \\
            \quad + \frac{134 m^{3} + 57 m^{2} - 5 m}{6} N_{2} \\
            \quad + \frac{311 m^{3} - 99 m^{2} - 2 m}{6} N_{3}, & r_{1} = 2 m. \\
        \end{dcases}
    \end{equation*}
    Note that the dimension can be expressed alternatively as
    \begin{equation*}
        \begin{aligned}
            \dim \mathcal{S}_{k}^{\bm{r}, \rho}(\mathcal{T}_{A}(\Omega)) = \; & \binom{2 r_{2} + 3}{3} N_{0} + 2 \binom{r_{2} + 2}{3} N_{1} \\
            & + \bigg[\binom{r_{2}}{2} (r_{1} + 1) + 2 r_{2} \binom{r_{1} + 1}{2} + \binom{r_{1} + 2}{3}\bigg] N_{2} \\
            & + \bigg[\binom{r_{2} + 2 r_{1}}{3} - 4 \binom{r_{1}}{3}\bigg] N_{3}.
        \end{aligned}
    \end{equation*}
    This is in complete agreement with the dimension results from \cite{lai2007} for $r = 1,2$ and \cite{lai2013} for arbitrary $r \ge 1$.
\end{example}

\begin{table}[t]
    \centering
    \begin{tabular}{c|c|c|c|c}
        \hline
        \multicolumn{5}{c}{$d = 2$} \\
        \hline
        $r_{1}$ & $r_{2} = \lceil\frac{3 r_{1} - 1}{2}\rceil$ & $\rho = 3 r_{2} - 2 r_{1} + 1$ & $k = 2 r_{2} + 1$ & $b = 2 r_{1} - r_{2}$ \\
        \hline
        $2 m - 1$ & $3 m - 2$ & $5 m - 3$ & $6 m - 3$ & $m$ \\
        \hline
        $2 m$ & $3 m$ & $5 m + 1$ & $6 m + 1$ & $m$ \\
        \hline
    \end{tabular}
    \\[10pt]
    \begin{tabular}{c|c|c|c|c|c}
        \hline
        \multicolumn{6}{c}{$d = 3$} \\
        \hline
        $r_{1}$ & $r_{2} = \lceil\frac{3 r_{1} - 1}{2}\rceil$ & $r_{3} = 2 r_{2}$ & $\rho = 5 r_{2} - 2 r_{1} + 1$ & $k = 4 r_{2} + 1$ & $b = 2 r_{1} - r_{2}$ \\
        \hline
        $2 m - 1$ & $3 m - 2$ & $6 m - 4$ & $11 m - 7$ & $12 m - 7$ & $m$ \\
        \hline
        $2 m$ & $3 m$ & $6 m$ & $11 m + 1$ & $12 m + 1$ & $m$ \\
        \hline
    \end{tabular}
    \caption{Choices of $\bm{r} \coloneqq (r_{1}, \ldots, r_{d})$, $\rho$, $k$, and $b$, which minimize the dimension of $\mathcal{S}_{k}^{\bm{r}, \rho}(\mathcal{T}_{A}(K))$ for a given value of $r_1$ in the cases $d=2$ (Example~\ref{ex:dimension-2d}) and $d=3$ (Example~\ref{ex:dimension-3d}).}
    \label{tab:minimum-degree-dim-2-3}
\end{table}

\section{Pair of Refined Intrinsic Decompositions}
\label{sec:decomposition}

In this section, we propose a pair of refined intrinsic decompositions for the Alfeld split. The notion of refined intrinsic decomposition was introduced in \cite{hu2024construction}, which provides a correspondence between a set of degrees of freedom and a multi-index set. 

It should be emphasized that the degrees of freedom \eqref{eq:set-dof-1} and \eqref{eq:set-dof-2} might be linearly dependent as linear functionals on $\mathcal{P}_{k}(K)$, which constitutes the main difficulty in proving the unisolvence of the degrees of freedom \eqref{eq:set-dof-1}--\eqref{eq:set-dof-3}. To overcome this difficulty, we separate the degrees of freedom into two groups, each of which corresponds to a refined intrinsic decomposition. 

The first group consists of the degrees of freedom in \eqref{eq:set-dof-1} and a subset of those in \eqref{eq:set-dof-2}, which are linearly independent as linear functionals on $\mathcal{P}_{k}(K)$ and correspond to Decomposition~\ref{decomp:A}. The second group consists of the potentially linearly dependent part of \eqref{eq:set-dof-2} together with the bubble function degrees of freedom in \eqref{eq:set-dof-3}, which are linearly independent as linear functionals on $Q_{k}^{\rho}(\mathcal{P}_{\rho}(K))$ and correspond to Decomposition~\ref{decomp:B}.

To compute such decompositions, we first recall the assumptions on the index set in \cite{hu2024construction} for a $\bar{d}$-dimensional space. In this section, we always assume $d \ge 2$ and $\bar{d} \ge 1$.

\begin{condition}
    \label{cond:Cr-element}
    For the continuity vector $\bar{\bm{r}} \coloneqq (\bar{r}_{1}, \bar{r}_{2}, \ldots, \bar{r}_{\bar{d}}) \in \mathbb{N}_{0}^{\bar{d}}$ and the polynomial degree $\bar{k} \in \mathbb{N}_{0}$, we say $\bar{\bm{r}}$ and $\bar{k}$ satisfy Condition~\ref{cond:Cr-element} if 
    \begin{equation*}
        2^{\bar{d} - 1} \bar{r}_{1} \le 2^{\bar{d} - 2} \bar{r}_{2} \le \cdots \le \bar{r}_{\bar{d}}, \quad 2 \bar{r}_{\bar{d}} + 1 \le \bar{k}.
    \end{equation*}
\end{condition}

Given $\bar{\bm{r}}$ and $\bar{k}$, for each $t$-dimensional ($0 \le t \le \bar{d} - 1$) subsimplex $F$ of the $\bar{d}$-dimensional simplex $K$, and each $n$ such that $0 \le n \le \bar{r}_{\bar{d} - t}$, define the multi-index set $\Sigma^{\bar{\bm{r}}}_{F, n}(\mathbb{I}_{\bar{d}}, \bar{k})$ as
\begin{equation*}
    \begin{aligned}
        \Sigma^{\bar{\bm{r}}}_{F, n}(\mathbb{I}_{\bar{d}}, \bar{k}) \coloneqq \big\{\bm{\alpha} \in \Sigma(\mathbb{I}_{\bar{d}}, \bar{k}): \; & |\bm{\alpha}|_{\setminus F} = n, \text{ and }|\bm{\alpha}|_{\setminus E} \ge \bar{r}_{\bar{d} - t'} + 1 \\
        & \text{for each } t' \text{-dimensional } (0 \le t' \le t - 1) \\
        & \text{subsimplex } E \text{ of } K\big\}.
    \end{aligned}
\end{equation*}

\begin{proposition}[Refined intrinsic decomposition corresponding to $\bar{\bm{r}}$ \cite{hu2024construction}]
    Suppose that the continuity vector $\bar{\bm{r}} \coloneqq (\bar{r}_{1}, \ldots, \bar{r}_{\bar{d}}) \in \mathbb{N}_{0}^{\bar{d}}$ and the polynomial degree $\bar{k} \in \mathbb{N}_{0}$ jointly satisfy Condition~\ref{cond:Cr-element}. 
    Then, the refined intrinsic decomposition of $\Sigma(\mathbb{I}_{\bar{d}}, \bar{k})$ corresponding to $\bar{\bm{r}}$ is given by the following disjoint union 
    \begin{equation}\label{eq:decompose-full}
        \Sigma(\mathbb{I}_{\bar{d}}, \bar{k}) = \Sigma_{0}^{\bar{\bm{r}}}(\mathbb{I}_{\bar{d}}, \bar{k}) \cup \Bigg(\bigcup_{F \in \mathcal{T}^{\partial}(K)} \bigcup_{n = 0}^{\bar{r}_{\bar{d} - t_{F}}} \Sigma_{F, n}^{\bar{\bm{r}}}(\mathbb{I}_{\bar{d}}, \bar{k})\Bigg),
    \end{equation}
    where $\mathcal{T}^{\partial}(K)$ is the collection of all the $t$-dimensional ($0 \le t \le \bar{d} - 1$) subsimplices $F$ of $K$, and $t_{F}$ is the dimension of $F$. 
\end{proposition}

\begin{figure}[htp]
    \centering
    \includegraphics[width=0.8\textwidth]{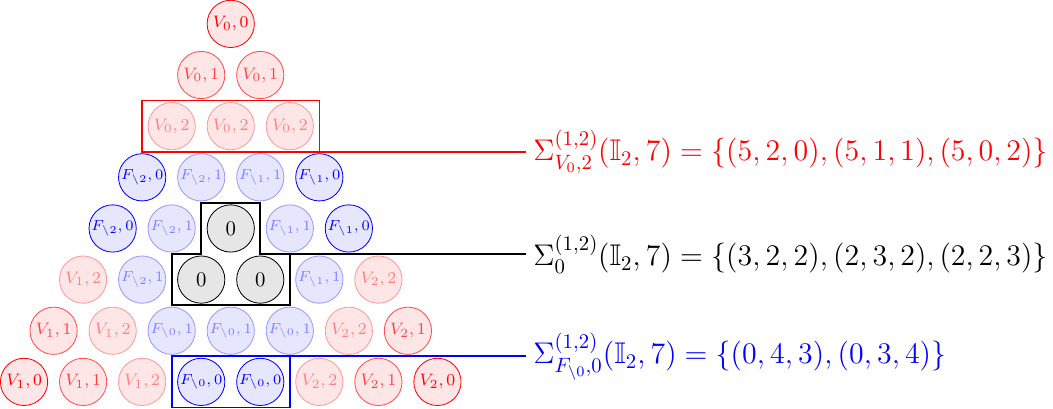}
    \caption{Illustration of the refined intrinsic decomposition of $\Sigma(\mathbb{I}_{\bar{d}}, \bar{k})$ in \eqref{eq:decompose-full} for $\bar{d} = 2$, $\bar{\bm{r}} = (1, 2)$, and $\bar{k} = 7$. The multi-indices related to a vertex are visualized as red disks, the multi-indices related to an edge as blue disks, and the remaining multi-indices as gray disks. Different shades of a color indicate different layers of multi-indices.}
    \label{fig:refined}
\end{figure}

The refined intrinsic decomposition in \eqref{eq:decompose-full} is illustrated in Figure~\ref{fig:refined} for $\bar{d} = 2$, $\bar{\bm{r}} = (1, 2)$, and $\bar{k} = 7$. 
The next lemma follows directly from \cite[Proposition 2.14]{hu2024construction}, with a change in notation. 
\begin{lemma}[Bijection lemma]
    \label{lem:bij-lemma}
    Suppose that the continuity vector $\bar{\bm{r}} \coloneqq (\bar{r}_{1}, \ldots, \bar{r}_{\bar{d}}) \in \mathbb{N}_{0}^{\bar{d}}$ and the polynomial degree $\bar{k} \in \mathbb{N}_{0}$ jointly satisfy Condition~\ref{cond:Cr-element}. Let $F$ be a $t$-dimensional subsimplex of $K$, where $\mathbb{I}_{F} \coloneqq \{i_{0}, i_{1}, \ldots, i_{t}\}$ and $\mathbb{I}_{\setminus F} \coloneqq \{i_{t + 1}, \ldots, i_{\bar{d}}\}$.
    \begin{enumerate}
        \item  For the case $t = 0$, i.e., $F$ is a vertex $V$, there exists a bijection
        \begin{equation*}
            \begin{aligned}
                \Sigma^{\bar{\bm{r}}}_{V, n}(\mathbb{I}_{\bar{d}}, \bar{k}) & \longrightarrow \Sigma(\mathbb{I}_{\bar{d}, \setminus 0}, n), \\
                \bm{\alpha} & \longmapsto \bm{\theta}
            \end{aligned}
        \end{equation*}
        such that $\bm{\theta} \coloneqq (\theta_{1}, \ldots, \theta_{\bar{d}}) = (\alpha_{i_{1}}, \ldots, \alpha_{i_{\bar{d}}})$.
        \item For the case $1 \le t \le \bar{d} - 1$, there exists a bijection
        \begin{equation*}
            \begin{aligned}
                \Sigma^{\bar{\bm{r}}}_{F, n}(\mathbb{I}_{\bar{d}}, \bar{k}) & \longrightarrow \Sigma(\mathbb{I}_{\bar{d}, \setminus t}, n) \times \Sigma_{0}^{\bar{\bm{q}}_{t, n}}(\mathbb{I}_{F}, \bar{k} - n), \\
                \bm{\alpha} & \longmapsto (\bm{\theta}, \bm{\sigma})
            \end{aligned}
        \end{equation*}
        such that $\bm{\theta} \coloneqq (\theta_{t + 1}, \ldots, \theta_{\bar{d}}) = (\alpha_{i_{t + 1}}, \ldots, \alpha_{i_{\bar{d}}})$ and $\sigma_{i} = \alpha_{i}$ for $i \in \mathbb{I}_{F}$, where the continuity vector $\bar{\bm{q}}_{t, n} \coloneqq (\bar{r}_{\bar{d} - t + 1} - n, \ldots, \bar{r}_{\bar{d}} - n)$. 
    \end{enumerate}
\end{lemma}

Each of these multi-index sets corresponds to the degrees of freedom defined on $F$ with derivatives of order $n$, through the bijection $\bm{\alpha} \mapsto \bm{\theta}$ or $\bm{\alpha} \mapsto (\bm{\theta}, \bm{\sigma})$. For a degree of freedom defined on $F$, the multi-index $\bm{\theta} \in \Sigma(\mathbb{I}_{\bar{d}, \setminus t_{F}}, n)$ determines the normal derivative, while the multi-index $\bm{\sigma} \in \Sigma_{0}^{\bar{\bm{q}}_{t_{F}, n}}(\mathbb{I}_{F}, \bar{k} - n)$ determines the associated polynomial $\llbracket\bm{\lambda}\rrbracket_{F}^{\bm{\sigma}}$, with $t_{F}$ the dimension of $F$. 

In the remainder of the section, we always assume that $\bm{r}$, $\rho$, $k$, and $b$ jointly satisfy Condition~\ref{cond:Cr-macro-element}. Given a continuity vector $\bm r$ and boundary layer number $b$, define the \emph{boundary continuity vector} as
    \begin{equation}
        \label{eq:r-boundary}
        \bm{r}^{\partial} \coloneqq (r^{\partial}_{1}, \ldots, r^{\partial}_{d}) = (b - 1, r_{2}, \ldots, r_{d}),
    \end{equation}
    and the \emph{interior continuity vector} as
    \begin{equation}
        \label{eq:r-interior}
        \bm{r}^{\circ} \coloneqq (r^{\circ}_{1}, \ldots, r^{\circ}_{d}) = (r_{1} - b, r_{2} - b, \ldots, r_{d} - b).
    \end{equation}

\begin{lemma}
    \label{lem:asm}
    Suppose that $\bm{r}$, $\rho$, $k$, and $b$ jointly satisfy Condition~\ref{cond:Cr-macro-element}. 
    Then, 
    \begin{enumerate}
        \item $\bm{r}^{\partial}$ and $k$ jointly satisfy Condition~\ref{cond:Cr-element};
        \item $\bm{r}^{\circ}$ and $\rho$ jointly satisfy Condition~\ref{cond:Cr-element}.
    \end{enumerate}
\end{lemma}

\begin{proof}
    For the first part, we have $r^{\partial}_{2} - 2r^{\partial}_{1} \ge 0$, since $2 (b - 1) \le 2 (r_{2} - r_{1}) \le r_{2} - 2 r_{1} + (2 r_{1} - 1) \le r_{2}$. The other inequalities in Condition~\ref{cond:Cr-element} follow directly from Condition~\ref{cond:Cr-macro-element}.

    For the second part, we have $r^{\circ}_{1} \ge 0$, since $b \le r_{2} - r_{1} + 1 \le (2 r_{1} - 1) - r_{1} + 1 = r_{1}$. Moreover, for $2 \le s \le d$, it holds that $r^{\circ}_{s} - 2 r^{\circ}_{s - 1} = r_{s} - 2 r_{s - 1} + b \ge 0$, and $\rho - 2 r^{\circ}_{d} - 1 = k - 2 r_{d} - 1 + b \ge 0$. This completes the proof.
\end{proof}

By Lemma~\ref{lem:asm}, two sets of refined intrinsic decompositions of multi-index sets can be constructed. The first decomposition corresponds to linearly independent functionals on $\mathcal{P}_{k}(K)$, namely the degrees of freedom in \eqref{eq:set-dof-1} and a subset of those in \eqref{eq:set-dof-2}.

\begin{decomposition}[Refined intrinsic decomposition corresponding to $\bm r^{\partial}$]
    \label{decomp:A}
    There exists a refined intrinsic decomposition of $\Sigma(\mathbb{I}_{d}, k)$ corresponding to $\bm{r}^{\partial} \coloneqq (r^{\partial}_{1}, \ldots, r^{\partial}_{d})$ defined in \eqref{eq:r-boundary}, such that
    \begin{equation*}
        \Sigma(\mathbb{I}_{d}, k) = \Sigma_{0}^{\bm{r}^{\partial}}(\mathbb{I}_{d}, k) \cup \Bigg(\bigcup_{F \in \mathcal{T}^{\partial}(K)} \bigcup_{n = 0}^{r^{\partial}_{d - t_{F}}} \Sigma_{F, n}^{\bm{r}^{\partial}}(\mathbb{I}_{d}, k)\Bigg),
    \end{equation*}
where $F$ is a $t_{F}$-dimensional subsimplex of $K$.
\end{decomposition}

Let $F$ be a $t$-dimensional subsimplex of $K$, where $\mathbb{I}_{F} \coloneqq \{i_{0}, i_{1}, \ldots, i_{t}\}$ and $\mathbb{I}_{\setminus F} \coloneqq \{i_{t + 1}, \ldots, i_{d}\}$. Now, we apply Lemma~\ref{lem:bij-lemma}:

\begin{itemize}
\renewcommand\labelitemi{-}
    \item For $t_{F} = 0$, i.e., $F$ is a vertex $V$, and $0 \le n \le r^{\partial}_{d} = r_{d}$, the multi-index set $\Sigma_{V, n}^{\bm{r}^{\partial}}(\mathbb{I}_{d}, k)$ is defined as
\begin{equation*}
    \Sigma_{V, n}^{\bm{r}^{\partial}}(\mathbb{I}_{d}, k) \coloneqq \big\{\bm{\alpha} \in \Sigma(\mathbb{I}_{d}, k): |\bm{\alpha}|_{\setminus V} = n, \text{ and } |\bm{\alpha}|_{V} = k - n\big\}.
\end{equation*}
There exists a bijection
\begin{equation*}
    \begin{aligned}
        \mathcal{R}^{\partial}: \Sigma_{V, n}^{\bm{r}^{\partial}}(\mathbb{I}_{d}, k) & \longrightarrow \Sigma(\mathbb{I}_{d, \setminus 0}, n), \\
        \bm{\alpha} & \longmapsto \bm{\theta}
    \end{aligned}
\end{equation*}
such that $\bm{\theta} \coloneqq (\theta_{1}, \ldots, \theta_{d}) = (\alpha_{i_{1}}, \ldots, \alpha_{i_{d}})$.
Then, we define the degree of freedom $\varphi_{\bm{\alpha}}$ for $\bm{\alpha} \in \Sigma_{V, n}^{\bm{r}^{\partial}}(\mathbb{I}_{d}, k)$ as
\begin{equation}
    \label{eq:alpha-dof-vertex}
    \varphi_{\bm{\alpha}}: u \longmapsto \bm{D}_{V}^{\bm{\theta}} u\big|_{V}.
\end{equation}
\item 
For $t_{F} = t$ such that $1 \le t \le d - 2$, and $0 \le n \le r^{\partial}_{d - t} = r_{d - t}$, the multi-index set $\Sigma_{F, n}^{\bm{r}^{\partial}}(\mathbb{I}_{d}, k)$ is defined as 
\begin{equation}
    \label{eq:Sigma-F-n}
    \begin{aligned}
        \Sigma_{F, n}^{\bm{r}^{\partial}}(\mathbb{I}_{d}, k) \coloneqq \big\{\bm{\alpha} \in \Sigma(\mathbb{I}_{d}, k): \; & |\bm{\alpha}|_{\setminus F} = n, \text{ and }|\bm{\alpha}|_{\setminus E} \ge r_{d - t'} + 1 \\
        & \text{for each } t' \text{-dimensional } (0 \le t' \le t - 1) \\
        & \text{subsimplex } E \text{ of } K\big\}.
    \end{aligned}
\end{equation}
There exists a bijection
\begin{equation*}
    \begin{aligned}
        \mathcal{R}^{\partial}: \Sigma_{F, n}^{\bm{r}^{\partial}}(\mathbb{I}_{d}, k) & \longrightarrow \Sigma(\mathbb{I}_{d, \setminus t}, n) \times \Sigma_{0}^{\bm{q}^{\partial}_{t, n}}(\mathbb{I}_{F}, k - n), \\
        \bm{\alpha} & \longmapsto (\bm{\theta}, \bm{\sigma})
    \end{aligned}
\end{equation*}
such that $\bm{\theta} \coloneqq (\theta_{t + 1}, \ldots, \theta_{d}) = (\alpha_{i_{t + 1}}, \ldots, \alpha_{i_{d}})$ and $\sigma_{i} = \alpha_{i}$ for $i \in \mathbb{I}_{F}$, where the continuity vector $\bm{q}^{\partial}_{t, n} \coloneqq (r^{\partial}_{d - t + 1} - n, \ldots, r^{\partial}_{d} - n) = (r_{d - t + 1} - n, \ldots, r_{d} - n) =: \bm{q}_{t, n}$.
Then, we define the degree of freedom $\varphi_{\bm{\alpha}}$ for $\bm{\alpha} \in \Sigma_{F, n}^{\bm{r}^{\partial}}(\mathbb{I}_{d}, k)$ as
\begin{equation}
    \label{eq:alpha-dof}
    \varphi_{\bm{\alpha}}: u \longmapsto \frac{1}{|F|} \int_{F} (\bm{D}_{F}^{\bm{\theta}} u) \cdot \llbracket\bm{\lambda}\rrbracket_{F}^{\bm{\sigma}}.
\end{equation}
\item 
For $t_{F} = d - 1$, and $0 \le n \le r^{\partial}_{1} = b - 1$, the multi-index set $\Sigma_{F, n}^{\bm{r}^{\partial}}(\mathbb{I}_{d}, k)$ is defined by \eqref{eq:Sigma-F-n}, and there exists a bijection
\begin{equation*}
    \begin{aligned}
        \mathcal{R}^{\partial}: \Sigma_{F, n}^{\bm{r}^{\partial}}(\mathbb{I}_{d}, k) & \longrightarrow \Sigma(\{d\}, n) \times \Sigma_{0}^{\bm{q}^{\partial}_{d - 1, n}}(\mathbb{I}_{F}, k - n), \\
        \bm{\alpha} & \longmapsto (\bm{\theta}, \bm{\sigma})
    \end{aligned}
\end{equation*}
such that $\bm{\theta} \coloneqq (\theta_{d}) = (\alpha_{j}) = (n)$ and $\sigma_{i} = \alpha_{i}$ for $i \in \mathbb{I}_{F}$, where $\{j\} = \mathbb{I}_{\setminus F}$ and the continuity vector $\bm{q}^{\partial}_{d - 1, n} \coloneqq (r^{\partial}_{2} - n, \ldots, r^{\partial}_{d} - n) = (r_{2} - n, \ldots, r_{d} - n) =: \bm{q}_{d - 1, n}$. Then, we define the degree of freedom $\varphi_{\bm{\alpha}}$ for $\bm{\alpha} \in \Sigma_{F, n}^{\bm{r}^{\partial}}(\mathbb{I}_{d}, k)$ as
\begin{equation}
    \label{eq:alpha-dof-d-1}
    \varphi_{\bm{\alpha}}: u \longmapsto \frac{1}{|F|} \int_{F} (\bm{D}_{F}^{\bm{\theta}} u) \cdot \llbracket\bm{\lambda}\rrbracket_{F}^{\bm{\sigma}} = \frac{1}{|F|} \int_{F} \bigg(\frac{\partial^{n}}{\partial \bm{n}_{F}^{n}} u\bigg) \cdot \llbracket\bm{\lambda}\rrbracket_{F}^{\bm{\sigma}}.
\end{equation}
\end{itemize}

\begin{remark}
    We do not use $\Sigma_0^{\bm r^{\partial}} (\mathbb I_d, k)$ to define our degrees of freedom. Therefore, we do not give further details about this multi-index set.
\end{remark}

The second decomposition corresponds to linearly independent functionals on $Q_{k}^{\rho}(\mathcal{P}_{\rho}(K))$, namely the potentially linearly dependent part of \eqref{eq:set-dof-2} together with the bubble function degrees of freedom in \eqref{eq:set-dof-3}. 

\begin{decomposition}[Refined intrinsic decomposition corresponding to $\bm r^{\circ}$]
    \label{decomp:B}
    There exists a refined intrinsic decomposition of $\Sigma(\mathbb{I}_{d}, \rho)$ corresponding to $\bm{r}^{\circ} \coloneqq (r^{\circ}_{1}, \ldots, r^{\circ}_{d})$ defined in \eqref{eq:r-interior}, such that
    \begin{equation}
        \label{eq:decompose-beta}
        \Sigma(\mathbb{I}_{d}, \rho) = \Sigma_{0}^{\bm{r}^{\circ}}(\mathbb{I}_{d}, \rho) \cup \Bigg(\bigcup_{F \in \mathcal{T}^{\partial}(K)} \bigcup_{n = 0}^{r^{\circ}_{d - t_{F}}} \Sigma_{F, n}^{\bm{r}^{\circ}}(\mathbb{I}_{d}, \rho)\Bigg),
    \end{equation}
    where $F$ is a $t_{F}$-dimensional subsimplex of $K$. 
\end{decomposition}

In this decomposition, $\Sigma_{0}^{\bm{r}^{\circ}}(\mathbb{I}_{d}, \rho)$ is defined by \eqref{eq:Sigma-0-d}.
For $t_{F} = d - 1$, and $0 \le n \le r^{\circ}_{1} = r_{1} - b$, there exists a bijection
\begin{equation*}
    \begin{aligned}
        \mathcal{R}^{\circ}: \Sigma_{F, n}^{\bm{r}^{\circ}}(\mathbb{I}_{d}, \rho) &\longrightarrow \Sigma(\{d\}, n) \times \Sigma_{0}^{\bm{q}^{\circ}_{d - 1, n}}(\mathbb{I}_{F}, \rho - n), \\
        \bm{\beta} & \longmapsto (\bm{\theta}, \bm{\sigma})\end{aligned}
\end{equation*}
such that $\bm{\theta} \coloneqq (\theta_{d}) = (\beta_{j}) = (n)$ and $\sigma_{i} = \beta_{i}$ for $i \in \mathbb{I}_{F}$, where $\{j\} = \mathbb{I}_{\setminus F}$ and the continuity vector $\bm{q}^{\circ}_{d - 1, n} \coloneqq (r^{\circ}_{2} - n, \ldots, r^{\circ}_{d} - n) = (r_{2} - (n + b), \ldots, r_{d} - (n + b)) =: \bm{q}_{d - 1, n + b}$.
Therefore, for $t_{F} = d - 1$, and $b \le n \le r_{1}$, the multi-index set $\Sigma_{F, n - b}^{\bm{r}^{\circ}}(\mathbb{I}_{d}, \rho)$ is defined as 
\begin{equation}
    \label{eq:Sigma-F-n-b}
    \begin{aligned}
        \Sigma_{F, n - b}^{\bm{r}^{\circ}}(\mathbb{I}_{d}, \rho) \coloneqq \big\{\bm{\beta} \in \Sigma(\mathbb{I}_{d}, \rho): \; & |\bm{\beta}|_{\setminus F} = n - b, \text{ and }|\bm{\beta}|_{\setminus E} \ge r_{d - t'} - b + 1 \\
        & \text{for each } t' \text{-dimensional } (0 \le t' \le d - 2) \\
        & \text{subsimplex } E \text{ of } K\big\}.
    \end{aligned}
\end{equation}
There exists a bijection
\begin{equation*}
    \begin{aligned}
        \mathcal{R}^{\circ}_{b}: \Sigma_{F, n - b}^{\bm{r}^{\circ}}(\mathbb{I}_{d}, \rho) & \longrightarrow \Sigma(\{d\}, n) \times \Sigma_{0}^{\bm{q}_{d - 1, n}}(\mathbb{I}_{F}, k - n), \\
        \bm{\beta} & \longmapsto (\bm{\theta}, \bm{\sigma})
    \end{aligned}
\end{equation*}
such that $\bm{\theta} \coloneqq (\theta_{d}) = (\beta_{j} + b) = (n)$ and $\sigma_{i} = \beta_{i}$ for $i \in \mathbb{I}_{F}$, where $\{j\} = \mathbb{I}_{\setminus F}$ and the continuity vector $\bm{q}_{d - 1, n} \coloneqq (r_{2} - n, \ldots, r_{d} - n)$. Then, we define the degree of freedom $\psi_{\bm{\beta}}$ for $\bm{\beta} \in \Sigma_{F, n - b}^{\bm{r}^{\circ}}(\mathbb{I}_{d}, \rho)$ as
\begin{equation}
    \label{eq:beta-dof-d-1}
    \psi_{\bm{\beta}}: u \longmapsto \frac{1}{|F|} \int_{F} (\bm{D}_{F}^{\bm{\theta}} u) \cdot \llbracket\bm{\lambda}\rrbracket_{F}^{\bm{\sigma}} = \frac{1}{|F|} \int_{F} \bigg(\frac{\partial^{n}}{\partial \bm{n}_{F}^{n}} u\bigg) \cdot \llbracket\bm{\lambda}\rrbracket_{F}^{\bm{\sigma}}.
\end{equation}

Finally, we define the degree of freedom $\psi_{\bm{\beta}}$ for $\bm{\beta} \in \Sigma_{0}^{\bm{r}^{\circ}}(\mathbb{I}_{d}, \rho)$ as
\begin{equation}
    \label{eq:beta-dof-d}
    \psi_{\bm{\beta}}: u \longmapsto \frac{1}{|K|} \int_{K} u \cdot Q_{k}^{\rho}(\llbracket\bm{\lambda}\rrbracket^{\bm{\beta}}).
\end{equation}

\begin{remark}
    We do not use $\Sigma_{F, n}^{\bm{r}^{\circ}}(\mathbb{I}_{d}, \rho)$ for any $F \in \mathcal{T}^{\partial}(K)$ with $0 \le t_F \le d-2$ to define our degrees of freedom. Therefore, we do not give further details about this multi-index set.
\end{remark}

With the help of the pair of refined intrinsic decompositions and the bijections $\mathcal{R}^{\partial}$ and $\mathcal{R}^{\circ}_{b}$, we can express the set of degrees of freedom defined in \eqref{eq:set-dof-1}--\eqref{eq:set-dof-3} in a multi-index way.

\begin{proposition}
    \label{prop:equivalent-set-dofs}
    Suppose that $\bm{r}$, $\rho$, $k$, and $b$ jointly satisfy Condition~\ref{cond:Cr-macro-element}. 
    Then, the set of degrees of freedom defined in \eqref{eq:set-dof-1}--\eqref{eq:set-dof-3} can be expressed as
    \begin{equation}
        \label{eq:equivalent-set-dofs}
        \big\{\varphi_{\bm{\alpha}}: \bm{\alpha} \in \Sigma^{\partial}\big\} \cup \big\{\psi_{\bm{\beta}}: \bm{\beta} \in \Sigma^{\circ}\big\},
    \end{equation}
    where the multi-index sets $\Sigma^{\partial}$ and $\Sigma^{\circ}$ are defined as
    \begin{equation}
        \label{eq:alpha-set}
        \Sigma^{\partial} \coloneqq \bigcup_{F \in \mathcal{T}^{\partial}(K)} \bigcup_{n = 0}^{r^{\partial}_{d - t_{F}}} \Sigma_{F, n}^{\bm{r}^{\partial}}(\mathbb{I}_{d}, k),
    \end{equation}
    and
    \begin{equation}
        \label{eq:beta-set}
        \Sigma^{\circ} \coloneqq \Sigma_{0}^{\bm{r}^{\circ}}(\mathbb{I}_{d}, \rho) \cup \Bigg(\bigcup_{j \in \mathbb{I}_{d}} \bigcup_{n = b}^{r_{1}} \Sigma_{F_{\setminus j}, n - b}^{\bm{r}^{\circ}}(\mathbb{I}_{d}, \rho)\Bigg).
    \end{equation}
    Recall that $F_{\setminus j}$ is the common ($d - 1$)-dimensional subsimplex of $K$ and $K_{j}$.
\end{proposition}

\begin{proof}
    For each vertex $V$ of $K$, since $\Sigma_{V, n}^{\bm{r}^{\partial}}(\mathbb{I}_{d}, k) \to \Sigma(\mathbb{I}_{d, \setminus 0}, n)$ is a bijection, the degrees of freedom \eqref{eq:set-dof-1} defined on $V$ can be expressed as $\big\{\varphi_{\bm{\alpha}}: \bm{\alpha} \in \bigcup_{n = 0}^{r^{\partial}_{d}} \Sigma_{V, n}^{\bm{r}^{\partial}}(\mathbb{I}_{d}, k)\big\}$ by \eqref{eq:alpha-dof-vertex}.

    For each $t$-dimensional ($1 \le t \le d - 2$) subsimplex $F$ of $K$, since $\Sigma_{F, n}^{\bm{r}^{\partial}}(\mathbb{I}_{d}, k) \to \Sigma(\mathbb{I}_{d, \setminus t}, n) \times \Sigma_{0}^{\bm{q}^{\partial}_{t, n}}(\mathbb{I}_{F}, k - n)$ is a bijection, $\bm{q}^{\partial}_{t, n} = \bm{q}_{t, n}$ and $r^{\partial}_{d - t} = r_{d - t}$, the degrees of freedom \eqref{eq:set-dof-2} defined on $F$ can be expressed as $\big\{\varphi_{\bm{\alpha}}: \bm{\alpha} \in \bigcup_{n = 0}^{r^{\partial}_{d - t}} \Sigma_{F, n}^{\bm{r}^{\partial}}(\mathbb{I}_{d}, k)\big\}$ by \eqref{eq:alpha-dof}.

    For each $(d - 1)$-dimensional subsimplex $F$ of $K$, since $\Sigma_{F, n}^{\bm{r}^{\partial}}(\mathbb{I}_{d}, k) \to \Sigma(\{d\}, n) \times \Sigma_{0}^{\bm{q}^{\partial}_{d - 1, n}}(\mathbb{I}_{F}, k - n)$ and $\Sigma_{F, n - b}^{\bm{r}^{\circ}}(\mathbb{I}_{d}, \rho) \to \Sigma(\{d\}, n) \times \Sigma_{0}^{\bm{q}_{d - 1, n}}(\mathbb{I}_{F}, k - n)$ are bijections, $\bm{q}^{\partial}_{d - 1, n} = \bm{q}_{d - 1, n}$ and $r^{\partial}_{1} = b - 1$, the degrees of freedom \eqref{eq:set-dof-2} defined on $F$ can be expressed as $\big\{\varphi_{\bm{\alpha}}: \bm{\alpha} \in \bigcup_{n = 0}^{r^{\partial}_{1}} \Sigma_{F, n}^{\bm{r}^{\partial}}(\mathbb{I}_{d}, k)\big\} \cup \big\{\psi_{\bm{\beta}}: \bm{\beta} \in \bigcup_{n = b}^{r_{1}} \Sigma_{F, n - b}^{\bm{r}^{\circ}}(\mathbb{I}_{d}, \rho)\big\}$ by \eqref{eq:alpha-dof-d-1} and \eqref{eq:beta-dof-d-1}.

    For the $d$-dimensional simplex $K$, the degrees of freedom \eqref{eq:set-dof-3} defined on $K$ can be expressed as $\big\{\psi_{\bm{\beta}}: \bm{\beta} \in \Sigma_{0}^{\bm{r}^{\circ}}(\mathbb{I}_{d}, \rho)\big\}$ by \eqref{eq:beta-dof-d}.
\end{proof}

The decompositions in \eqref{eq:alpha-set} and \eqref{eq:beta-set} are illustrated in Figure~\ref{fig:dofs-2d} for $d = 2$ and Figure~\ref{fig:dofs-3d} for $d = 3$. In the first case, we have $\bm{r} = (3, 4)$, $\rho = 7$, $k = 9$, and $b = 2$. In the second case, we have $\bm{r} = (3, 4, 8)$, $\rho = 15$, $k = 17$, and $b = 2$.

\begin{figure}[htp]
    \centering
    \includegraphics[width=0.8\textwidth]{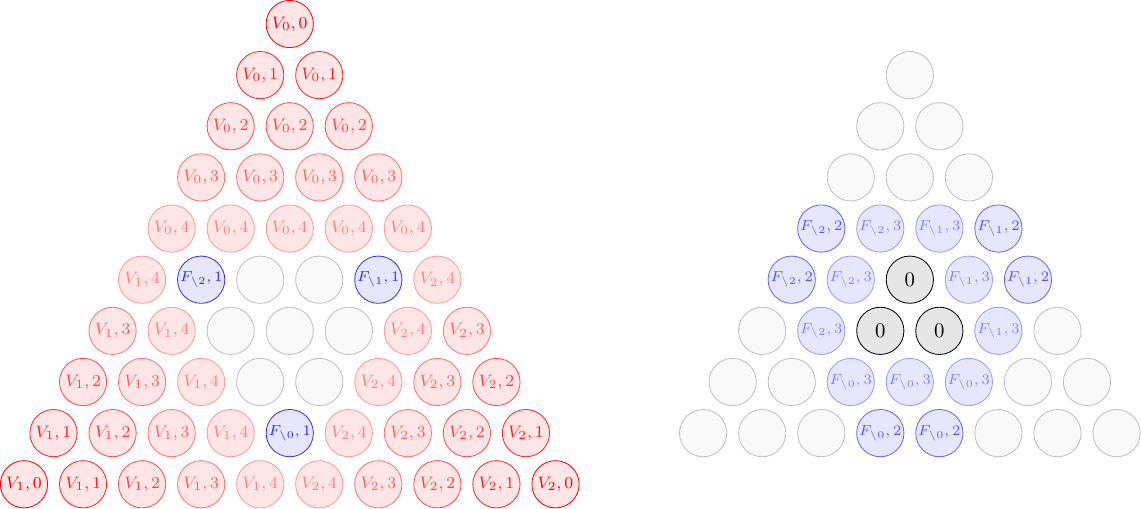}
    \caption{Illustration of the decomposition of $\Sigma^{\partial}$ in \eqref{eq:alpha-set} (left) and of $\Sigma^{\circ}$ in \eqref{eq:beta-set} (right) for $d = 2$, $\bm{r} = (3, 4)$, $\rho = 7$, $k = 9$, and $b = 2$. The multi-indices related to a vertex are visualized as red disks, the multi-indices related to an edge as blue disks, and the remaining multi-indices as gray disks. Different shades of a color indicate different layers of multi-indices.} 
    \label{fig:dofs-2d}
\end{figure}

\begin{figure}[htp]
    \centering
    \includegraphics[width=0.8\textwidth]{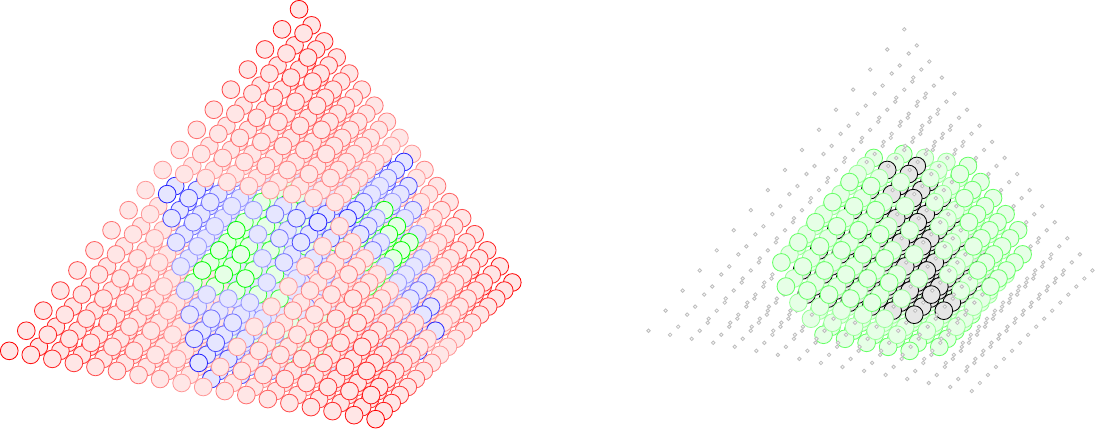}
    \caption{Illustration of the decomposition of $\Sigma^{\partial}$ in \eqref{eq:alpha-set} (left) and of $\Sigma^{\circ}$ in \eqref{eq:beta-set} (right) for $d = 3$, $\bm{r} = (3, 4, 8)$, $\rho = 15$, $k = 17$, and $b = 2$. The multi-indices related to a vertex are visualized as red balls, the multi-indices related to an edge as blue balls, the multi-indices related to a facet as green balls, and the remaining multi-indices as gray balls. Different shades of a color indicate different layers of multi-indices.}
    \label{fig:dofs-3d}
\end{figure}

\section{Properties of \texorpdfstring{$Q$}{Q}-Operators}
\label{sec:properties}

This section discusses the properties of $Q$-operators defined in Definition~\ref{def:Q-op}. First, we give a characterization of the range of the $Q$-operator.

Given the polynomial degree $k$ and the split point continuity $\rho$ such that $\rho \le k - 1$, set the boundary layer number $b \coloneqq k - \rho$ and define the piecewise polynomial space $\mathcal{Q}_{k}^{\rho}(\mathcal{T}_{A}(K))$ on the Alfeld split $\mathcal{T}_{A}(K)$ as
\begin{equation}
    \label{eq:Q-k-b-space}
    \begin{aligned}
        \mathcal{Q}_{k}^{\rho}(\mathcal{T}_{A}(K)) \coloneqq \big\{u \in L^{2}(K): \; & u|_{K_{j}} \in \mathcal{P}_{k}(K_{j}) \text{ for each } j \in \mathbb{I}_{d}; \\
        & \nabla^{n} (u|_{K_{j}})|_{F_{\setminus j}} = 0 \text{ for each } j \in \mathbb{I}_{d} \text{ and } 0 \le n \le b - 1, \\
        & \text{where } F_{\setminus j} \text{ is the } (d - 1) \text{-dimensional subsimplex of } K \text{ and } K_{j}; \\
        & \nabla^{n} u \text{ is single-valued at } V_{A} \text{ for } 0 \le n \le \rho\big\}.
    \end{aligned}
\end{equation}

For each $j \in \mathbb{I}_{d}$ and each piecewise polynomial function $u \in \mathcal{Q}_{k}^{\rho}(\mathcal{T}_{A}(K))$, we have $u|_{K_{j}} \in \mathcal{P}_{k}(K_{j})$ and $\nabla^{n} (u|_{K_{j}})|_{F_{\setminus j}} = 0$ for $0 \le n \le b - 1$. Therefore, there exists a bijection between the restriction space $\{u|_{K_{j}}: u \in \mathcal{Q}_{k}^{\rho}(\mathcal{T}_{A}(K))\}$ and the polynomial space $\mathcal{P}_{k - b}(K_{j}) = \mathcal{P}_{\rho}(K_{j})$. Moreover, since $\nabla^{n} u$ is single-valued at $V_{A}$ for $0 \le n \le \rho$, a piecewise polynomial function $u \in \mathcal{Q}_{k}^{\rho}(\mathcal{T}_{A}(K))$ is uniquely determined by its restriction $u|_{K_{j}}$, which is related to a polynomial in $\mathcal{P}_{\rho}(K_{j})$.

The following proposition shows that the operator $Q_{k}^{\rho}$ in Definition~\ref{def:Q-op} is a bijection between $\mathcal{P}_{\rho}(K)$ and $\mathcal{Q}_{k}^{\rho}(\mathcal{T}_{A}(K))$.

\begin{proposition}
    \label{prop:Q-k-b-bijection}
    Given the polynomial degree $k$ and the split point continuity $\rho$ such that $\rho \le k - 1$, $Q_{k}^{\rho}$ is a linear bijection operator from $\mathcal{P}_{\rho}(K)$ to $\mathcal{Q}_{k}^{\rho}(\mathcal{T}_{A}(K))$ defined in \eqref{eq:Q-k-b-space}.
\end{proposition}

\begin{proof}
    For each $d$-dimensional simplex $K_{j} \in \mathcal{T}_{A}(K)$, $j \in \mathbb{I}_{d}$, consider the linearly independent vectors $\bm{\tau}_{i}$ with $i \in \mathbb{I}_{\setminus j}$, where $\bm{\tau}_{i}$ is directed from $V_{A}$ to $V_{i}$. Then, it holds that $\frac{\partial}{\partial \bm{\tau}_{i}} \lambda_{j, j} = - 1$ and $\frac{\partial}{\partial \bm{\tau}_{i}} \lambda_{j, i'} = \delta_{i, i'}$ for $i, i' \in \mathbb{I}_{\setminus j}$ with Kronecker’s delta $\delta$.    
    For each $0 \le n \le \rho$, each multi-index $\bm{\gamma} \in \Sigma(\mathbb{I}_{\setminus j}, n)$, and each multi-index $\bm{\beta} \in \Sigma(\mathbb{I}_{d}, \rho)$, it holds that
    \begin{equation*}
        \begin{aligned}
            \bigg(\prod_{i \in \mathbb{I}_{\setminus j}} \frac{\partial^{\gamma_{i}}}{\partial \bm{\tau}_{i}^{\gamma_{i}}}\bigg) \llbracket\bm{\lambda}\rrbracket_{j}^{\bm{\beta}} \bigg|_{V_{A}} & = \bigg(\prod_{i \in \mathbb{I}_{\setminus j}} \frac{\partial^{\gamma_{i}}}{\partial \bm{\tau}_{i}^{\gamma_{i}}}\bigg) \Bigg(\frac{\lambda_{j, j}^{\beta_{j}}}{\beta_{j}!} \prod_{i \in \mathbb{I}_{\setminus j}} \frac{\lambda_{j, i}^{\beta_{i}}}{\beta_{i}!}\Bigg)\Bigg|_{V_{A}} \\
            & = \frac{(- 1)^{\beta_{j} - (\rho - n)}}{(\rho - n)!} \prod_{i \in \mathbb{I}_{\setminus j}} \binom{\gamma_{i}}{\beta_{i}}.
        \end{aligned}
    \end{equation*}
    Note that $\binom{\gamma_{i}}{\beta_{i}} = 0$ if $\gamma_{i} < \beta_{i}$.
    From the definition of $Q_{k}^{\rho}$ in \eqref{eq:Q-k-b-operator}, it follows that
    \begin{equation*}
        \begin{aligned}
            \bigg(\prod_{i \in \mathbb{I}_{\setminus j}} \frac{\partial^{\gamma_{i}}}{\partial \bm{\tau}_{i}^{\gamma_{i}}}\bigg) \big(Q_{k}^{\rho}(\llbracket\bm{\lambda}\rrbracket_{j}^{\bm{\beta}})|_{K_{j}}\big)\bigg|_{V_{A}} & = \bigg(\prod_{i \in \mathbb{I}_{\setminus j}} \frac{\partial^{\gamma_{i}}}{\partial \bm{\tau}_{i}^{\gamma_{i}}}\bigg) \Bigg(\frac{\lambda_{j, j}^{\beta_{j} + b}}{(\beta_{j} +b)!} \prod_{i \in \mathbb{I}_{\setminus j}} \frac{\lambda_{j, i}^{\beta_{i}}}{\beta_{i}!}\Bigg)\Bigg|_{V_{A}} \\
            & = \frac{(- 1)^{\beta_{j} + b - (k - n)}}{(k - n)!} \prod_{i \in \mathbb{I}_{\setminus j}} \binom{\gamma_{i}}{\beta_{i}} \\
            & = \frac{(- 1)^{\beta_{j} - (\rho - n)}}{(k - n)!} \prod_{i \in \mathbb{I}_{\setminus j}} \binom{\gamma_{i}}{\beta_{i}},
        \end{aligned}
    \end{equation*}
    which implies that for each $0 \le n \le \rho$ and each $p \in \mathcal{P}_{\rho}(K)$,
    \begin{equation}
        \label{eq:Q-k-b-derivative}
        \nabla^{n} \big(Q_{k}^{\rho}(p)|_{K_{j}}\big)\big|_{V_{A}} = \frac{(\rho - n)!}{(k - n)!} \nabla^{n} p|_{V_{A}}.
    \end{equation}
    From \eqref{eq:Q-k-b-operator} it is also easy to see that for each $0 \le n \le b - 1$ and each $p \in \mathcal{P}_{\rho}(K)$,
    \begin{equation*}
        \nabla^{n} \big(Q_{k}^{\rho}(p)|_{K_{j}}\big)\big|_{F_{\setminus j}} = 0.
    \end{equation*}
    Therefore, $Q_{k}^{\rho}(p) \in \mathcal{Q}_{k}^{\rho}(\mathcal{T}_{A}(K))$ for each $p \in \mathcal{P}_{\rho}(K)$, and $Q_{k}^{\rho}$ is injective. 
    
    Moreover, for each $u \in \mathcal{Q}_{k}^{\rho}(\mathcal{T}_{A}(K))$ and each $j \in \mathbb{I}_{d}$, since $u|_{K_{j}} \in \mathcal{P}_{k}(K_{j})$ and $\nabla^{n} (u|_{K_{j}})|_{F_{\setminus j}} = 0$ holds for $0 \le n \le b - 1$, $u|_{K_{j}}$ can be expressed as $u|_{K_{j}} = \sum_{|\bm{\beta}| = \rho} c_{\bm{\beta}} \llbracket\bm{\lambda}\rrbracket_{j}^{\bm{\beta} + b \bm{e}_{j}}$ by the following Lemma~\ref{lem:continuity}. Therefore, by Definition~\ref{def:Q-op}, the polynomial $p_{j} := \sum_{|\bm{\beta}| = \rho} c_{\bm{\beta}} \llbracket\bm{\lambda}\rrbracket_{j}^{\bm{\beta}} \in \mathcal{P}_{\rho}(K)$ satisfies $Q_{k}^{\rho}(p_{j})|_{K_{j}} = u|_{K_{j}}$. For $j, j' \in \mathbb{I}_{d}$, it follows from \eqref{eq:Q-k-b-derivative} that $\nabla^{n} p_{j}|_{V_{A}} = \nabla^{n} p_{j'}|_{V_{A}}$ for $0 \le n \le \rho$. It implies that there exists a $p \in \mathcal{P}_{\rho}(K)$, such that $p_{j} = p$ for $j \in \mathbb{I}_{d}$, and $Q_{k}^{\rho}$ is surjective.
\end{proof}

The next lemma is proved in \cite{de1986b}.

\begin{lemma}
    \label{lem:continuity}
    Given a $t$-dimensional ($0 \le t \le d - 1$) subsimplex $F$ of $K$ and the polynomial $p \in \mathcal{P}_{l}(K)$, then $F$ is a subsimplex of $K_{j}$ for $j \in \mathbb{I}_{\setminus F}$. Consider the barycentric coordinate representation with respect to $K$ of $p$ such that
    \begin{equation*}
        p = \sum_{|\bm{\alpha}| = l} c_{\bm{\alpha}} \llbracket\bm{\lambda}\rrbracket^{\bm{\alpha}},
    \end{equation*}
    and the barycentric coordinate representation with respect to $K_{j}$ of $p$ such that
    \begin{equation*}
        p = \sum_{|\bm{\alpha}| = l} c_{j, \bm{\alpha}} \llbracket\bm{\lambda}\rrbracket_{j}^{\bm{\alpha}}.
    \end{equation*}
    Given $q \ge - 1$, the following conditions are equivalent:
    \begin{enumerate}
        \item $\nabla^{n} p|_{F} = 0$ for any $0 \le n \le q$;
        \item $c_{\bm{\alpha}} = 0$ for any $0 \le |\bm{\alpha}|_{\setminus F} \le q$;
        \item $c_{j, \bm{\alpha}} = 0$ for any $0 \le |\bm{\alpha}|_{j, \setminus F} \le q$.
    \end{enumerate}
    Moreover, given $q \ge - 1$, the following conditions are equivalent: 
    \begin{enumerate}
        \item $\nabla^{n} p|_{V_{A}} = 0$ for any $0 \le n \le q$;
        \item $c_{j, \bm{\alpha}} = 0$ for any $0 \le |\bm{\alpha}|_{j, \setminus V_{A}} \le q$.
    \end{enumerate}
\end{lemma}

With the help of the previous lemma, we obtain the following result.

\begin{lemma}
    \label{lem:continuity-equivalent}
    Given a $t$-dimensional ($0 \le t \le d - 1$) subsimplex $F$ of $K$, the polynomial $p \in \mathcal{P}_{\rho}(K)$, and $q \ge - 1$, the following conditions are equivalent:
    \begin{enumerate}
        \item $\nabla^{n} p|_{F} = 0$ for any $0 \le n \le q$;
        \item $\nabla^{n} (Q_{k}^{\rho}(p))|_{F} = 0$ for any $0 \le n \le q + b$. 
    \end{enumerate}
\end{lemma}

\begin{proof}
    Note that the polynomial $p \in \mathcal{P}_{\rho}(K)$ satisfies $\nabla^{n} p|_{F} = 0$, if and only if $\nabla^{n} p|_{E} = 0$ for any $t'$-dimensional ($0 \le t' \le t$) subsimplex $E$ of $F$. Similarly, the piecewise polynomial $Q_{k}^{\rho}(p)$ satisfies $\nabla^{n} (Q_{k}^{\rho}(p))|_{F} = 0$, if and only if $\nabla^{n} (Q_{k}^{\rho}(p)|_{K_{j}})|_{E} = 0$ for any $t'$-dimensional ($0 \le t' \le t$) subsimplex $E$ of $F$, and any $K_{j}$, $j \in \mathbb{I}_{d}$ such that $E$ is a subsimplex of $K_{j}$.
    
    For each $t'$-dimensional ($0 \le t' \le t$) subsimplex $E$ of $F$, and each $K_{j}$, $j \in \mathbb{I}_{d}$ such that $E$ is a subsimplex of $K_{j}$, consider the barycentric coordinate representation with respect to $K_{j}$ of the polynomial $p \in \mathcal{P}_{\rho}(K)$ such that
    \begin{equation*}
        p = \sum_{|\bm{\beta}| = \rho} c_{\bm{\beta}} \llbracket\bm{\lambda}\rrbracket_{j}^{\bm{\beta}}.
    \end{equation*}
    Here the summation ranges over multi-indices $\bm{\beta} \coloneqq (\beta_{0}, \beta_{1}, \ldots, \beta_{d}) \in \mathbb{N}_{0}^{d + 1}$ such that $|\bm{\beta}| \coloneqq \sum_{i \in \mathbb{I}_{d}} \beta_{i} = \rho$. 
    From Lemma~\ref{lem:continuity}, it holds that $\nabla^{n} p|_{E} = 0$ for any $0 \le n \le q$, if and only if $c_{\bm{\beta}} = 0$ for any $|\bm{\beta}|_{j, \setminus E} \le q$.  Furthermore, consider the barycentric coordinate representation with respect to $K_{j}$ of the polynomial $Q_{k}^{\rho}(p)|_{K_{j}} \in \mathcal{P}_{k}(K_{j})$ such that
    \begin{equation*}
        Q_{k}^{\rho}(p)|_{K_{j}} = \sum_{|\bm{\alpha}| = k} c'_{\bm{\alpha}} \llbracket\bm{\lambda}\rrbracket_{j}^{\bm{\alpha}}.
    \end{equation*}
    Here the summation ranges over multi-indices $\bm{\alpha} \coloneqq (\alpha_{0}, \alpha_{1}, \ldots, \alpha_{d}) \in \mathbb{N}_{0}^{d + 1}$ such that $|\bm{\alpha}| \coloneqq \sum_{i \in \mathbb{I}_{d}} \alpha_{i} = k$. From the definition of $Q_{k}^{\rho}$ in \eqref{eq:Q-k-b-operator}, it follows that $c'_{\bm{\alpha}} = 0$ if $\alpha_{j} \le b - 1$, and $c'_{\bm{\alpha}} =c_{\bm{\beta}}$ for $\bm{\beta} = (\beta_{0}, \beta_{1}, \ldots, \beta_{d})$ such that $\alpha_{i} = \beta_{i} + b \cdot \delta_{j, i}$, with Kronecker’s delta $\delta$. Since $V_{A}$ is not a vertex of $E$, it holds that $j \in \mathbb{I}_{j, \setminus E}$, and $|\bm{\alpha}|_{j, \setminus E} = |\bm{\beta}|_{j, \setminus E} + b$. It yields that $c_{\bm{\beta}} = 0$ for any $|\bm{\beta}|_{j, \setminus E} \le q$, if and only if $c'_{\bm{\alpha}} = 0$ for any $|\bm{\alpha}|_{j, \setminus E} \le q + b$, which is equivalent to $\nabla^{n} (Q_{k}^{\rho}(p)|_{K_{j}})|_{E} = 0$ for any $0 \le n \le q + b$. 

    In conclusion, the polynomial $p \in \mathcal{P}_{\rho}(K)$ satisfies $\nabla^{n} p|_{F} = 0$ for any $0 \le n \le q$, if and only if the piecewise polynomial $Q_{k}^{\rho}(p)$ satisfies $\nabla^{n} (Q_{k}^{\rho}(p))|_{F} = 0$ for any $0 \le n \le q + b$.
\end{proof}

Several corollaries follow from Lemma~\ref{lem:continuity-equivalent}. Among them, Corollary~\ref{coro:Q-alpha-beta} will be used to characterize the space $\mathcal{Q}^{\circ}$ defined below in \eqref{eq:Q-beta}. Two other corollaries, Corollaries~\ref{coro:Q-beta-beta} and~\ref{coro:Q-beta-beta-bubble}, will be presented in Section~\ref{sec:unisolvence} and will be used to establish the unisolvence.

\begin{corollary}
    \label{coro:Q-alpha-beta}
        Suppose that $\bm{r} \coloneqq (r_{1}, \ldots, r_{d})$, $\rho$, $k$, and $b$ jointly satisfy Condition~\ref{cond:Cr-macro-element}. Consider any $\bm{\beta} \in \Sigma^{\circ}$. For each $t$-dimensional ($0 \le t \le d - 1$) subsimplex $F$ of $K$, it holds that $\nabla^{n} Q_{k}^{\rho}(\llbracket\bm{\lambda}\rrbracket^{\bm{\beta}})|_{F} = 0$ for $0 \le n \le r^{\partial}_{d - t}$.
\end{corollary}

\begin{proof}
    For each $(d - 1)$-dimensional subsimplex $F$ of $K$, from the definition of $Q_{k}^{\rho}$ in \eqref{eq:Q-k-b-operator} and Lemma~\ref{lem:continuity}, it holds that $\nabla^{n} Q_{k}^{\rho}(\llbracket\bm{\lambda}\rrbracket^{\bm{\beta}})|_{F} = 0$ for $0 \le n \le b - 1 = r^{\partial}_{1}$.
    
    For each $t$-dimensional ($0 \le t \le d - 2$) subsimplex $F$ of $K$, from \eqref{eq:Sigma-F-n-b} and \eqref{eq:Sigma-0-d}, it holds that $|\bm{\beta}|_{\setminus F} \ge r^{\circ}_{d - t} + 1 = r_{d - t} - b + 1$. Hence, it follows that $\nabla^{n} \llbracket\bm{\lambda}\rrbracket^{\bm{\beta}}|_{F} = 0$ for $0 \le n \le r_{d - t} - b = r^{\partial}_{d - t} - b$, which completes the proof by Lemma~\ref{lem:continuity-equivalent}.
\end{proof}

Now, we consider a specific subspace of $\mathcal{Q}_{k}^{\rho}(\mathcal{T}_{A}(K))$, corresponding to $\Sigma^{\circ}$, $\bm{r}^{\circ}$, and $\rho$ defined in \eqref{eq:beta-set}, such that
\begin{equation}
    \label{eq:Q-beta}
    \mathcal{Q}^{\circ} \coloneqq \Span \big\{Q_{k}^{\rho}(\llbracket\bm{\lambda}\rrbracket^{\bm{\beta}}): \bm{\beta} \in \Sigma^{\circ}\big\}.
\end{equation}
The subspace $\mathcal{Q}^{\circ}$ is introduced to characterize the piecewise components of the superspline space $\mathcal{S}_{k}^{\bm{r}, \rho}(\mathcal{T}_{A}(K))$, rather than the globally polynomial part. The following proposition shows that 
if a piecewise polynomial in $\mathcal{Q}_{k}^{\rho}(\mathcal{T}_{A}(K))$ satisfies certain boundary conditions, then it necessarily belongs to $\mathcal{Q}^{\circ}$.

\begin{proposition}
    \label{prop:boundary-condition}
        Suppose that $\bm{r}$, $\rho$, $k$, and $b$ jointly satisfy Condition~\ref{cond:Cr-macro-element}. 
        If the piecewise polynomial $u \in \mathcal{Q}_{k}^{\rho}(\mathcal{T}_{A}(K))$ satisfies that for each $t$-dimensional ($0 \le t \le d - 1$) subsimplex $F$ of $K$, $\nabla^{n} u|_{F} = 0$ for each $0 \le n \le r^{\partial}_{d - t}$, then it holds that the piecewise polynomial $u \in \mathcal{Q}^{\circ}$.
\end{proposition}

\begin{proof}
    From Proposition~\ref{prop:Q-k-b-bijection}, there exists a polynomial $p \in \mathcal{P}_{\rho}(K)$ such that $u = Q_{k}^{\rho}(p)$, which satisfies that for each $t$-dimensional ($0 \le t \le d - 1$) subsimplex $F$ of $K$, $\nabla^{n} Q_{k}^{\rho}(p)|_{F} = \nabla^{n} u|_{F} = 0$ for $0 \le n \le r^{\partial}_{d - t}$. From Lemma~\ref{lem:continuity-equivalent}, for each $t$-dimensional ($0 \le t \le d - 2$) subsimplex $F$ of $K$, it holds that $\nabla^{n} p|_{F} = 0$ for $0 \le n \le r^{\partial}_{d - t} - b = r_{d - t} - b = r^{\circ}_{d - t}$. Consider the barycentric coordinate representation with respect to $K$ of the polynomial $p \in \mathcal{P}_{\rho}(K)$ such that
    \begin{equation*}
        p = \sum_{|\bm{\beta}| = \rho} c_{\bm{\beta}} \llbracket\bm{\lambda}\rrbracket^{\bm{\beta}}.
    \end{equation*}
    Here the summation ranges over multi-indices $\bm{\beta} \coloneqq (\beta_{0}, \beta_{1}, \ldots, \beta_{d}) \in \mathbb{N}_{0}^{d + 1}$ such that $|\bm{\beta}| \coloneqq \sum_{i \in \mathbb{I}_{d}} \beta_{i} = \rho$. From Lemma~\ref{lem:continuity}, for each $t$-dimensional ($0 \le t \le d - 2$) subsimplex $F$ of $K$, it holds that $c_{\bm{\beta}} = 0$ for each $|\bm{\beta}|_{\setminus F} \le r^{\circ}_{d - t}$.
    
    By the definition \eqref{eq:beta-set} of $\Sigma^{\circ}$, the multi-index $\bm{\beta} \in \Sigma^{\circ}$ if and only if for each $t$-dimensional ($0 \le t \le d - 2$) subsimplex $F$ of $K$, it holds that $|\bm{\beta}|_{\setminus F} \ge r^{\circ}_{d - t} + 1$. Therefore, it follows that $p = \sum_{\bm{\beta} \in \Sigma^{\circ}} c_{\bm{\beta}} \llbracket\bm{\lambda}\rrbracket^{\bm{\beta}}$.

    In conclusion, it holds that $u = Q_{k}^{\rho}(p) = Q_{k}^{\rho}\big(\sum_{\bm{\beta} \in \Sigma^{\circ}} c_{\bm{\beta}} \llbracket\bm{\lambda}\rrbracket^{\bm{\beta}}\big) = \sum_{\bm{\beta} \in \Sigma^{\circ}} c_{\bm{\beta}} Q_{k}^{\rho}(\llbracket\bm{\lambda}\rrbracket^{\bm{\beta}}) \in \mathcal{Q}^{\circ}$.
\end{proof}

\begin{proposition}
    \label{prop:subspace-Q-S}
        Suppose that $\bm{r}$, $\rho$, $k$, and $b$ jointly satisfy Condition~\ref{cond:Cr-macro-element}. 
        The piecewise polynomial space $\mathcal{Q}^{\circ}$ is a subspace of $\mathcal{S}_{k}^{\bm{r}, \rho}(\mathcal{T}_{A}(K))$ defined in \eqref{eq:shape-function},
    \begin{equation*}
        \mathcal{Q}^{\circ} \subseteq \mathcal{S}_{k}^{\bm{r}, \rho}(\mathcal{T}_{A}(K)).
    \end{equation*}
\end{proposition}

\begin{proof}
    For each piecewise polynomial $u \in \mathcal{Q}^{\circ}$ and each $j \in \mathbb{I}_{d}$, define the polynomial $u_{j} \in \mathcal{P}_{k}(K)$ such that $u_{j}|_{K_{j}} = u|_{K_{j}}$. 
    
    For each $d$-dimensional simplex $K_{j} \in \mathcal{T}_{A}(K)$, $j \in \mathbb{I}_{d}$, and each $d$-dimensional simplex $K_{j'} \in \mathcal{T}_{A}(K)$ such that $j' \in \mathbb{I}_{\setminus j}$, let $F$ be the common $(d - 1)$-dimensional subsimplex of $K_{j}$ and $K_{j'}$, such that $V_{A}$ is a vertex of $F$. Define $E$ as the $(d - 2)$-dimensional subsimplex of $F$ opposite to the vertex $V_{A}$. Then, from Corollary~\ref{coro:Q-alpha-beta}, it holds that for $0 \le n \le r^{\partial}_{2} = r_{2}$, $\nabla^{n} u|_{E} = \nabla^{n} u_{j}|_{E} = \nabla^{n} u_{j'}|_{E} = 0$. Moreover, from Proposition~\ref{prop:Q-k-b-bijection}, it follows that for $0 \le n \le \rho$, $\nabla^{n} u|_{V_{A}} = \nabla^{n} u_{j}|_{V_{A}} = \nabla^{n} u_{j'}|_{V_{A}}$.

    Let $u_{j, j'} = u_{j} - u_{j'}$, and consider the barycentric coordinate representation with respect to $K_{j}$ of the polynomial $u_{j, j'}$ such that
    \begin{equation*}
        u_{j, j'} = \sum_{|\bm{\alpha}| = k} c_{\bm{\alpha}} \llbracket\bm{\lambda}\rrbracket_{j}^{\bm{\alpha}}.
    \end{equation*}
    Here the summation ranges over multi-indices $\bm{\alpha} \coloneqq (\alpha_{0}, \alpha_{1}, \ldots, \alpha_{d}) \in \mathbb{N}_{0}^{d + 1}$ such that $|\bm{\alpha}| \coloneqq \sum_{i \in \mathbb{I}_{d}} \alpha_{i} = k$. Since for $0 \le n \le r_{2}$, $\nabla^{n} u_{j, j'}|_{E} = 0$, and for $0 \le n \le \rho$, $\nabla^{n} u_{j, j'}|_{V_{A}} = 0$, we have $c_{\bm{\alpha}} = 0$ for each $\bm{\alpha}$ such that $|\bm{\alpha}|_{j, \setminus E} \le r_{2}$ or $|\bm{\alpha}|_{j, \setminus V_{A}} \le \rho$.

    For each $\bm{\alpha}$ such that $\alpha_{j'} \le r_{2} - b + 1$, if $\alpha_{j} \le b - 1$, then it holds that $|\bm{\alpha}|_{j, \setminus E} = \alpha_{j} + \alpha_{j'} \le r_{2}$, and if $\alpha_{j} \ge b$, then it follows that $|\bm{\alpha}|_{j, \setminus V_{A}} = k - \alpha_{j} \le k - b = \rho$. Therefore, for each $\bm{\alpha}$ such that $|\bm{\alpha}|_{j, \setminus F} \coloneqq \alpha_{j'} \le r_{2} - b + 1$, it holds that $c_{\bm{\alpha}} = 0$, which implies that for $0 \le n \le r_{2} - b + 1$, $\nabla^{n} u_{j, j'}|_{F} = 0$, and $\nabla^{n} u_{j}|_{F} = \nabla^{n} u_{j'}|_{F}$. By using $r_{1} \le r_{2} - b + 1$, it follows that $u|_{K_{j} \cup K_{j'}} \in C^{r_{1}}(K_{j} \cup K_{j'})$. Since $j$ and $j'$ are arbitrary, we additionally deduce that $u \in C^{r_{1}}(K)$.

    In conclusion, it holds that $\mathcal{Q}^{\circ} \subseteq C^{r_{1}}(K)$ and, more precisely, $\mathcal{Q}^{\circ} \subseteq \mathcal{S}_{k}^{\bm{r}, \rho}(\mathcal{T}_{A}(K))$.
\end{proof}

\begin{remark}
    Similarly, for each $u \in \mathcal{Q}^{\circ}$, it can be shown that for each $t$-dimensional $(1 \le t \le d - 1)$ subsimplex $F$ in $\mathcal{T}_{A}(K)$ such that $V_{A}$ is a vertex of $F$, $\nabla^{n} u|_{F}$ is single-valued for $0 \le n \le r_{d - t + 1} - b + 1$.
    This can be seen as follows.
    Let $F$ be a common $t$-dimensional ($1 \le t \le d - 1$) subsimplex of both $K_{j}$ and $K_{j'}$, such that $V_{A}$ is a vertex of $F$. Define $E$ as the $(t - 1)$-dimensional subsimplex of $F$ opposite to the vertex $V_{A}$. Then, it holds that $c_{\bm{\alpha}} = 0$ for each $\bm{\alpha}$ such that $|\bm{\alpha}|_{j, \setminus E} \le r_{d - t + 1}$ or $|\bm{\alpha}|_{j, \setminus V_{A}} \le \rho$. For each $\bm{\alpha}$ such that $|\bm{\alpha}|_{j, \setminus F} \le r_{d - t + 1} - b + 1$, if $\alpha_{j} \le b - 1$, then it holds that $|\bm{\alpha}|_{j, \setminus E} = \alpha_{j} + |\bm{\alpha}|_{j, \setminus F} \le r_{d - t + 1}$, and if $\alpha_{j} \ge b$, then it follows that $|\bm{\alpha}|_{j, \setminus V_{A}} = k - \alpha_{j} \le k - b = \rho$. Therefore, for each $\bm{\alpha}$ such that $|\bm{\alpha}|_{j, \setminus F} \le r_{d - t + 1} - b + 1$, it holds that $c_{\bm{\alpha}} = 0$, which implies that for $0 \le n \le r_{d - t + 1} - b + 1$, $\nabla^{n} u_{j}|_{F} = \nabla^{n} u_{j'}|_{F}$.
\end{remark}

\section{Proof of Unisolvence}
\label{sec:unisolvence}

In this section, we show that the degrees of freedom expressed in \eqref{eq:equivalent-set-dofs} are linearly independent and we provide a basis for the shape function space $\mathcal{S}_{k}^{\bm{r}, \rho}(\mathcal{T}_{A}(K))$.

\begin{theorem}
    \label{thm:unsolvence}
    Suppose that the continuity vector $\bm{r} \coloneqq (r_{1}, \ldots, r_{d}) \in \mathbb{N}_{0}^{d}$, the split point continuity $\rho \in \mathbb{N}_{0}$, the polynomial degree $k \in \mathbb{N}_{0}$, and the boundary layer number $b \in \mathbb{N}_{0}$ jointly satisfy Condition~\ref{cond:Cr-macro-element}. 
    Let $\varphi_{\bm{\alpha}}$ and $\psi_{\bm{\beta}}$ be the degrees of freedom as in Proposition~\ref{prop:equivalent-set-dofs}.
    For $u \in \mathcal{S}_{k}^{\bm{r}, \rho}(\mathcal{T}_{A}(K))$, if $\varphi_{\bm{\alpha}}(u) = 0$ for all $\bm{\alpha} \in \Sigma^{\partial}$ and $\psi_{\bm{\beta}}(u) = 0$ for all $\bm{\beta} \in \Sigma^{\circ}$, then $u = 0$.
    Moreover, the shape function space $\mathcal{S}_{k}^{\bm{r}, \rho}(\mathcal{T}_{A}(K))$ has a basis
    \begin{equation}\label{eq:basis}
        \big\{\llbracket\bm{\lambda}\rrbracket^{\bm{\alpha}}: \bm{\alpha} \in \Sigma^{\partial}\big\} \cup \big\{Q_{k}^{\rho}(\llbracket\bm{\lambda}\rrbracket^{\bm{\beta}}): \bm{\beta} \in \Sigma^{\circ}\big\}.
    \end{equation}
\end{theorem}

The remainder of this section is devoted to the proof of Theorem~\ref{thm:unsolvence}. Note that the basis in \eqref{eq:basis} consists of a set of polynomials (related to the multi-index set $\Sigma^{\partial}$), enriched with a set of splines, i.e., piecewise polynomials (related to the multi-index set $\Sigma^{\circ}$).
There is a similarity with the basis construction in \cite{lyche2025} for $C^1$ splines on the Alfeld split in $d$ dimensions, where a subset of Bernstein basis polynomials (related to the boundary) is enriched with so-called simplex splines (in the interior).

The following corollaries follow directly from Lemma~\ref{lem:continuity-equivalent} and show that certain derivatives of $Q_{k}^{\rho}(\llbracket\bm{\lambda}\rrbracket^{\bm{\beta}})$ vanish on the subsimplices $F$ of $K$ for $\bm{\beta} \in \Sigma^{\circ}$.

\begin{corollary}
    \label{coro:Q-beta-beta}
    Suppose that $\bm{r} \coloneqq (r_{1}, \ldots, r_{d})$, $\rho$, $k$, and $b$ jointly satisfy Condition~\ref{cond:Cr-macro-element}. 
    Consider any $(d - 1)$-dimensional subsimplex $F_{\setminus j}$ of $K$, any $n$ such that $b \le n \le r_{1}$, and any $\bm{\beta} \in \Sigma_{F_{\setminus j}, n - b}^{\bm{r}^{\circ}}(\mathbb{I}_{d}, \rho)$. 
    \begin{enumerate}
        \item  It holds that $\nabla^{n'} Q_{k}^{\rho}(\llbracket\bm{\lambda}\rrbracket^{\bm{\beta}})|_{F_{\setminus j}} = 0$ for $0 \le n' \le n - 1$. 
        \item Additionally, for each $(d - 1)$-dimensional subsimplex $F_{\setminus j'}$ of $K$ such that $j' \in \mathbb{I}_{\setminus j}$, it holds that $\nabla^{n'} Q_{k}^{\rho}(\llbracket\bm{\lambda}\rrbracket^{\bm{\beta}})|_{F_{\setminus j'}} = 0$ for $0 \le n' \le r_{1}$.
    \end{enumerate}
\end{corollary}

\begin{proof}
    For each $(d - 1)$-dimensional subsimplex $F_{\setminus j}$ of $K$, $j \in \mathbb{I}_{d}$, each $n$ such that $b \le n \le r_{1}$, and each $\bm{\beta} \in \Sigma_{F_{\setminus j}, n - b}^{\bm{r}^{\circ}}(\mathbb{I}_{d}, \rho)$, from \eqref{eq:Sigma-F-n-b}, it holds that $|\bm{\beta}|_{\setminus F_{\setminus j}} = \beta_{j} = n - b$. Hence, it follows that $\nabla^{n'} \llbracket\bm{\lambda}\rrbracket^{\bm{\beta}}|_{F_{\setminus j}} = 0$ for $0 \le n' \le n - b - 1$. 
    
    Moreover, it holds that $|\bm{\beta}|_{\setminus E} \ge r^{\circ}_{2} + 1 = r_{2} - b + 1$ for each $(d - 2)$-dimensional subsimplex $E$ of $K$. For each $(d - 1)$-dimensional subsimplex $F_{\setminus j'}$ of $K$ such that $j' \in \mathbb{I}_{\setminus j}$, let $E$ be the common $(d - 2)$-dimensional subsimplex of both $F_{\setminus j}$ and $F_{\setminus j'}$. Then, $\mathbb{I}_{\setminus E} = \{j, j'\}$ and $|\bm{\beta}|_{\setminus F_{\setminus j'}} = \beta_{j'} = |\bm{\beta}|_{\setminus E} - |\bm{\beta}|_{\setminus F_{\setminus j}} \ge (r_{2} - b + 1) - (n - b) \ge r_{2} - r_{1} + 1 \ge r_{1} - b + 1$. Therefore, it follows that $\nabla^{n'} \llbracket\bm{\lambda}\rrbracket^{\bm{\beta}}|_{F_{\setminus j'}} = 0$ for $0 \le n' \le r_{1} - b$.
    
    This completes the proof by Lemma~\ref{lem:continuity-equivalent}.
\end{proof}

\begin{corollary}
    \label{coro:Q-beta-beta-bubble}
    Suppose that $\bm{r} \coloneqq (r_{1}, \ldots, r_{d})$, $\rho$, $k$, and $b$ jointly satisfy Condition~\ref{cond:Cr-macro-element}. 
    Consider any $\bm{\beta} \in \Sigma_{0}^{\bm{r}^{\circ}}(\mathbb{I}_{d}, \rho)$. For each $(d - 1)$-dimensional subsimplex $F$ of $K$, it holds that $\nabla^{n} Q_{k}^{\rho}(\llbracket\bm{\lambda}\rrbracket^{\bm{\beta}})|_{F} = 0$ for $0 \le n \le r_{1}$. 
\end{corollary}

\begin{proof}
    For each $\bm{\beta} \in \Sigma_{0}^{\bm{r}^{\circ}}(\mathbb{I}_{d}, \rho)$, and each $(d - 1)$-dimensional subsimplex $F$ of $K$, from \eqref{eq:Sigma-0-d}, it holds that $|\bm{\beta}|_{\setminus F} \ge r^{\circ}_{1} + 1 = r_{1} - b + 1$. Then, it follows that $\nabla^{n} \llbracket\bm{\lambda}\rrbracket^{\bm{\beta}}|_{F} = 0$ for $0 \le n \le r_{1} - b$, which completes the proof by Lemma~\ref{lem:continuity-equivalent}.
\end{proof}

First, consider a specific subspace of $\mathcal{P}_{k}(K)$, corresponding to $\Sigma^{\partial}$, $\bm{r}^{\partial}$, and $k$ defined in \eqref{eq:alpha-set}, such that
\begin{equation*}
    \mathcal{P}^{\partial} \coloneqq \Span \big\{\llbracket\bm{\lambda}\rrbracket^{\bm{\alpha}}: \bm{\alpha} \in \Sigma^{\partial}\big\}.
\end{equation*}
The following proposition indicates that the set of degrees of freedom $\{\varphi_{\bm{\alpha}}: \bm{\alpha} \in \Sigma^{\partial}\}$ is unisolvent for the shape function space $\mathcal{P}^{\partial}$. Since $\bm{r}^{\partial}$ and $k$ jointly satisfy Condition~\ref{cond:Cr-element}, the proof follows the same argument as in \cite[Proposition 5.4]{hu2024construction}, which relies on a partial order of the pair $(F, n)$. Two lemmas from \cite{hu2024construction} are used in this proof: \cite[Lemma 5.2]{hu2024construction} is applied to reduce the proof to the unisolvence on smaller polynomial subspaces by mathematical induction with respect to this partial order, while \cite[Lemma 5.3]{hu2024construction} gives the unisolvence on each of these subspaces.

\begin{proposition}
    \label{prop:sub-unsolvence-alpha}
    Suppose that $\bm{r}$, $\rho$, $k$, and $b$ jointly satisfy Condition~\ref{cond:Cr-macro-element}. 
    For $u \in \mathcal{P}^{\partial}$, if $\varphi_{\bm{\alpha}}(u) = 0$ for all $\bm{\alpha} \in \Sigma^{\partial}$, then $u = 0$.
\end{proposition}

\begin{proof}
    Consider $u \in \mathcal{P}^{\partial}$ such that $\varphi_{\bm{\alpha}}(u) = 0$ for all $\bm{\alpha} \in \Sigma^{\partial}$. For each $t$-dimensional ($0 \le t \le d - 1$) subsimplex $F$ of $K$ and $0 \le n \le r^{\partial}_{d - t}$, define the polynomial space
    \begin{equation*}
        \mathcal{P}_{F, n} \coloneqq \Span \big\{\llbracket\bm{\lambda}\rrbracket^{\bm{\alpha}}: \bm{\alpha} \in \Sigma_{F, n}^{\bm{r}^{\partial}}(\mathbb{I}_{d}, k)\big\}.
    \end{equation*}
    There exist $u_{F, n} \in \mathcal{P}_{F, n}$ for each $t$-dimensional ($0 \le t \le d - 1$) subsimplex $F$ of $K$ and $0 \le n \le r^{\partial}_{d - t}$, such that
    \begin{equation*}
        u = \sum_{F \in \mathcal{T}^{\partial}(K)} \sum_{n = 0}^{r^{\partial}_{d - t_{F}}} u_{F, n},
    \end{equation*}
    where $F$ is a $t_{F}$-dimensional subsimplex.
    
    We introduce a partial order of all pairs $(F, n)$, for each $t$-dimensional ($0 \le t \le d - 1$) subsimplex $F$ of $K$ and $0 \le n \le r^{\partial}_{d - t}$. Say $(E, n') \preceq (F, n)$, if $E$ is a $t'$-dimensional ($0 \le t' \le t - 1$) subsimplex of $F$, or $E = F$ and $n' \le n$. 
    
    Suppose that $u_{E, n'} = 0$ for all $(E, n') \preceq (F, n)$ with $(E, n') \ne (F, n)$. By \cite[Lemma 5.2]{hu2024construction}, for each $\bm{\alpha} \in \Sigma_{F, n}^{\bm{r}^{\partial}}(\mathbb{I}_{d}, k)$, it holds that
    \begin{equation*}
        \varphi_{\bm{\alpha}}(u_{E, n'}) = 0, \quad \forall (E, n') \not \preceq (F, n).
    \end{equation*}
    Therefore, $\varphi_{\bm{\alpha}}(u_{F, n}) = 0$ for all $\bm{\alpha} \in \Sigma_{F, n}^{\bm{r}^{\partial}}(\mathbb{I}_{d}, k)$, which implies $u_{F, n} = 0$ from \cite[Lemma 5.3]{hu2024construction}. By mathematical induction, it holds that $u_{F, n} = 0$ for each $t$-dimensional ($0 \le t \le d - 1$) subsimplex $F$ of $K$ and $0 \le n \le r^{\partial}_{d - t}$, and it follows that $u = 0$.
\end{proof}

\begin{proposition}
    \label{prop:sub-unsolvence-beta}
    Suppose that $\bm{r}$, $\rho$, $k$, and $b$ jointly satisfy Condition~\ref{cond:Cr-macro-element}. 
    For $u \in \mathcal{Q}^{\circ}$, if $\psi_{\bm{\beta}}(u) = 0$ for all $\bm{\beta} \in \Sigma^{\circ}$, then $u = 0$.
\end{proposition}

\begin{proof}
    First, we define the interior bubble function space as 
    \begin{equation*}
        \mathcal{B} \coloneqq \Span \big\{Q_{k}^{\rho}(\llbracket\bm{\lambda}\rrbracket^{\bm{\beta}}): \bm{\beta} \in \Sigma^{\bm{r}^{\circ}}_{0}(\mathbb{I}_{d}, \rho)\big\},
    \end{equation*}
    which is a subspace of $\mathcal{Q}^{\circ}$ and vanishes on all the boundary degrees of freedom by Corollaries~\ref{coro:Q-alpha-beta} and~\ref{coro:Q-beta-beta-bubble}. We prove the partial unisolvence of $\{\psi_{\bm{\beta}}: \bm{\beta} \in \Sigma_{0}^{\bm{r}^{\circ}}(\mathbb{I}_{d}, \rho)\}$ for the interior bubble function space $\mathcal{B}$. For $u \in \mathcal{B}$, there exist $c_{\bm{\beta}}$ for $\bm{\beta} \in \Sigma_{0}^{\bm{r}^{\circ}}(\mathbb{I}_{d}, \rho)$, such that
    \begin{equation*}
        u = \sum_{\bm{\beta} \in \Sigma_{0}^{\bm{r}^{\circ}}(\mathbb{I}_{d}, \rho)} c_{\bm{\beta}} Q_{k}^{\rho}(\llbracket\bm{\lambda}\rrbracket^{\bm{\beta}}).
    \end{equation*}
    If $\psi_{\bm{\beta}}(u) = 0$ for all $\bm{\beta} \in \Sigma_{0}^{\bm{r}^{\circ}}(\mathbb{I}_{d}, \rho)$, it follows that
    \begin{equation*}
        \begin{aligned}
            0 & = \sum_{\bm{\beta} \in \Sigma_{0}^{\bm{r}^{\circ}}(\mathbb{I}_{d}, \rho)} c_{\bm{\beta}} \psi_{\bm{\beta}}(u) \\
            & = \sum_{\bm{\beta} \in \Sigma_{0}^{\bm{r}^{\circ}}(\mathbb{I}_{d}, \rho)} c_{\bm{\beta}} \cdot \frac{1}{|K|} \int_{K} u \cdot Q_{k}^{\rho}(\llbracket\bm{\lambda}\rrbracket^{\bm{\beta}}) \\
            & = \frac{1}{|K|} \int_{K} u \cdot \Bigg(\sum_{\bm{\beta} \in \Sigma_{0}^{\bm{r}^{\circ}}(\mathbb{I}_{d}, \rho)} c_{\bm{\beta}} Q_{k}^{\rho}(\llbracket\bm{\lambda}\rrbracket^{\bm{\beta}})\Bigg) \\
            & = \frac{1}{|K|} \int_{K} u \cdot u,
        \end{aligned}
    \end{equation*}
    which implies that $u = 0$.

    Next, for each $j \in \mathbb{I}_{d}$ and each $b \le n \le r_{1}$, we prove the partial unisolvence of $\{\psi_{\bm{\beta}}: \bm{\beta} \in \Sigma_{F_{\setminus j}, n - b}^{\bm{r}^{\circ}}(\mathbb{I}_{d}, \rho)\}$ for the piecewise function space
    \begin{equation*}
        \mathcal{Q}_{j, n} \coloneqq \Span \big\{Q_{k}^{\rho}(\llbracket\bm{\lambda}\rrbracket^{\bm{\beta}}): \bm{\beta} \in \Sigma_{F_{\setminus j}, n - b}^{\bm{r}^{\circ}}(\mathbb{I}_{d}, \rho)\big\}.
    \end{equation*}
    For each $\bm{\beta} \coloneqq (\beta_{0}, \beta_{1}, \ldots, \beta_{d}) \in \Sigma_{F_{\setminus j}, n - b}^{\bm{r}^{\circ}}(\mathbb{I}_{d}, \rho)$, it holds that $\beta_{j} = n - b$. Consider $u \in \mathcal{Q}_{j, n}$ defined as
    \begin{equation*}
        u = \sum_{\bm{\beta} \in \Sigma_{F_{\setminus j}, n - b}^{\bm{r}^{\circ}}(\mathbb{I}_{d}, \rho)} c_{\bm{\beta}} Q_{k}^{\rho}(\llbracket\bm{\lambda}\rrbracket^{\bm{\beta}}),
    \end{equation*}
    such that for each $\bm{\beta} \in \Sigma_{F_{\setminus j}, n - b}^{\bm{r}^{\circ}}(\mathbb{I}_{d}, \rho)$, it holds that $\psi_{\bm{\beta}}(u) = 0$.
    
    Given $\bm{\beta} \in \Sigma_{F_{\setminus j}, n - b}^{\bm{r}^{\circ}}(\mathbb{I}_{d}, \rho)$, using the relationship between barycentric coordinates presented in Lemma~\ref{lem:transformation}, we can calculate that
    \begin{equation*}
        \begin{aligned}
            \llbracket\bm{\lambda}\rrbracket^{\bm{\beta}} & = \frac{\lambda_{j}^{n - b}}{(n - b)!} \prod_{i \in \mathbb{I}_{\setminus j}} \frac{\lambda_{i}^{\beta_{i}}}{\beta_{i}!} \\
            & = \frac{(\mu_{j} \lambda_{j, j})^{n - b}}{(n - b)!} \prod_{i \in \mathbb{I}_{\setminus j}} \frac{(\lambda_{j, i} + \mu_{i} \lambda_{j, j})^{\beta_{i}}}{\beta_{i}!} \\
            & = \mu_{j}^{n - b} \frac{\lambda_{j, j}^{n - b}}{(n - b)!} \prod_{i \in \mathbb{I}_{\setminus j}} \frac{\lambda_{j, i}^{\beta_{i}}}{\beta_{i}!} + \lambda_{j, j}^{n - b + 1} \bar{p}_{\bm{\beta}},
        \end{aligned}
    \end{equation*}
    where $\bar{p}_{\bm{\beta}}$ is a polynomial such that $\bar{p}_{\bm{\beta}} \in \mathcal{P}_{k - n - 1}(K)$. Hence, it holds that
    \begin{equation*}
        Q_{k}^{\rho}(\llbracket\bm{\lambda}\rrbracket^{\bm{\beta}})|_{K_{j}} = \mu_{j}^{n - b} \frac{\lambda_{j, j}^{n}}{n!} \prod_{i \in \mathbb{I}_{\setminus j}} \frac{\lambda_{j, i}^{\beta_{i}}}{\beta_{i}!} + \lambda_{j, j}^{n + 1} \bar{p}_{\bm{\beta}, b} = \mu_{j}^{- b} \frac{\lambda_{j}^{n}}{n!} \prod_{i \in \mathbb{I}_{\setminus j}} \frac{\lambda_{j, i}^{\beta_{i}}}{\beta_{i}!} + \lambda_{j, j}^{n + 1} \bar{p}_{\bm{\beta}, b},
    \end{equation*}
    where $\bar{p}_{\bm{\beta}, b}$ is a polynomial such that $\bar{p}_{\bm{\beta}, b} \in \mathcal{P}_{k - n - 1}(K)$.
    
    Throughout the following part of the proof, we write $F := F_{\setminus j}$ for notational simplicity, and note that each $\bm{\beta} \in \Sigma_{F, n - b}^{\bm{r}^{\circ}}(\mathbb{I}_{d}, \rho)$ corresponds to a unique $\bm{\sigma} \in \Sigma_{0}^{\bm{q}_{d - 1, n}}(\mathbb{I}_{F}, k - n)$ by the bijection $\mathcal{R}^{\circ}_{b}: \bm{\beta} \mapsto (\bm{\theta}, \bm{\sigma})$, namely, $\sigma_{i} = \beta_{i}$ for each $i \in \mathbb{I}_{F} = \mathbb{I}_{\setminus j}$. 
    Considering the $n$-th order normal derivative on $F$, it holds that
    \begin{equation*}
        \begin{aligned}
        \frac{\partial^{n}}{\partial \bm{n}_{F}^{n}} Q_{k}^{\rho}(\llbracket\bm{\lambda}\rrbracket^{\bm{\beta}})\bigg|_{F} & = \frac{\partial^{n}}{\partial \bm{n}_{F}^{n}} Q_{k}^{\rho}(\llbracket\bm{\lambda}\rrbracket^{\bm{\beta}})|_{K_{j}}\bigg|_{F} \\
        & = \frac{\partial^{n}}{\partial \bm{n}_{F}^{n}} \bigg(\mu_{j}^{- b} \frac{\lambda_{j}^{n}}{n!} \prod_{i \in \mathbb{I}_{\setminus j}} \frac{\lambda_{j, i}^{\beta_{i}}}{\beta_{i}!} + \lambda_{j, j}^{n + 1} \bar{p}_{\bm{\beta}, b} \bigg) \bigg|_{F} = \mu_{j}^{- b} \bar{c}_{j}^{n} \llbracket\bm{\lambda}\rrbracket_{F}^{\bm{\sigma}},
        \end{aligned}
    \end{equation*}
    where $0 < \mu_{j} < 1$ and $\bar{c}_{j} = \frac{\partial}{\partial \bm{n}_{F}} \lambda_{j} \ne 0$ are constants corresponding to $\mathcal{T}_{A}(K)$. Then, we have 
    \begin{equation*}
        \frac{\partial^{n}}{\partial \bm{n}_{F}^{n}} u\bigg|_{F} = \sum_{\bm{\beta} \in \Sigma_{F, n - b}^{\bm{r}^{\circ}}(\mathbb{I}_{d}, \rho)} c_{\bm{\beta}} \bigg(\frac{\partial^{n}}{\partial \bm{n}_{F}^{n}} Q_{k}^{\rho}(\llbracket\bm{\lambda}\rrbracket^{\bm{\beta}})\bigg)\bigg|_{F} = \mu_{j}^{- b} \bar{c}_{j}^{n} \sum_{\bm{\beta} \in \Sigma_{F, n - b}^{\bm{r}^{\circ}}(\mathbb{I}_{d}, \rho)} c_{\bm{\beta}} \llbracket\bm{\lambda}\rrbracket_{F}^{\bm{\sigma}}.
    \end{equation*}
    This gives
    \begin{equation*}
        \begin{aligned}
            0 & = \sum_{\bm{\beta} \in \Sigma_{F, n - b}^{\bm{r}^{\circ}}(\mathbb{I}_{d}, \rho)} c_{\bm{\beta}} \psi_{\bm{\beta}}(u) \\
            & = \sum_{\bm{\beta} \in \Sigma_{F, n - b}^{\bm{r}^{\circ}}(\mathbb{I}_{d}, \rho)} c_{\bm{\beta}} \cdot \frac{1}{|F|} \int_{F} \bigg(\frac{\partial^{n}}{\partial \bm{n}_{F}^{n}} u\bigg) \cdot \llbracket\bm{\lambda}\rrbracket_{F}^{\bm{\sigma}} \\
            & = \frac{1}{|F|} \int_{F} \bigg(\frac{\partial^{n}}{\partial \bm{n}_{F}^{n}} u\bigg) \cdot \Bigg(\sum_{\bm{\beta} \in \Sigma_{F, n - b}^{\bm{r}^{\circ}}(\mathbb{I}_{d}, \rho)} c_{\bm{\beta}} \llbracket\bm{\lambda}\rrbracket_{F}^{\bm{\sigma}}\Bigg) \\
            & = \mu_{j}^{- b} \bar{c}_{j}^{n} \frac{1}{|F|} \int_{F} \Bigg(\sum_{\bm{\beta} \in \Sigma_{F, n - b}^{\bm{r}^{\circ}}(\mathbb{I}_{d}, \rho)} c_{\bm{\beta}} \llbracket\bm{\lambda}\rrbracket_{F}^{\bm{\sigma}}\Bigg) \cdot \Bigg(\sum_{\bm{\beta} \in \Sigma_{F, n - b}^{\bm{r}^{\circ}}(\mathbb{I}_{d}, \rho)} c_{\bm{\beta}} \llbracket\bm{\lambda}\rrbracket_{F}^{\bm{\sigma}}\Bigg),
        \end{aligned}
    \end{equation*}
    which implies that $c_{\bm{\beta}} = 0$ for each $\bm{\beta} \in \Sigma_{F, n - b}^{\bm{r}^{\circ}}(\mathbb{I}_{d}, \rho)$.
    
    We end this part of the proof by concluding that for $u \in \mathcal{Q}_{j, n}$, if $\psi_{\bm{\beta}}(u) = 0$ for all $\bm{\beta} \in \Sigma_{F_{\setminus j}, n - b}^{\bm{r}^{\circ}}(\mathbb{I}_{d}, \rho)$, then $u = 0$. Note that here we used again the original meaning of $F$.

    Finally, we prove the unisolvence of $\{\psi_{\bm{\beta}}: \bm{\beta} \in \Sigma^{\circ}\}$ for the piecewise polynomial space $\mathcal{Q}^{\circ}$. Consider $u \in \mathcal{Q}^{\circ}$ such that $\psi_{\bm{\beta}}(u) = 0$ for all $\bm{\beta} \in \Sigma^{\circ}$. There exist $u^{\circ} \in \mathcal{B}$ and $u_{j, n} \in \mathcal{Q}_{j, n}$ for $j \in \mathbb{I}_{d}$ and $b \le n \le r_{1}$, such that
    \begin{equation*}
        u = u^{\circ} + \sum_{j \in \mathbb{I}_{d}} \sum_{n = b}^{r_{1}} u_{j, n}.
    \end{equation*}
    Fix $j \in \mathbb{I}_{d}$ and $n$ such that $b \le n \le r_{1}$. Moreover, suppose that $u_{j, n'} = 0$ for all $b \le n' \le n - 1$, so $u = u^{\circ} + \sum_{j' \in \mathbb{I}_{\setminus j}} \sum_{n' = b}^{r_{1}} u_{j', n'} + \sum_{n' = n}^{r_{1}} u_{j, n'}$. For each $\bm{\beta} \in \Sigma_{F_{\setminus j}, n - b}^{\bm{r}^{\circ}}(\mathbb{I}_{d}, \rho)$, by Corollaries~\ref{coro:Q-beta-beta} and~\ref{coro:Q-beta-beta-bubble}, it holds that
    \begin{equation*}
        \psi_{\bm{\beta}}(u_{j, n'}) = 0, \quad \forall n + 1 \le n' \le r_{1}, \qquad \psi_{\bm{\beta}}(u_{j', n'}) = 0, \quad \forall j' \in \mathbb{I}_{\setminus j}, \ b \le n' \le r_{1}, \qquad \psi_{\bm{\beta}}(u^{\circ}) = 0.
    \end{equation*}
    Therefore, we have $\psi_{\bm{\beta}}(u_{j, n}) = 0$ for all $\bm{\beta} \in \Sigma_{F_{\setminus j}, n - b}^{\bm{r}^{\circ}}(\mathbb{I}_{d}, \rho)$, which implies $u_{j, n} = 0$. By mathematical induction, it holds that $u_{j, n} = 0$ for all $j \in \mathbb{I}_{d}$ and all $b \le n \le r_{1}$.
    Then, the piecewise polynomial $u$ is equal to an interior bubble function $u^{\circ} \in \mathcal{B}$. Since $\psi_{\bm{\beta}}(u^{\circ}) = 0$ for all $\bm{\beta} \in \Sigma_{0}^{\bm{r}^{\circ}}(\mathbb{I}_{d}, \rho)$, it follows that $u = u^{\circ} = 0$.
\end{proof}

Now, we are ready to show that the set of degrees of freedom is unisolvent for $\mathcal{S}_{k}^{\bm{r}, \rho}(\mathcal{T}_{A}(K))$.

\begin{proof}[Proof of Theorem~\ref{thm:unsolvence}]
    First, we prove that $\llbracket\bm{\lambda}\rrbracket^{\bm{\alpha}}$ for $\bm{\alpha} \in \Sigma^{\partial}$ and $Q_{k}^{\rho}(\llbracket\bm{\lambda}\rrbracket^{\bm{\beta}})$ for $\bm{\beta} \in \Sigma^{\circ}$ are linearly independent. Suppose that there exist $u^{\partial} \in \mathcal{P}^{\partial}$ and $u^{\circ} \in \mathcal{Q}^{\circ}$, such that $u^{\partial} + u^{\circ} = 0$. Then, for each $\bm{\alpha} \in \Sigma^{\partial}$, it holds that $\varphi_{\bm{\alpha}}(u^{\partial}) + \varphi_{\bm{\alpha}}(u^{\circ}) = \varphi_{\bm{\alpha}}(u^{\partial} + u^{\circ}) = 0$. 
    
    From \eqref{eq:alpha-dof-vertex}, \eqref{eq:alpha-dof}, \eqref{eq:alpha-dof-d-1} and Corollary~\ref{coro:Q-alpha-beta}, it follows that $\varphi_{\bm{\alpha}}(u^{\circ}) = 0$. Therefore, for each $\bm{\alpha} \in \Sigma^{\partial}$, it holds that $\varphi_{\bm{\alpha}}(u^{\partial}) = 0$, which implies that $u^{\partial} = 0$ by Proposition~\ref{prop:sub-unsolvence-alpha} and $u^{\circ} = 0$. Since both $\{\llbracket\bm{\lambda}\rrbracket^{\bm{\alpha}}: \bm{\alpha} \in \Sigma^{\partial}\}$ and $\{Q_{k}^{\rho}(\llbracket\bm{\lambda}\rrbracket^{\bm{\beta}}): \bm{\beta} \in \Sigma^{\circ}\}$ are bases of $\mathcal{P}^{\partial}$ and $\mathcal{Q}^{\circ}$, respectively, $\llbracket\bm{\lambda}\rrbracket^{\bm{\alpha}}$ for $\bm{\alpha} \in \Sigma^{\partial}$ and $Q_{k}^{\rho}(\llbracket\bm{\lambda}\rrbracket^{\bm{\beta}})$ for $\bm{\beta} \in \Sigma^{\circ}$ are linearly independent.

    Next, we prove that $\mathcal{S}_{k}^{\bm{r}, \rho}(\mathcal{T}_{A}(K)) = \mathcal{P}^{\partial} + \mathcal{Q}^{\circ}$. Obviously, $\mathcal{P}^{\partial} \subseteq \mathcal{S}_{k}^{\bm{r}, \rho}(\mathcal{T}_{A}(K))$, and $\mathcal{Q}^{\circ} \subseteq \mathcal{S}_{k}^{\bm{r}, \rho}(\mathcal{T}_{A}(K))$ is proved by Proposition~\ref{prop:subspace-Q-S}. It follows that $\mathcal{P}^{\partial} + \mathcal{Q}^{\circ} \subseteq \mathcal{S}_{k}^{\bm{r}, \rho}(\mathcal{T}_{A}(K))$ and
    \begin{equation}
        \label{eq:dim-S-P-Q}
        \dim \mathcal{S}_{k}^{\bm{r}, \rho}(\mathcal{T}_{A}(K)) \ge \card(\Sigma^{\partial}) + \card(\Sigma^{\circ}).
    \end{equation}
    Consider a subspace $\mathcal{S}^{\circ}$ of $\mathcal{S}_{k}^{\bm{r}, \rho}(\mathcal{T}_{A}(K))$, such that
    \begin{equation*}
        \mathcal{S}^{\circ} \coloneqq \big\{u \in \mathcal{S}_{k}^{\bm{r}, \rho}(\mathcal{T}_{A}(K)): \varphi_{\bm{\alpha}}(u) = 0 \text{ for each } \bm{\alpha} \in \Sigma^{\partial}\big\}.
    \end{equation*}
    Then, it follows that
    \begin{equation}
        \label{eq:dim-S-S-P}
        \dim \mathcal{S}_{k}^{\bm{r}, \rho}(\mathcal{T}_{A}(K)) \le \card(\Sigma^{\partial}) + \dim \mathcal{S}^{\circ}.
    \end{equation}

     For each $u \in \mathcal{S}^{\circ}$, from \eqref{eq:shape-function}, it holds that for $j \in \mathbb{I}_{d}$, $u|_{K_{j}} \in \mathcal{P}_{k}(K_{j})$ and $\nabla^{n} u|_{K_{j}}$ for $0 \le n \le \rho$ is single-valued at $V_{A}$. Moreover, for each $t$-dimensional ($0 \le t \le d - 1$) subsimplex $F$ of $K$, there exists $j \in \mathbb{I}_{d}$ such that $F$ is a subsimplex of $K_{j}$. Let $u_{j} \in \mathcal{P}_{k}(K)$ such that $u_{j}|_{K_{j}} = u|_{K_{j}}$. Then, for each $t'$-dimensional ($0 \le t' \le t$) subsimplex $E$ of $F$, and each $0 \le n' \le r^{\partial}_{d - t'}$, it holds that $\nabla^{n'} u_{j}|_{E} = \nabla^{n'} (u_{j}|_{K_{j}})|_{E} = \nabla^{n'} (u|_{K_{j}})|_{E} = \nabla^{n'} u|_{E}$. Consequently, $\varphi_{\bm{\alpha}}(u_{j}) = \varphi_{\bm{\alpha}}(u) = 0$ holds for each $\bm{\alpha} \in \Sigma_{E, n'}^{\bm{r}^{\partial}}(\mathbb{I}_{d}, k)$. Since $\bm{r}^{\partial}$ and $k$ jointly satisfy Condition~\ref{cond:Cr-element}, from \cite[Proposition A.2]{hu2026sharpness}, for $0 \le n \le r^{\partial}_{d - t}$, it follows that $\nabla^{n} u_{j}|_{F} = 0$, which leads to $\nabla^{n} u|_{F} = 0$.
     
     Specifically, for $t = d - 1$, namely, $F$ is a $(d - 1)$-dimensional subsimplex of $K$, $\nabla^{n} u|_{F} = 0$ holds for $0 \le n \le r^{\partial}_{1} = b - 1$, which implies that $u \in \mathcal{Q}_{k}^{\rho}(\mathcal{T}_{A}(K))$. Additionally, by Proposition~\ref{prop:boundary-condition}, it follows that $u \in \mathcal{Q}^{\circ}$ and $\mathcal{S}^{\circ} \subseteq \mathcal{Q}^{\circ}$, which leads to
    \begin{equation*}
        \dim \mathcal{S}^{\circ} \le \dim \mathcal{Q}^{\circ} = \card(\Sigma^{\circ}).
    \end{equation*}
    Together with \eqref{eq:dim-S-P-Q} and \eqref{eq:dim-S-S-P}, it follows that
    \begin{equation*}
        \dim \mathcal{S}_{k}^{\bm{r}, \rho}(\mathcal{T}_{A}(K)) = \card(\Sigma^{\partial}) + \card(\Sigma^{\circ}), \quad \mathcal{S}_{k}^{\bm{r}, \rho}(\mathcal{T}_{A}(K)) = \mathcal{P}^{\partial} + \mathcal{Q}^{\circ},
    \end{equation*}
    and $\{\llbracket\bm{\lambda}\rrbracket^{\bm{\alpha}}: \bm{\alpha} \in \Sigma^{\partial}\} \cup \{Q_{k}^{\rho}(\llbracket\bm{\lambda}\rrbracket^{\bm{\beta}}): \bm{\beta} \in \Sigma^{\circ}\}$ is a basis of $\mathcal{S}_{k}^{\bm{r}, \rho}(\mathcal{T}_{A}(K))$.

    Suppose that $u \in \mathcal{S}_{k}^{\bm{r}, \rho}(\mathcal{T}_{A}(K))$ such that $\varphi_{\bm{\alpha}}(u) = 0$ for each $\bm{\alpha} \in \Sigma^{\partial}$ and $\psi_{\bm{\beta}}(u) = 0$ for each $\bm{\beta} \in \Sigma^{\circ}$. There exist $u^{\partial} \in \mathcal{P}^{\partial}$ and $u^{\circ} \in \mathcal{Q}^{\circ}$ such that $u = u^{\partial} + u^{\circ}$. Then, for each $\bm{\alpha} \in \Sigma^{\partial}$, it holds that $\varphi_{\bm{\alpha}}(u^{\partial}) = \varphi_{\bm{\alpha}}(u^{\partial}) + \varphi_{\bm{\alpha}}(u^{\circ}) = \varphi_{\bm{\alpha}}(u) = 0$, which implies that $u^{\partial} = 0$ and $u = u^{\circ}$. Moreover, for each $\bm{\beta} \in \Sigma^{\circ}$, it holds that $\psi_{\bm{\beta}}(u^{\circ}) = \psi_{\bm{\beta}}(u) = 0$, which yields that $u = u^{\circ} = 0$. 
\end{proof}

\section{Proof of Continuity}
\label{sec:continuity}

In this section, we discuss the continuity relations between the local shape function spaces defined on two adjacent $d$-dimensional simplices, sharing the same degrees of freedom on the common parts. Afterwards, we show that the resulting finite element space over the triangulation $\mathcal{T}(\Omega)$ of the $d$-dimensional domain $\Omega$ is the superspline space $\mathcal{S}_{k}^{\bm{r}, \rho}(\mathcal{T}_{A}(\Omega))$ defined in \eqref{eq:superspline}. To this end, we need to update some notations and generalize the degrees of freedom of the local shape function space $\mathcal{S}_{k}^{\bm{r}, \rho}(\mathcal{T}_{A}(K))$ defined in Section~\ref{sec:construction} to a global finite element space. 

Given a triangulation $\mathcal{T}(\Omega)$ of a $d$-dimensional domain $\Omega$, let $V_{0}, V_{1}, \ldots, V_{N_{0} - 1}$ be the vertices of $\mathcal{T}(\Omega)$. For a $t$-dimensional ($0 \le t \le d$) subsimplex $F$ of $\mathcal{T}(\Omega)$, with vertices $V_{i_{0}}, V_{i_{1}}, \ldots, V_{i_{t}}$, let $\mathbb{I}_{F} \coloneqq \{i_{0}, i_{1}, \ldots, i_{t}\}$ be the generalized vertex index set. In particular, when $F$ is a $d$-dimensional simplex $K$ or $F$ is a vertex $V$, the notations $\mathbb{I}_{K}$ and $\mathbb{I}_{V}$ are also used, respectively. Moreover, denote $\mathbb{I}_{K, \setminus F} \coloneqq \mathbb{I}_{K} \setminus \mathbb{I}_{F}$ for each $K$ such that $F$ is a subsimplex of $K$, and $\mathbb{I}_{F, \setminus E} \coloneqq \mathbb{I}_{F} \setminus \mathbb{I}_{E}$ for each $E$ such that $E$ is a subsimplex of $F$. 

The generalized barycentric coordinates corresponding to $K$ and $F$ are denoted as $\lambda_{K, i}$ for $i \in \mathbb{I}_{K}$ and $\lambda_{F, i}$ for $i \in \mathbb{I}_{F}$, respectively, and simplified as $\lambda_{i}$. Consequently, the generalized normalized monomials $\llbracket\bm{\lambda}\rrbracket_{K}^{\bm{\alpha}}$ and $\llbracket\bm{\lambda}\rrbracket_{F}^{\bm{\sigma}}$ can be defined similarly to their local versions in Section~\ref{sec:notation}. 
Moreover, the generalized multi-index sets $\Sigma(\mathbb{I}_{K}, l)$, $\Sigma_{0}^{\bm{q}}(\mathbb{I}_{K}, l)$, $\Sigma(\mathbb{I}_{F}, l)$, $\Sigma_{0}^{\bm{q}}(\mathbb{I}_{F}, l)$, and the generalized partial sums $|\bm{\alpha}|_{K}$, $|\bm{\alpha}|_{K, \setminus F}$, $|\bm{\sigma}|_{F}$, $|\bm{\sigma}|_{F, \setminus E}$ can be defined in the same way as their local versions in Sections~\ref{sec:notation} and~\ref{sec:construction}. It is important to note that the indices involved in these global definitions should be interpreted locally with respect to the simplex of interest.

Recall the orthogonal unit normal vectors $\bm{n}_{V, 1}, \ldots, \bm{n}_{V, d}$ and the orthogonal unit normal vector(s) $\bm{n}_{F, t + 1}, \ldots, \bm{n}_{F, d}$ in Definition~\ref{def:local-dofs}, which only depend on the vertex $V$ and the $t$-dimensional subsimplex $F$, respectively. For the global degrees of freedom, we fix these vectors for each vertex $V$ of $\mathcal{T}(\Omega)$ and each subsimplex $F$ of $\mathcal{T}(\Omega)$, and use the same vectors in the local degrees of freedom on every $d$-dimensional simplex $K \in \mathcal{T}(\Omega)$ containing the corresponding vertex or subsimplex.

\begin{definition}[Global degrees of freedom]
\label{def:global-dofs}
Let $u \in \mathcal{S}_{k}^{\bm{r}, \rho}(\mathcal{T}_{A}(\Omega))$. The set of global degrees of freedom for $u$ is divided into three subsets.
\begin{enumerate}
    \item For each vertex $V$ of $\mathcal{T}(\Omega)$, the global degrees of freedom at $V$ for $u$ are defined as follows: for each $0 \le n \le r_{d}$, and each $\bm{\theta} \in \Sigma(\mathbb{I}_{d, \setminus 0}, n)$,
    \begin{equation}
        \label{eq:set-dof-1-global}
        \eta_{V, \bm{\theta}}: u \longmapsto \bm{D}_{V}^{\bm{\theta}} u \big|_{V}.
    \end{equation}
    \item For each $t$-dimensional $(1 \le t \le d - 1)$ subsimplex $F$ of $\mathcal{T}(\Omega)$, the global degrees of freedom on $F$ for $u$ are defined as follows: for each $0 \le n \le r_{d - t}$, each $\bm{\theta} \in \Sigma(\mathbb{I}_{d, \setminus t}, n)$, and each $\bm{\sigma} \in \Sigma_{0}^{\bm{q}_{t, n}}(\mathbb{I}_{F}, k - n)$, 
    \begin{equation}
        \label{eq:set-dof-2-global}
        \eta_{F, \bm{\theta}, \bm{\sigma}}: u \longmapsto \frac{1}{|F|} \int_{F} (\bm{D}_{F}^{\bm{\theta}} u) \cdot \llbracket\bm{\lambda}\rrbracket_{F}^{\bm{\sigma}},
    \end{equation}
    where the continuity vector $\bm{q}_{t, n} \coloneqq (r_{d - t + 1} - n, \ldots, r_{d} - n)$.
    \item For each $d$-dimensional simplex $K$ of $\mathcal{T}(\Omega)$, the global degrees of freedom defined on $K$ are defined as follows: for each $\bm{\beta} \in \Sigma_{0}^{\bm{r}^{\circ}}(\mathbb{I}_{K}, \rho)$,
    \begin{equation}
        \label{eq:set-dof-3-global}
        \eta_{K, \bm{\beta}}: u \longmapsto \frac{1}{|K|} \int_{K} u \cdot Q_{k}^{\rho}(\llbracket\bm{\lambda}\rrbracket_{K}^{\bm{\beta}}),
    \end{equation}
    where the continuity vector $\bm{r}^{\circ} \coloneqq (r_{1} - b, \ldots, r_{d} - b)$.
\end{enumerate}
\end{definition}

For simplicity of notation, a generic global degree of freedom is denoted by $\eta$.
Note that the global degrees of freedom defined on $V$ (or $F$, $K$) are exactly the local degrees of freedom defined on $V$ (or $F$, $K$) for the element $K$ and the local shape function space $\mathcal{S}_{k}^{\bm{r}, \rho}(\mathcal{T}_{A}(K))$, where $V$ (or $F$, $K$) is a subsimplex of $K$.

The proof of the following proposition follows exactly the same argument as in \cite[Proposition~A.2]{hu2026sharpness}, except that the assumption on the relationship between $r_{1}$ and $r_{2}$ is different.
\begin{proposition}
    \label{prop:continuity}
    Suppose that $\bm{r} \coloneqq (r_{1}, \ldots, r_{d})$, $\rho$, $k$, and $b$ jointly satisfy Condition~\ref{cond:Cr-macro-element}. 
    Let $F$ be a $t$-dimensional ($0 \le t \le d - 1$) subsimplex, shared by two $d$-dimensional simplices $K_{+}$ and $K_{-}$, with their Alfeld splits $\mathcal{T}_{A}(K_{+})$ and $\mathcal{T}_{A}(K_{-})$, respectively. If the piecewise polynomials $u_{+} \in \mathcal{S}_{k}^{\bm{r}, \rho}(\mathcal{T}_{A}(K_{+}))$ and $u_{-} \in \mathcal{S}_{k}^{\bm{r}, \rho}(\mathcal{T}_{A}(K_{-}))$ satisfy that for each $t'$-dimensional ($0 \le t' \le t$) subsimplex $E$ of $F$, it holds that $\eta(u_{+}) = \eta(u_{-})$ for each global degree of freedom $\eta$ defined on $E$, then for each $n$ such that $0 \le n \le r_{d - t}$, it follows that
    \begin{equation*}
        \nabla^{n} u_{+}|_{F} = \nabla^{n} u_{-}|_{F}.
    \end{equation*}
\end{proposition}

\begin{proof}
    It suffices to check the cases for $1 \le t \le d - 1$ and the normal derivatives $\bm{D}_{F}^{\bm{\theta}}$, for each $n$ such that $0 \le n \le r_{d - t}$ and each $\bm{\theta} \in \Sigma(\mathbb{I}_{d, \setminus t}, n)$, i.e.,
    \begin{equation*}
        \bm{D}_{F}^{\bm{\theta}} u_{+}|_{F} = \bm{D}_{F}^{\bm{\theta}} u_{-}|_{F}.
    \end{equation*}
    
    Let the continuity vector $\bar{\bm{r}} \coloneqq (\bar{r}_{1}, \ldots, \bar{r}_{t}) = (r_{d - t + 1} - n, \ldots, r_{d} - n)$ and the polynomial degree $\bar{k} = k - n$. Since $1 \le t \le d - 1$, $0 \le n \le r_{d - t}$, and $\bm{r}$, $\rho$, $k$, and $b$ jointly satisfy Condition~\ref{cond:Cr-macro-element}, for $2 \le s \le t$, it holds that $\bar{r}_{s} - 2 \bar{r}_{s - 1} = r_{d - t + s} - 2 r_{d - t + s - 1} + n \ge 0$, and $\bar{k} - 2 \bar{r}_{t} - 1 = k - 2 r_{d} - 1 + n \ge 0$. Therefore, the continuity vector $\bar{\bm{r}}$ and the polynomial degree $\bar{k}$ jointly satisfy Condition~\ref{cond:Cr-element}. Consider the finite element for $F$ defined in \cite{hu2024construction}, corresponding to $\bar{\bm{r}}$ and $\bar{k}$, which has the set of degrees of freedom defined as follows. This set is divided into three subsets.     
    \begin{enumerate}
        \item For a vertex $V$ of $F$, the degrees of freedom at $V$ for $u$ are defined for all $0 \le n' \le \bar{r}_{t} = r_{d} - n$ and each $\bm{\theta}' \in \Sigma(\mathbb{I}_{t, \setminus 0}, n')$ as
        \begin{equation*}
            \bm{D}_{F, V}^{\bm{\theta}'} u|_{V}.
        \end{equation*}
        Here $\bm{D}_{F, V}^{\bm{\theta}'}$ denotes a derivative of order $n'$ within $F$. Since $\bm{D}_{F, V}^{\bm{\theta}'} \bm{D}_{F}^{\bm{\theta}}$ can be linearly represented by $\bm{D}_{V}^{\bm{\theta}''}$ for $\bm{\theta}'' \in \Sigma(\mathbb{I}_{d, \setminus 0}, n' + n)$, and $0 \le n' + n \le r_{d}$, it follows from $\eta(u_{+}) = \eta(u_{-})$ for degree of freedom $\eta$ at $V$ that
        \begin{equation*}
            \bm{D}_{F, V}^{\bm{\theta}'} (\bm{D}_{F}^{\bm{\theta}} u_{+})|_{V} = \bm{D}_{F, V}^{\bm{\theta}'} (\bm{D}_{F}^{\bm{\theta}} u_{-})|_{V}.
        \end{equation*}
        
        \item For a $t'$-dimensional ($1 \le t' \le t - 1$) subsimplex $E$ of $F$, the degrees of freedom on $E$ for $u$ are defined for each $0 \le n' \le \bar{r}_{t - t'} = r_{d - t'} - n$, each $\bm{\theta}' \in \Sigma(\mathbb{I}_{t, \setminus t'}, n')$, and each $\bm{\sigma} \in \Sigma_{0}^{\bar{\bm{q}}_{t', n'}}(\mathbb{I}_{E}, \bar{k} - n')$ as
        \begin{equation*}
            \frac{1}{|E|} \int_{E} (\bm{D}_{F, E}^{\bm{\theta}'} u) \cdot \llbracket\bm{\lambda}\rrbracket_{E}^{\bm{\sigma}},
        \end{equation*}
        where $\bar{k} - n' = k - (n' + n)$ and the continuity vector $\bar{\bm{q}}_{t', n'} \coloneqq (\bar{r}_{t - t' + 1} - n', \ldots, \bar{r}_{t} - n') = (r_{d - t' + 1} - (n' + n), \ldots, r_{d} - (n' + n)) =: \bm{q}_{t', n' + n}$. Here $\bm{D}_{F, E}^{\bm{\theta}'}$ denotes a normal derivative of order $n'$ to $E$ within $F$. Since $\bm{D}_{F, E}^{\bm{\theta}'} \bm{D}_{F}^{\bm{\theta}}$ can be linearly represented by $\bm{D}_{E}^{\bm{\theta}''}$ for $\bm{\theta}'' \in \Sigma(\mathbb{I}_{d, \setminus t'}, n' + n)$ with $0 \le n' + n \le r_{d - t'}$, and $\bm{\sigma} \in \Sigma_{0}^{\bar{\bm{q}}_{t', n'}}(\mathbb{I}_{E}, \bar{k} - n') = \Sigma_{0}^{\bm{q}_{t', n' + n}}(\mathbb{I}_{E}, k - (n' + n))$, it follows from $\eta(u_{+}) = \eta(u_{-})$ for degree of freedom $\eta$ on $E$ that
        \begin{equation*}
            \frac{1}{|E|} \int_{E} [\bm{D}_{F, E}^{\bm{\theta}'} (\bm{D}_{F}^{\bm{\theta}} u_{+})] \cdot \llbracket\bm{\lambda}\rrbracket_{E}^{\bm{\sigma}} = \frac{1}{|E|} \int_{E} [\bm{D}_{F, E}^{\bm{\theta}'} (\bm{D}_{F}^{\bm{\theta}} u_{-})] \cdot \llbracket\bm{\lambda}\rrbracket_{E}^{\bm{\sigma}}.
        \end{equation*}
    
        \item For the $t$-dimensional subsimplex $F$, the degrees of freedom on $F$ for $u$ are defined for each $\bm{\sigma} \in \Sigma_{0}^{\bar{\bm{r}}}(\mathbb{I}_{F}, \bar{k})$ as
        \begin{equation*}
            \frac{1}{|F|} \int_{F} u \cdot \llbracket\bm{\lambda}\rrbracket_{F}^{\bm{\sigma}},
        \end{equation*}
        where $\bar{k} = k - n$ and the continuity vector $\bar{\bm{r}} \coloneqq (r_{d - t + 1} - n, \ldots, r_{d} - n) = \bm{q}_{t, n}$. Since $\bm{\theta} \in \Sigma(\mathbb{I}_{d, \setminus t}, n)$ with $0 \le n \le r_{d - t}$, and $\bm{\sigma} \in \Sigma_{0}^{\bar{\bm{r}}}(\mathbb{I}_{F}, \bar{k}) = \Sigma_{0}^{\bm{q}_{t, n}}(\mathbb{I}_{F}, k - n)$, it follows from $\eta(u_{+}) = \eta(u_{-})$ for degree of freedom $\eta$ on $F$ that
        \begin{equation*}
            \frac{1}{|F|} \int_{F} (\bm{D}_{F}^{\bm{\theta}} u_{+}) \cdot \llbracket\bm{\lambda}\rrbracket_{F}^{\bm{\sigma}} = \frac{1}{|F|} \int_{F} (\bm{D}_{F}^{\bm{\theta}} u_{-}) \cdot \llbracket\bm{\lambda}\rrbracket_{F}^{\bm{\sigma}}.
        \end{equation*}
    \end{enumerate}

    Furthermore, both $\bm{D}_{F}^{\bm{\theta}} u_{+}|_{F}$ and $\bm{D}_{F}^{\bm{\theta}} u_{-}|_{F}$ belong to $\mathcal{P}_{k - n}(F)$, while the set of degrees of freedom above is unisolvent for the shape function space $\mathcal{P}_{k - n}(F)$, which implies that $\bm{D}_{F}^{\bm{\theta}} u_{+}|_{F} = \bm{D}_{F}^{\bm{\theta}} u_{-}|_{F}$.
    This completes the proof.
\end{proof}

\begin{proposition}
    \label{prop:global-superspline}
        Suppose that $\bm{r}$, $\rho$, $k$, and $b$ jointly satisfy Condition~\ref{cond:Cr-macro-element}. 
        Given a triangulation $\mathcal{T}(\Omega)$ of a $d$-dimensional domain $\Omega$, let $\mathcal{E}_{k}^{\bm{r}, \rho}(\mathcal{T}(\Omega))$ be the resulting finite element space corresponding to the set of global degrees of freedom defined in \eqref{eq:set-dof-1-global}--\eqref{eq:set-dof-3-global}, i.e.,
        \begin{equation*}
            \begin{aligned}
                \mathcal{E}_{k}^{\bm{r}, \rho}(\mathcal{T}(\Omega)) \coloneqq \big\{u \in L^{2}(\Omega): \; & u|_{K} \in \mathcal{S}_{k}^{\bm{r}, \rho}(\mathcal{T}_{A}(K)) \text{ for each } d \text{-simplex } K \text{ of } \mathcal{T}(\Omega); \\
                & \eta(u) \text{ is single-valued on } F, \\
                & \text{for each } t \text{-dimensional } (0 \le t \le d - 1) \text{ subsimplex } F \text{ of } \mathcal{T}(\Omega), \\
                & \text{and each global degree of freedom } \eta \text{ defined on } F\big\}.
            \end{aligned}
        \end{equation*}
    The set of global degrees of freedom is a dual basis of $\mathcal{E}_{k}^{\bm{r}, \rho}(\mathcal{T}(\Omega))$. Moreover, it holds that
    \begin{equation*}
        \mathcal{E}_{k}^{\bm{r}, \rho}(\mathcal{T}(\Omega)) = \mathcal{S}_{k}^{\bm{r}, \rho}(\mathcal{T}_{A}(\Omega)),
    \end{equation*}
    where $\mathcal{S}_{k}^{\bm{r}, \rho}(\mathcal{T}_{A}(\Omega))$ is defined in \eqref{eq:superspline}.
\end{proposition}

\begin{proof}
    From the coincidence of the global and the local degrees of freedom and the unisolvence of the local shape function space $\mathcal{S}_{k}^{\bm{r}, \rho}(\mathcal{T}_{A}(K))$, if $u \in \mathcal{E}_{k}^{\bm{r}, \rho}(\mathcal{T}(\Omega))$ such that $\eta(u) = 0$ for each global degree of freedom $\eta$, it directly follows that $u = 0$, since $u|_{K} = 0$ for each $d$-dimensional simplex $K$ of $\mathcal{T}(\Omega)$. Moreover, for each global degree of freedom $\eta$, there exists $u \in \mathcal{E}_{k}^{\bm{r}, \rho}(\mathcal{T}(\Omega))$ such that $\eta(u) = 1$ and $\eta'(u) = 0$ for each global degree of freedom $\eta' \ne \eta$ due to the  unisolvence of the local shape function space. More precisely, for each $K \in \mathcal{T}(\Omega)$, choose $u|_{K} \in \mathcal{S}_{k}^{\bm{r}, \rho}(\mathcal{T}_{A}(K))$, so that the local degree of freedom corresponding to $\eta$, if present, has value $1$ and all other local degrees of freedom have value $0$. Therefore, the set of global degrees of freedom is a dual basis of $\mathcal{E}_{k}^{\bm{r}, \rho}(\mathcal{T}(\Omega))$, which proves the first part of the proposition. Note that Proposition~\ref{prop:continuity} reveals the continuity of $\mathcal{E}_{k}^{\bm{r}, \rho}(\mathcal{T}(\Omega))$ on each subsimplex $F$, which leads to $\mathcal{E}_{k}^{\bm{r}, \rho}(\mathcal{T}(\Omega)) \subseteq \mathcal{S}_{k}^{\bm{r}, \rho}(\mathcal{T}_{A}(\Omega))$.

    For each $t$-dimensional ($0 \le t \le d - 1$) subsimplex $F$ of $\mathcal{T}(\Omega)$ and each degree of freedom $\eta$ defined on $F$, the order $n$ of the derivative used in $\eta$ satisfies $0 \le n \le r_{d - t}$. Then, for each $u \in \mathcal{S}_{k}^{\bm{r}, \rho}(\mathcal{T}_{A}(\Omega))$, $\eta(u)$ is single-valued on $F$ by \eqref{eq:superspline}, which implies that $u \in \mathcal{E}_{k}^{\bm{r}, \rho}(\mathcal{T}(\Omega))$. Therefore, it holds that $\mathcal{S}_{k}^{\bm{r}, \rho}(\mathcal{T}_{A}(\Omega)) = \mathcal{E}_{k}^{\bm{r}, \rho}(\mathcal{T}(\Omega))$.
\end{proof}

From the proposition it immediately follows that the finite element space $\mathcal{E}_{k}^{\bm{r}, \rho}(\mathcal{T}(\Omega))$ is a subspace of $C^{r_{1}}(\Omega)$.

\begin{remark}
    In view of Proposition~\ref{prop:global-superspline}, it is clear that the global finite element space is independent of the choices of the vector bases used to define $\bm{D}_{V}^{\bm{\theta}}$ and $\bm{D}_{F}^{\bm{\theta}}$.
\end{remark}

\section{Dimension of Superspline Spaces}
\label{sec:dimension}

In this section, we study the dimension of the superspline space $\mathcal{S}_{k}^{\bm{r}, \rho}(\mathcal{T}_{A}(\Omega))$, and as a special case also the shape function space $\mathcal{S}_{k}^{\bm{r}, \rho}(\mathcal{T}_{A}(K))$. We provide several expressions for it. 

Recall the definition of $B_{t, n}$ in \eqref{eq:s-t-n}, namely, $B_{0, n} \coloneqq 1$ and $B_{t, n} \coloneqq \card \big(\Sigma_{0}^{\bm{q}_{t, n}}(\mathbb{I}_{t}, k - n)\big)$. Note that for a $t$-dimensional ($1 \le t \le d - 1$) subsimplex $F$, it holds that $\card \big(\Sigma_{0}^{\bm{q}}(\mathbb{I}_{F}, l)\big) = \card \big(\Sigma_{0}^{\bm{q}}(\mathbb{I}_{t}, l)\big)$. Consequently,
\begin{equation*}
    \card \big(\Sigma_{0}^{\bm{q}_{t, n}}(\mathbb{I}_{F}, k - n)\big) = B_{t, n},
\end{equation*}
where the continuity vector $\bm{q}_{t, n} \coloneqq (r_{d - t + 1} - n, \ldots, r_{d} - n)$. Then, the number of (global) degrees of freedom defined on $F$ can be obtained from the definition, which leads to the following lemma.

\begin{lemma}
    \label{lem:number-dof-F}
    Suppose that $\bm{r}$, $\rho$, $k$, and $b$ jointly satisfy Condition~\ref{cond:Cr-macro-element}. Given a triangulation $\mathcal{T}(\Omega)$ of a $d$-dimensional domain $\Omega$ and a $t$-dimensional ($0 \le t \le d - 1$) subsimplex $F$ of $\mathcal{T}(\Omega)$, the number of (global) degrees of freedom \eqref{eq:set-dof-1}--\eqref{eq:set-dof-2} (or \eqref{eq:set-dof-1-global}--\eqref{eq:set-dof-2-global}) defined on $F$ is
    \begin{equation}
        \label{eq:number-dof-F}
        \sum_{n = 0}^{r_{d - t}} \binom{n + (d - t - 1)}{d - t - 1} B_{t, n}.
    \end{equation}
\end{lemma}

\begin{proof}
    For $t = 0$, i.e., $F$ is a vertex $V$ and $B_{t, n} = 1$, the number of degrees of freedom defined on $F$ is
    \begin{equation*}
        \sum_{n = 0}^{r_{d}} \binom{n + (d - 1)}{d - 1} = \sum_{n = 0}^{r_{d - t}} \binom{n + (d - t - 1)}{d - t - 1} B_{t, n}.
    \end{equation*}
    For $1 \le t \le d - 1$, the number of degrees of freedom defined on $F$ is
    \begin{equation*}
        \sum_{n = 0}^{r_{d - t}} \binom{n + (d - t - 1)}{d - t - 1} \card \big(\Sigma_{0}^{\bm{q}_{t, n}}(\mathbb{I}_{F}, k - n)\big) = \sum_{n = 0}^{r_{d - t}} \binom{n + (d - t - 1)}{d - t - 1} B_{t, n}.
    \end{equation*}
    This completes the proof.
\end{proof}

Moreover, for a $d$-dimensional simplex $K \in \mathcal{T}(\Omega)$, note that $\bm{r}^{\circ}$ and $\rho$ jointly satisfy Condition~\ref{cond:Cr-element}, and $\card \big(\Sigma_{0}^{\bm{r}^{\circ}}(\mathbb{I}_{K}, \rho)\big) = \card \big(\Sigma_{0}^{\bm{r}^{\circ}}(\mathbb{I}_{d}, \rho)\big)$. Then, the number of (global) degrees of freedom defined on $K$ can be derived from \eqref{eq:decompose-beta} and Lemma~\ref{lem:bij-lemma}, leading to the following lemma.

\begin{lemma}
    \label{lem:number-dof-K}
    Suppose that $\bm{r}$, $\rho$, $k$, and $b$ jointly satisfy Condition~\ref{cond:Cr-macro-element}. Given a triangulation $\mathcal{T}(\Omega)$ of a $d$-dimensional domain $\Omega$ and a $d$-dimensional simplex $K \in \mathcal{T}(\Omega)$, the number of (global) degrees of freedom \eqref{eq:set-dof-3} (or \eqref{eq:set-dof-3-global}) defined on $K$ is 
    \begin{equation}
        \label{eq:number-dof-K}
        \card \big(\Sigma_{0}^{\bm{r}^{\circ}}(\mathbb{I}_{d}, \rho)\big) = \binom{\rho + d}{d} - \sum_{t = 0}^{d - 1} \binom{d + 1}{t + 1} \sum_{n = b}^{r_{d - t}} \binom{n - b + (d - t - 1)}{d - t - 1} B_{t, n}.
    \end{equation}
\end{lemma}

\begin{proof}
    From Lemma~\ref{lem:bij-lemma} we know that, for each vertex $V$ of $K$, there exists a bijection $\Sigma_{V, n}^{\bm{r}^{\circ}}(\mathbb{I}_{d}, \rho) \to \Sigma(\mathbb{I}_{d, \setminus 0}, n)$ for $0 \le n \le r^{\circ}_{d} = r_{d} - b$. Then, for $b \le n \le r_{d}$, it holds that
    \begin{equation*}
        \card \big(\Sigma_{V, n - b}^{\bm{r}^{\circ}}(\mathbb{I}_{d}, \rho)\big) = \binom{n - b + (d - 1)}{d - 1} = \binom{n - b + (d - 1)}{d - 1} B_{0, n}.
    \end{equation*}
    Moreover, for each $t$-dimensional ($1 \le t \le d - 1$) subsimplex $F$, there exists a bijection $\Sigma^{\bm{r}^{\circ}}_{F, n}(\mathbb{I}_{d}, \rho) \to \Sigma(\mathbb{I}_{d, \setminus t}, n) \times \Sigma_{0}^{\bm{q}^{\circ}_{t, n}}(\mathbb{I}_{F}, \rho - n)$ for $0 \le n \le r^{\circ}_{d - t} = r_{d - t} - b$. Since $\rho = k - b$ and $\bm{q}^{\circ}_{t, n} = \bm{q}_{t, n + b}$, it follows that
    \begin{equation*}
        \Sigma_{0}^{\bm{q}^{\circ}_{t, n}}(\mathbb{I}_{F}, \rho - n) = \Sigma_{0}^{\bm{q}_{t, n + b}}(\mathbb{I}_{F}, k - (n + b)).
    \end{equation*}
    Then, for $b \le n \le r_{d - t}$, it holds that
    \begin{equation*}
        \card \big(\Sigma_{F, n - b}^{\bm{r}^{\circ}}(\mathbb{I}_{d}, \rho)\big) = \binom{n - b + (d - t - 1)}{d - t - 1} B_{t, n}.
    \end{equation*}
    Therefore, \eqref{eq:number-dof-K} can be derived from \eqref{eq:decompose-beta} by $\card (\Sigma(\mathbb{I}_{d}, \rho)) = \binom{\rho + d}{d}$.
\end{proof}

Actually, for $1 \le t \le d - 1$, the number $B_{t, n} \coloneqq \card \big(\Sigma_{0}^{\bm{q}_{t, n}}(\mathbb{I}_{t}, k - n)\big)$ can be computed recursively through the following lemma, starting from $B_{0, n} \coloneqq 1$.

\begin{lemma}
    \label{lem:number-B-t-n}
    Suppose that $\bm{r}$, $\rho$, $k$, and $b$ jointly satisfy Condition~\ref{cond:Cr-macro-element}. For each $t$ such that $1 \le t \le d - 1$, and each $n$ such that $0 \le n \le r_{d - t}$, the number $B_{t, n} \coloneqq \card \big(\Sigma_{0}^{\bm{q}_{t, n}}(\mathbb{I}_{t}, k - n)\big)$ equals
    \begin{equation*}
        B_{t, n} = \binom{k - n + t}{t} - \sum_{t' = 0}^{t - 1} \binom{t + 1}{t' + 1} \sum_{n' = n}^{r_{d - t'}} \binom{n' - n + (t - t' - 1)}{t - t' - 1} B_{t', n'}.
    \end{equation*}
\end{lemma}

\begin{proof}
    Let the continuity vector $\bar{\bm{r}} \coloneqq (\bar{r}_{1}, \ldots, \bar{r}_{t}) = (r_{d - t + 1} - n, \ldots, r_{d} - n) =: \bm{q}_{t, n}$ and the polynomial degree $\bar{k} = k - n$. Since $1 \le t \le d - 1$, $0 \le n \le r_{d - t}$, and $\bm{r}$, $\rho$, $k$, and $b$ jointly satisfy Condition~\ref{cond:Cr-macro-element}, for $2 \le s \le t$, it holds that $\bar{r}_{s} - 2 \bar{r}_{s - 1} = r_{d - t + s} - 2 r_{d - t + s - 1} + n \ge 0$, and $\bar{k} - 2 \bar{r}_{t} - 1 = k - 2 r_{d} - 1 + n \ge 0$. Therefore, the continuity vector $\bar{\bm{r}}$ and the polynomial degree $\bar{k}$ jointly satisfy Condition~\ref{cond:Cr-element}. By the refined intrinsic decomposition \eqref{eq:decompose-full} of $\Sigma(\mathbb{I}_{t}, \bar{k})$ corresponding to $\bar{\bm{r}}$, it follows from Lemma~\ref{lem:bij-lemma} that
    \begin{equation*}
        \begin{aligned}
            \binom{\bar{k} + t}{t} = \; & \card \big(\Sigma_{0}^{\bar{\bm{r}}}(\mathbb{I}_{t}, \bar{k})\big) + \binom{t + 1}{1} \sum_{n' = 0}^{\bar{r}_{t}} \binom{n' + (t - 1)}{t - 1} \\
             & + \sum_{t' = 1}^{t - 1} \binom{t + 1}{t' + 1} \sum_{n' = 0}^{\bar{r}_{t - t'}} \binom{n' + (t - t' - 1)}{t - t' - 1} \card \big(\Sigma_{0}^{\bar{\bm{q}}_{t', n'}}(\mathbb{I}_{t'}, \bar{k} - n')\big),
        \end{aligned}
    \end{equation*}
    where the continuity vector $\bar{\bm{q}}_{t', n'} \coloneqq (\bar{r}_{t - t' + 1} - n', \ldots, \bar{r}_{t} - n') = (r_{d - t' + 1} - (n' + n), \ldots, r_{d} - (n' + n)) =: \bm{q}_{t', n' + n}$. Moreover, it holds that $B_{0, n' + n} = 1$ and
    \begin{equation*}
        \card \big(\Sigma_{0}^{\bar{\bm{q}}_{t', n'}}(\mathbb{I}_{t'}, \bar{k} - n')\big) = \card \big(\Sigma_{0}^{\bm{q}_{t', n' + n}}(\mathbb{I}_{t'}, k - (n' + n))\big) = B_{t', n' + n},
    \end{equation*}
    and it follows that
    \begin{equation*}
        \begin{aligned}
            B_{t, n} = \; & \binom{\bar{k} + t}{t} - \binom{t + 1}{1} \sum_{n' = 0}^{\bar{r}_{t}} \binom{n' + (t - 1)}{t - 1} \\
            & - \sum_{t' = 1}^{t - 1} \binom{t + 1}{t' + 1} \sum_{n' = 0}^{\bar{r}_{t - t'}} \binom{n' + (t - t' - 1)}{t - t' - 1} \card \big(\Sigma_{0}^{\bar{\bm{q}}_{t', n'}}(\mathbb{I}_{t'}, \bar{k} - n')\big) \\
            = \; & \binom{k - n + t}{t} - \sum_{t' = 0}^{t - 1} \binom{t + 1}{t' + 1} \sum_{n' = 0}^{r_{d - t'} - n} \binom{n' + (t - t' - 1)}{t - t' - 1} B_{t', n' + n} \\
            = \; & \binom{k - n + t}{t} - \sum_{t' = 0}^{t - 1} \binom{t + 1}{t' + 1} \sum_{n' = n}^{r_{d - t'}} \binom{n' - n + (t - t' - 1)}{t - t' - 1} B_{t', n'}.
        \end{aligned}
    \end{equation*}
    This concludes the proof.
\end{proof}

By Lemma~\ref{lem:number-B-t-n}, the term $\binom{\rho + d}{d}$ in \eqref{eq:number-dof-K} can be expressed in terms of $\binom{k + d}{d}$ and $B_{t, n}$. This yields the following corollary, which provides an alternative formulation for $\card \big(\Sigma_{0}^{\bm{r}^{\circ}}(\mathbb{I}_{d}, \rho)\big)$.

\begin{corollary}
    \label{coro:b-boundary}
    Suppose that $\bm{r}$, $\rho$, $k$, and $b$ jointly satisfy Condition~\ref{cond:Cr-macro-element}. It holds that
    \begin{equation*}
        \begin{aligned}
            \binom{k + d}{d} - \binom{\rho + d}{d} = \; & \sum_{t = 0}^{d - 1} \binom{d}{t + 1} \sum_{n = 0}^{r_{d - t}} \binom{n + (d - t - 1)}{d - t - 1} B_{t, n}  \\
            & - \sum_{t = 0}^{d - 1} \binom{d}{t + 1} \sum_{n = b}^{r_{d - t}} \binom{n - b + (d - t - 1)}{d - t - 1} B_{t, n}.
        \end{aligned}
    \end{equation*}
    Moreover, given a triangulation $\mathcal{T}(\Omega)$ of a $d$-dimensional domain $\Omega$ and a $d$-dimensional simplex $K \in \mathcal{T}(\Omega)$, the number of (global) degrees of freedom \eqref{eq:set-dof-3} (or \eqref{eq:set-dof-3-global}) defined on $K$ can also be expressed as
    \begin{equation}
        \label{eq:number-dof-K-another}
        \begin{aligned}
             \card \big(\Sigma_{0}^{\bm{r}^{\circ}}(\mathbb{I}_{d}, \rho)\big) = \; & \binom{k + d}{d} - \sum_{t = 0}^{d - 1} \binom{d}{t + 1} \sum_{n = 0}^{r_{d - t}} \binom{n + (d - t - 1)}{d - t - 1} B_{t, n} \\
            & - \sum_{t = 0}^{d - 1} \binom{d}{t} \sum_{n = b}^{r_{d - t}} \binom{n - b + (d - t - 1)}{d - t - 1} B_{t, n}.
        \end{aligned}
    \end{equation}
\end{corollary}

\begin{proof}
    It follows from Lemma~\ref{lem:number-B-t-n} that
    \begin{equation*}
        \begin{aligned}
            \binom{k + d}{d} - \binom{\rho + d}{d} = \; & \sum_{n = 0}^{b - 1} \binom{k - n + (d - 1)}{d - 1} \\
            = \; & \sum_{n = 0}^{b - 1} \bigg(B_{d - 1, n} + \sum_{t = 0}^{d - 2} \binom{d}{t + 1} \sum_{n' = n}^{r_{d - t}} \binom{n' - n + (d - t - 2)}{d - t - 2} B_{t, n'}\bigg) \\
            = \; & \sum_{n = 0}^{b - 1} B_{d - 1, n} + \sum_{t = 0}^{d - 2} \binom{d}{t + 1} \sum_{n = 0}^{b - 1} \sum_{n' = n}^{r_{d - t}} \binom{n' - n + (d - t - 2)}{d - t - 2} B_{t, n'}.
        \end{aligned}
    \end{equation*}
    Note that for $b' = 0$ and $b' = b$, we have
    \begin{equation*}
        \begin{aligned}
            \sum_{n = b'}^{r_{d - t}} \sum_{n' = n}^{r_{d - t}} \binom{n' - n + (d - t - 2)}{d - t - 2} B_{t, n'} = \; & \sum_{n = b'}^{r_{d - t}} \sum_{n' = b'}^{n} \binom{n - n' + (d - t - 2)}{d - t - 2} B_{t, n} \\
            = \; & \sum_{n = b'}^{r_{d - t}} \binom{n - b' + (d - t - 1)}{d - t - 1} B_{t, n}. \\
        \end{aligned}
    \end{equation*}
    Therefore, it holds that 
    \begin{equation*}
        \begin{aligned}
            \binom{k + d}{d} - \binom{\rho + d}{d} = \; & \sum_{n = 0}^{r_{1}} B_{d - 1, n} + \sum_{t = 0}^{d - 2} \binom{d}{t + 1} \sum_{n = 0}^{r_{d - t}} \binom{n + (d - t - 1)}{d - t - 1} B_{t, n} \\
            & - \sum_{n = b}^{r_{1}} B_{d - 1, n} - \sum_{t = 0}^{d - 2} \binom{d}{t + 1} \sum_{n = b}^{r_{d - t}} \binom{n - b + (d - t - 1)}{d - t - 1} B_{t, n} \\
            = \; & \sum_{t = 0}^{d - 1} \binom{d}{t + 1} \sum_{n = 0}^{r_{d - t}} \binom{n + (d - t - 1)}{d - t - 1} B_{t, n}  \\
            & - \sum_{t = 0}^{d - 1} \binom{d}{t + 1} \sum_{n = b}^{r_{d - t}} \binom{n - b + (d - t - 1)}{d - t - 1} B_{t, n}.
        \end{aligned}
    \end{equation*}
    The expression in \eqref{eq:number-dof-K-another} follows immediately from \eqref{eq:number-dof-K} and the identity
    \begin{equation*}
        \binom{d}{t} = \binom{d + 1}{t + 1} - \binom{d}{t + 1}.
    \end{equation*}
    This concludes the proof.
\end{proof}

Combining Lemma~\ref{lem:number-dof-F}, Lemma~\ref{lem:number-dof-K}, and Corollary~\ref{coro:b-boundary}, we can prove Theorem~\ref{thm:dim-superspline}.

\begin{proof}[Proof of Theorem~\ref{thm:dim-superspline}]
    Since $N_{t}$ is the number of $t$-dimensional $(0 \le t \le d)$ simplices present in $\mathcal{T}(\Omega)$, and the numbers of global degrees of freedom are given by \eqref{eq:number-dof-F}--\eqref{eq:number-dof-K-another}, it follows that
    \begin{equation*}
        \begin{aligned}
            \dim \mathcal{S}_{k}^{\bm{r}, \rho}(\mathcal{T}_{A}(\Omega)) = \; & \binom{\rho + d}{d} N_{d} + \sum_{t = 0}^{d - 1} N_{t} \sum_{n = 0}^{r_{d - t}} \binom{n + (d - t - 1)}{d - t - 1} B_{t, n} \\
            & - \sum_{t = 0}^{d - 1} \binom{d + 1}{t + 1} N_{d} \sum_{n = b}^{r_{d - t}} \binom{n - b + (d - t - 1)}{d - t - 1} B_{t, n},
        \end{aligned}
    \end{equation*}
    and
    \begin{equation*}
        \begin{aligned}
            \dim \mathcal{S}_{k}^{\bm{r}, \rho}(\mathcal{T}_{A}(\Omega)) = \; & \binom{k + d}{d} N_{d} + \sum_{t = 0}^{d - 1} \bigg(N_{t} - \binom{d}{t + 1} N_{d}\bigg) \sum_{n = 0}^{r_{d - t}} \binom{n + (d - t - 1)}{d - t - 1} B_{t, n} \\
            & - \sum_{t = 0}^{d - 1} \binom{d}{t} N_{d} \sum_{n = b}^{r_{d - t}} \binom{n - b + (d - t - 1)}{d - t - 1} B_{t, n}.
        \end{aligned}
    \end{equation*}
    This completes the proof.
\end{proof}

For the special case $\Omega = K$ (i.e., a single $d$-dimensional simplex), the number of $t$-dimensional ($0 \le t \le d$) simplices present in $K$ is $N_{t} = \binom{d + 1}{t + 1}$, and the dimension of the shape function space $\mathcal{S}_{k}^{\bm{r}, \rho}(\mathcal{T}_{A}(K))$ is given by
\begin{equation}
    \label{eq:dim-shape}
    \begin{aligned}
        \dim \mathcal{S}_{k}^{\bm{r}, \rho}(\mathcal{T}_{A}(K)) = \; & \binom{\rho + d}{d} + \sum_{t = 0}^{d - 1} \binom{d + 1}{t + 1} \sum_{n = 0}^{r_{d - t}} \binom{n + (d - t - 1)}{d - t - 1} B_{t, n} \\
        & - \sum_{t = 0}^{d - 1} \binom{d + 1}{t + 1} \sum_{n = b}^{r_{d - t}} \binom{n - b + (d - t - 1)}{d - t - 1} B_{t, n} \\
        = \; & \binom{k + d}{d} + \sum_{t = 0}^{d - 1} \binom{d}{t} \sum_{n = 0}^{r_{d - t}} \binom{n + (d - t - 1)}{d - t - 1} B_{t, n} \\
        & - \sum_{t = 0}^{d - 1} \binom{d}{t} \sum_{n = b}^{r_{d - t}} \binom{n - b + (d - t - 1)}{d - t - 1} B_{t, n}.
    \end{aligned}
\end{equation}

Recall that $\mathcal{S}_{k}^{\bm{r}, \rho}(\mathcal{T}_{A}(K))$ is a $C^{r_1}$ spline space of degree $k$ over the simplex $K$, with supersmoothness at the lower-dimensional subsimplices $F$ of $K$ and at the split point $V_{A}$.
Now, let $\mathcal{S}_{k}^{r_1}(\mathcal{T}_{A}(K))$ be the $C^{r_1}$ spline space of degree $k$ over the simplex $K$, without imposing supersmoothness. It is clear that $\mathcal{S}_{k}^{\bm{r}, \rho}(\mathcal{T}_{A}(K)) \subseteq \mathcal{S}_{k}^{r_1}(\mathcal{T}_{A}(K))$.
From \cite{schenck2014} we know that its dimension is given by
\begin{equation*}
\dim \mathcal{S}_{k}^{r_1}(\mathcal{T}_{A}(K)) = \binom{k+d}{d}
+ \begin{dcases}
 d\binom{k+d-m(d+1)}{d}, & r_1 = 2m-1,\\
 \sum_{t=0}^{d-1}\binom{k+t-m(d+1)}d, &
r_1 = 2m.
\end{dcases}
\end{equation*}
According to \cite{floater2020}, this space has intrinsic supersmoothness at the split point $V_{A}$ of order
\begin{equation*}
    \rho^* \coloneqq r_1 + (d-1)\bigg\lfloor\frac{r_{1} + 1}{2}\bigg\rfloor.
\end{equation*}
Condition~\ref{cond:Cr-macro-element} ensures that the split point continuity $\rho$ of $\mathcal{S}_{k}^{\bm{r}, \rho}(\mathcal{T}_{A}(K))$ is never smaller than this intrinsic value. Indeed,
\begin{equation*}
    \rho = k - b
    \ge r_{1} + (2^{d - 1} - 1) r_{2}
    \ge r_{1} + (d - 1) \bigg(\bigg\lceil\frac{r_{1} + 1}{2}\bigg\rceil + r_{1} - 1\bigg)
    \ge \rho^*.
\end{equation*}

\section{Conclusion}
\label{sec:conclusion}

In this paper, we presented a unified $C^r$ conforming finite element construction on the Alfeld split for arbitrary smoothness $r \ge 1$ in any spatial dimension $d \ge 2$. Specifically, for any continuity vector $\bm{r} \coloneqq (r_{1}, r_{2}, \ldots, r_{d}) \in \mathbb{N}_{0}^{d}$ and polynomial degree $k \in \mathbb{N}_{0}$ satisfying Condition~\ref{cond:Cr-macro-element}, we provided a set of degrees of freedom such that the resulting finite element space has $C^{r_s}$ continuity on $(d-s)$-subsimplices; see Definition~\ref{def:local-dofs} and Theorem~\ref{thm:Cr-macro-element}. The $C^r$ macro-element results in \cite{lai2013,lai2001} can be seen as special instances of our general construction for $d = 2, 3$ (with minimal values for $k$).
Furthermore, the conditions on $\bm{r}$ and $k$ are relaxed compared to the recent finite element construction proposed in \cite{hu2024construction} for general $d \geq 1$ and, as a consequence, it allows for smaller minimum degrees.

A key ingredient in our analysis is the pair of refined intrinsic decompositions for the Alfeld split, which enabled us to rewrite the degrees of freedom in an equivalent form. It should be emphasized that the degrees of freedom are potentially linearly dependent as linear functionals on $\mathcal{P}_{k}(K)$ for a given simplex $K$, and this constitutes the main difficulty in proving the unisolvence of the full set of degrees of freedom. To overcome this difficulty, we separated them into two groups, each of which corresponds to a refined intrinsic decomposition. This led to the division of the space $\mathcal{S}_{k}^{\bm{r}, \rho}(\mathcal{T}_{A}(K))$ into two subspaces: a boundary space $\mathcal{P}^{\partial}$ and a complementary interior space $\mathcal{Q}^{\circ}$. The latter space was analyzed with the help of a specific lifting operator $Q_{k}^{\rho}$, which maps a polynomial of degree $\rho$ to a piecewise polynomial of degree $k$.
As a side result, we also built an explicit basis for the local shape function space; see Theorem~\ref{thm:unsolvence}.

Finally, we discussed the dimension of the corresponding superspline space in detail and provided explicit expressions for it; see Theorem~\ref{thm:dim-superspline}. Finding the dimension of smooth spline spaces over general triangulations is a challenging problem \cite{lai2007} and the topic is barely touched for $d > 3$. The use of refined intrinsic decompositions could be a useful tool for their study.

\section*{Acknowledgements}

H.~Speleers is a member of the research group GNCS (Gruppo Nazionale per il Calcolo Scientifico) of INdAM (Istituto Nazionale di Alta Matematica).
This work was partially supported by the MUR Excellence Department Project MatMod@TOV (CUP E83C23000330006) awarded to the Department of Mathematics of the University of Rome Tor Vergata
and by the Italian Research Center on High Performance Computing, Big Data and Quantum Computing (CUP E83C22003230001).
Q.~Wu was supported by NSFC project No.~125B2020.



\bibliographystyle{plain}
\bibliography{ref}

@article{alfeld1984a,
  title={A trivariate {C}lough--{T}ocher scheme for tetrahedral data},
  author={Alfeld, P.},
  journal={Computer Aided Geometric Design},
  volume={1},
  number={2},
  pages={169--181},
  year={1984},
}

@article{alfeld1984b,
  title={A bivariate {$C^2$} {C}lough--{T}ocher scheme},
  author={Alfeld, P.},
  journal={Computer Aided Geometric Design},
  volume={1},
  number={3},
  pages={257--267},
  year={1984},
}

@article{alfeld2002a,
  title={Smooth macro-elements based on {C}lough--{T}ocher triangle splits},
  author={Alfeld, P. and Schumaker, L. L.},
  journal={Numerische Mathematik},
  volume={90},
  number={4},
  pages={597--616},
  year={2002},
}

@article{alfeld2002b,
  title={Smooth macro-elements based on {P}owell--{S}abin triangle splits},
  author={Alfeld, P. and Schumaker, L. L.},
  journal={Advances in Computational Mathematics},
  volume={16},
  pages={29--46},
  year={2002},
}

@article{alfeld2005,
  title={A {$C^2$} trivariate macro-element based on the {C}lough--{T}ocher-split of a tetrahedron},
  author={Alfeld, P. and Schumaker, L. L.},
  journal={Computer Aided Geometric Design},
  volume={22},
  number={7},
  pages={710--721},
  year={2005},
}

@article{alfeld2009,
  title={Two tetrahedral {$C^1$} cubic macro elements},
  author={Alfeld, P. and Sorokina, T.},
  journal={Journal of Approximation Theory},
  volume={157},
  number={1},
  pages={53--69},
  year={2009},
}

@article{argyris1968,
  title={The {TUBA} family of plate elements for the matrix displacement method},
  author={Argyris, J. H. and Fried, I. and Scharpf, D. W.},
  journal={The Aeronautical Journal},
  volume={72},
  number={692},
  pages={701--709},
  year={1968},
}

@incollection{awanou2002,
  title={{$C^1$} quintic spline interpolation over tetrahedral partitions},
  author={Awanou, G. and Lai, M. J.},
  booktitle={Approximation Theory X: Wavelets, Splines, and Applications},
  pages={1--16},
  year={2002},
  publisher={Vanderbilt University Press},
}

@book{de1986b,
  title={B-form basics},
  author={Boor, C. de},
  year={1986},
  publisher={Mathematics Research Center, University of Wisconsin-Madison}
}

@article{ciarlet1974,
  title={Sur l'\'el\'ement de {C}lough et {T}ocher},
  author={Ciarlet, P. G.},
  journal={RAIRO Analyse Num\'erique},
  volume={8},
  number={2},
  pages={19--27},
  year={1974},
}

@incollection{clough1965,
  title={Finite element stiffness matrices for analysis of plate bending},
  author={Clough, R. W. and Tocher, J. L.},
  booktitle={Proceedings of Conference on Matrix Methods in Structural Mechanics},
  pages={515--545},
  year={1965},
  publisher={Wright-Patterson Air Force Base, Ohio},
}

@article{douglas1979,
  title={A family of {$C^1$} finite elements with optimal approximation properties for various {G}alerkin methods for 2nd and 4th order problems},
  author={Douglas, J. and Dupont, T. and Percell, P. and Scott, R.},
  journal={RAIRO Analyse Num\'erique},
  volume={13},
  number={3},
  pages={227--255},
  year={1979},
}

@article{floater2020,
  title={A characterization of supersmoothness of multivariate splines},
  author={Floater, M. S. and Hu, K.},
  journal={Advances in Computational Mathematics},
  volume={46},
  pages={70},
  year={2020},
}

@article{groselj2016,
  title={A normalized representation of super splines of arbitrary degree on {P}owell--{S}abin triangulations},
  author={Gro\v{s}elj, J.},
  journal={BIT Numerical Mathematics},
  volume={56},
  pages={1257--1280},
  year={2016},
}

@article{groselj2022,
  title={Generalized {$C^1$} {C}lough--{T}ocher splines for {CAGD} and {FEM}},
  author={Gro\v{s}elj, J. and Knez, M.},
  journal={Computer Methods in Applied Mechanics and Engineering},
  volume={395},
  pages={114983},
  year={2022},
}

@article{hu2024construction,
  title={A construction of {$C^r$} conforming finite element spaces in any dimension},
  author={Hu, J. and Lin, T. and Wu, Q.},
  journal={Foundations of Computational Mathematics},
  volume={24},
  number={6},
  pages={1941--1977},
  year={2024},
}

@article{hu2026sharpness,
  title={The sharpness condition for constructing a finite element from a superspline},
  author={Hu, J. and Lin, T. and Wu, Q. and Yuan, B.},
  journal={Mathematics of Computation},
  volume={95},
  number={361},
  pages={2185--2205},
  year={2026},
}

@incollection{kolesnikov2014,
  title={Multivariate {$C^1$}-continuous splines on the {A}lfeld split of a simplex},
  author={Kolesnikov, A. and Sorokina, T.},
  booktitle={Approximation theory {XIV}},
  series={Springer Proceedings in Mathematics \& Statistics},
  volume={83},
  pages={283--294},
  year={2014},
  publisher={Springer, Cham},
}

@article{lai2013,
  title={A {$C^r$} trivariate macro-element based on the {A}lfeld split of tetrahedra},
  author={Lai, M.-J. and Matt, M. A.},
  journal={Journal of Approximation Theory},
  volume={175},
  pages={114--131},
  year={2013},
}

@article{lai2001,
  title={Macro-elements and stable local bases for splines on {C}lough--{T}ocher triangulations},
  author={Lai, M.-J. and Schumaker, L. L.},
  journal={Numerische Mathematik},
  volume={88},
  number={1},
  pages={105--119},
  year={2001},
}

@article{lai2003,
  title={Macro-elements and stable local bases for splines on {P}owell--{S}abin triangulations},
  author={Lai, M.-J. and Schumaker, L. L.},
  journal={Mathematics of Computation},
  volume={72},
  number={241},
  pages={335--354},
  year={2003},
}

@book{lai2007,
  title={Spline Functions on Triangulations},
  author={Lai, M.-J. and Schumaker, L. L.},
  series={Encyclopedia of Mathematics and its Applications},
  year={2007}, 
  publisher={Cambridge University Press},
}

@article{lyche2022,
  title={Construction of {$C^2$} cubic splines on arbitrary triangulations},
  author={Lyche, T. and Manni, C. and Speleers, H.},
  journal={Foundations of Computational Mathematics},
  volume={22},
  pages={1309--1350},
  year={2022},
}

@article{lyche2024,
  title={A local simplex spline basis for {$C^3$} quartic splines on arbitrary triangulations},
  author={Lyche, T. and Manni, C. and Speleers, H.},
  journal={Applied Mathematics and Computation},
  volume={462},
  pages={128330},
  year={2024},
}

@article{lyche2025,
  title={A {$C^1$} simplex-spline basis for the {A}lfeld split in {$\mathbb{R}^s$}},
  author={Lyche, T. and Merrien, J.-L. and Speleers, H.},
  journal={Computer Aided Geometric Design},
  volume={117},
  pages={102412},
  year={2025},
}

@article{powell1977,
  title={Piecewise quadratic approximations on triangles},
  author={Powell, M. J. D. and Sabin, M. A.},
  journal={ACM Transactions on Mathematical Software},
  volume={3},
  pages={316--325},
  year={1977},
}

@article{schenck2014,
  title={Splines on the {A}lfeld split of a simplex and type {A} root systems},
  author={Schenck, H.},
  journal={Journal of Approximation Theory},
  volume={182},
  pages={1--6},
  year={2014},
}

@article{schumaker2006,
  title={Smooth macro-elements on {P}owell--{S}abin-12 splits},
  author={Schumaker, L. L. and Sorokina, T.},
  journal={Mathematics of Computation},
  volume={75},
  number={254},
  pages={711--726},
  year={2006},
}

@article{schumaker2009,
  title={A {$C^1$} quadratic trivariate macro-element space defined over arbitrary tetrahedral partitions},
  author={Schumaker, L. L. and Sorokina, T. and Worsey, A. J.},
  journal={Journal of Approximation Theory},
  volume={158},
  number={1},
  pages={126--142},
  year={2009},
}

@article{sorokina2008,
  title={A multivariate {P}owell--{S}abin interpolant},
  author={Sorokina, T. and Worsey, A. J.},
  journal={Advances in Computational Mathematics},
  volume={29},
  pages={71--89},
  year={2008},
}

@article{speleers2013a,
  title={Construction of normalized {B}-splines for a family of smooth spline spaces over {P}owell--{S}abin triangulations},
  author={Speleers, H.},
  journal={Constructive Approximation},
  volume={37},
  number={1},
  pages={41--72},
  year={2013},
}

@article{speleers2013b,
  title={Multivariate normalized {P}owell--{S}abin {B}-splines and quasi-interpolants},
  author={Speleers, H.},
  journal={Computer Aided Geometric Design},
  volume={30},
  number={1},
  pages={2--19},
  year={2013},
}

@article{wang1992,
  title={A {$C^2$}-quintic spline interpolation scheme on triangulation},
  author={Wang, T.},
  journal={Computer Aided Geometric Design},
  volume={9},
  number={5},
  pages={379--386},
  year={1992},
}

@incollection{wang1990,
  title={{$S_{\mu+1}^{\mu}$} surface interpolations over triangulations},
  author={Wang, R.-H. and Shi, X.-Q.},
  booktitle={Approximation, Optimization and Computing: Theory and Applications},
  pages={205--208},
  year={1990},
  publisher={Elsevier Science Publishers B.V.},
}

@article{worsey1987,
  title={An $n$-dimensional {C}lough--{T}ocher interpolant},
  author={Worsey, A. J. and Farin, G.},
  journal={Constructive Approximation},
  volume={3},
  number={1},
  pages={99--110},
  year={1987},
}

@article{worsey1988,
  title={A trivariate {P}owell--{S}abin interpolant},
  author={Worsey, A. J. and Piper, B.},
  journal={Computer Aided Geometric Design},
  volume={5},
  number={3},
  pages={177--186},
  year={1988},
}

@article{zenisek1973,
  title={Polynomial approximation on tetrahedrons in the finite element method},
  author={{\v{Z}en\'i\v{s}ek}, A.},
  journal={Journal of Approximation Theory},
  volume={7},
  number={4},
  pages={334--351},
  year={1973},
}

@article{zenisek1974,
  title={A general theorem on triangular finite {$C^{(m)}$}-elements},
  author={{\v{Z}en\'i\v{s}ek}, A.},
  journal={RAIRO Analyse Num\'erique},
  volume={8},
  number={2},
  pages={119--127},
  year={1974},
}

\end{document}